\documentclass[reqno]{amsart}

\usepackage[bbgreekl]{mathbbol}
\usepackage{amsthm,amsmath,amsfonts,amssymb,graphicx,latexsym,bm,nicefrac,dsfont,esint,bbm,mathrsfs} 
\usepackage{enumitem} 
\usepackage[scr=stixfancy]{mathalpha}
\usepackage{comment}

\DeclareSymbolFontAlphabet{\mathbbm}{bbold}
\DeclareSymbolFontAlphabet{\mathbb}{AMSb} 

\usepackage{scalerel}
\newlength\bshft
\def\fakebold#1{\ThisStyle{\ooalign{$\SavedStyle#1$\cr%
  \kern-\bshft$\SavedStyle#1$\cr%
  \kern\bshft$\SavedStyle#1$}}}

\usepackage{multirow,float,xcolor}
\usepackage{url}
\usepackage{hyperref}
\usepackage[hyphenbreaks]{breakurl}
\hypersetup{
  colorlinks   = true, 
  urlcolor     = blue, 
  linkcolor    = blue, 
  citecolor    = blue, 
  breaklinks   = true
}

\usepackage{csquotes}
\usepackage[style=alphabetic,sorting=debug,maxbibnames=99]{biblatex} 

\newtheorem{definition}{Definition}[section]
\newtheorem{theorem}[definition]{Theorem}
\newtheorem*{theorem*}{Theorem}
\newtheorem{lemma}[definition]{Lemma}
\newtheorem{proposition}[definition]{Proposition}
\newtheorem{corollary}[definition]{Corollary}
\newtheorem{remark}[definition]{Remark}

\newtheorem{example}[definition]{Example}

\newtheorem{model}[definition]{Model}

\renewcommand{\theequation}{\arabic{section}.\arabic{equation}}

\numberwithin{equation}{section}
\numberwithin{table}{section}
\numberwithin{figure}{section}

\allowdisplaybreaks

\newcommand{\wen}{\color{black}}

\newcommand{\ma}[1]{\textcolor{magenta}{#1}}

\newcommand{\C}{\mathbb{C}}
\newcommand{\N}{\mathbb{N}}

\newcommand{\R}{\mathbb{R}}
\newcommand{\Z}{\mathbb{Z}}

\newcommand{\B}{\mathcal{B}}

\newcommand{\eps}{\varepsilon}

\newcommand{\E}{\mathbb{E}}
\renewcommand{\Pr}{P}

\renewcommand{\d}{ \mathrm{d}}	

\bmdefine\h{h}
\bmdefine\P{P}     
\bmdefine\r{r}
\bmdefine\u{u} 
\bmdefine\v{v} 
\bmdefine\x{x}   
\bmdefine\y{y}
\bmdefine\z{z}
\bmdefine\w{w} 
\bmdefine\balpha{\alpha} 
\bmdefine\kap{\kappa}
\bmdefine\btheta{\theta} 
\bmdefine\bxi{\xi}  
\bmdefine\bom{\omega}
\newcommand{\nkap}{\|\kap\|}
\newcommand{\Pk}{\P(\kap)}

\newcommand{\Pt} {\P(\btheta)}
\bmdefine\bzeta{\zeta}
\bmdefine\etab{\eta}
\bmdefine\bzeta{\zeta}
\bmdefine\bphi{\varphi}
\bmdefine\Z{\mathcal Z}
\bmdefine\bZ{Z}

\newcommand{\xalt}{\tilde{\x}}
\newcommand{\talt}{\tilde{t}}
\newcommand{\xbb}{\mathbbm{x}}
\newcommand{\xbbalt}{\tilde{\mathbbm{x}}}
\newcommand{\tbb}{\mathbbm{t}}
\newcommand{\tbbalt}{\tilde{\mathbbm{t}}}

\newcommand{\fxbb}{\mathbbm{f}^2}
\newcommand{\fxsbb}{\mathbbm{f}}

\newcommand{\etabb}{\bbeta}
\newcommand{\etad}{\bigl(\tfrac{\d}{\d s} \eta\bigr)}
\newcommand{\U}{{U}}
\newcommand{\Pbb}{\mathbbm{P}}

\newcommand{\wrt}{w.r.t.\ }
\newcommand{\Wlog}{w.l.o.g.\ }
\newcommand{\rf}{\bm{w}}

\newcommand{\bdot}{\,\bm{\cdot}\,}

\newcommand{\trace}{\mathrm{tr\,}}

\newcommand{\ind}{\mathds{1}}
\newcommand{\indicator}[1]{\mathbbm{1}_{\smash{#1}}}

\newcommand{\st}{\sigma_{\mathrm{t}}}
\newcommand{\su}{\sigma_{\mathrm{u}}}
\newcommand{\sx}{\sigma_{\mathrm{x\;\!}}}
\newcommand{\sz}{\sigma_{\mathrm{z}}}

\newcommand{\ft}{f^{\phantom{*}}_{\mathrm{t}}\!}
\newcommand{\fts}{f^{1/2}_{\mathrm{t}}\!}
\newcommand{\fx}{f^{\phantom{*}}_{\mathrm{x}}\!}
\newcommand{\fxs}{f^{1/2}_{\mathrm{x}}\!}

\newcommand{\wnr}{\xi}
\newcommand{\wn}{\bm{\xi}}

\newcommand{\flow}{\bphi}

\makeatletter
\newcommand*\wt[1]{\mathpalette\wthelper{#1}}
\newcommand*\wthelper[2]{%
        \hbox{\dimen@\accentfontxheight#1%
                \accentfontxheight#11.2\dimen@
                $\m@th#1\widetilde{#2}$%
                \accentfontxheight#1\dimen@
        }%
}

\newcommand*\accentfontxheight[1]{%
        \fontdimen5\ifx#1\displaystyle
                \textfont
        \else\ifx#1\textstyle
                \textfont
        \else\ifx#1\scriptstyle
                \scriptfont
        \else
                \scriptscriptfont
        \fi\fi\fi3
}
\makeatother

\makeatletter
\newcommand*{\mint}[1]{%
  \mint@l{#1}{}%
}
\newcommand*{\mint@l}[2]{%
  \@ifnextchar\limits{%
    \mint@l{#1}%
  }{%
    \@ifnextchar\nolimits{%
      \mint@l{#1}%
    }{%
      \@ifnextchar\displaylimits{%
        \mint@l{#1}%
      }{%
        \mint@s{#2}{#1}%
      }%
    }%
  }%
}
\newcommand*{\mint@s}[2]{%
  \@ifnextchar_{%
    \mint@sub{#1}{#2}%
  }{%
    \@ifnextchar^{%
      \mint@sup{#1}{#2}%
    }{%
      \mint@{#1}{#2}{}{}%
    }%
  }%
}
\def\mint@sub#1#2_#3{%
  \@ifnextchar^{%
    \mint@sub@sup{#1}{#2}{#3}%
  }{%
    \mint@{#1}{#2}{#3}{}%
  }%
}
\def\mint@sup#1#2^#3{%
  \@ifnextchar_{%
    \mint@sup@sub{#1}{#2}{#3}%
  }{%
    \mint@{#1}{#2}{}{#3}%
  }%
}
\def\mint@sub@sup#1#2#3^#4{%
  \mint@{#1}{#2}{#3}{#4}%
}
\def\mint@sup@sub#1#2#3_#4{%
  \mint@{#1}{#2}{#4}{#3}%
}
\newcommand*{\mint@}[4]{%
  \mathop{}%
  \mkern-\thinmuskip
  \mathchoice{%
    \mint@@{#1}{#2}{#3}{#4}%
        \displaystyle\textstyle\scriptstyle
  }{%
    \mint@@{#1}{#2}{#3}{#4}%
        \textstyle\scriptstyle\scriptstyle
  }{%
    \mint@@{#1}{#2}{#3}{#4}%
        \scriptstyle\scriptscriptstyle\scriptscriptstyle
  }{%
    \mint@@{#1}{#2}{#3}{#4}%
        \scriptscriptstyle\scriptscriptstyle\scriptscriptstyle
  }%
  \mkern-\thinmuskip
  \int#1%
  \ifx\\#3\\\else_{#3}\fi
  \ifx\\#4\\\else^{#4}\fi  
}
\newcommand*{\mint@@}[7]{%
  \begingroup
    \sbox0{$#5\int\m@th$}%
    \sbox2{$#5\int_{}\m@th$}%
    \dimen2=\wd0 %
    \let\mint@limits=#1\relax
    \ifx\mint@limits\relax
      \sbox4{$#5\int_{\kern1sp}^{\kern1sp}\m@th$}%
      \ifdim\wd4>\wd2 %
        \let\mint@limits=\nolimits
      \else
        \let\mint@limits=\limits
      \fi
    \fi
    \ifx\mint@limits\displaylimits
      \ifx#5\displaystyle
        \let\mint@limits=\limits
      \fi
    \fi
    \ifx\mint@limits\limits
      \sbox0{$#7#3\m@th$}%
      \sbox2{$#7#4\m@th$}%
      \ifdim\wd0>\dimen2 %
        \dimen2=\wd0 %
      \fi
      \ifdim\wd2>\dimen2 %
        \dimen2=\wd2 %
      \fi
    \fi
    \rlap{%
      $#5%
        \vcenter{%
          \hbox to\dimen2{%
            \hss
            $#6{#2}\m@th$%
            \hss
          }%
        }%
      $%
    }%
  \endgroup
}

\begin{document}

\title[]{Random field reconstruction of \\ 
inhomogeneous turbulence \\ 
Part I: Modeling and analysis}
\author[Antoni et al.]{Markus Antoni$^{1}$}
\author[]{Quinten K\"urpick$^{1}$}
\author[]{Felix Lindner$^{1}$}
\author[]{Nicole Marheineke$^{2}$}
\author[]{Raimund Wegener$^{3}$}

\date{\today\\
$^1$ Universit\"at Kassel, Institut f\"ur Mathematik, Heinrich-Plett-Str. 40, 
D-34132 Kassel, Germany\\
$^2$ Universit\"at Trier, Arbeitsgruppe Modellierung und Numerik, Universit\"atsring 15, D-54296 Trier, Germany\\
$3$ Petrusstr. 1, D-54292 Trier, Germany\\
}
\begin{abstract}

We develop and analyze a random field model for 
the reconstruction of turbulent velocity fluctuations from 
inhomogeneous characteristic flow quantities 
provided by RANS simulations that is accessible to both a rigorous analytical validation of the model properties 
and efficient numerical simulation.   
The model is fully continuous and based on an explicit representation formula in terms of stochastic integrals 
combining moving average and Fourier-type representations in time and space, respectively.   
The structure of the model is 
systematically 
derived from spectral representations of homogeneous fields 
by means of suitable stochastic integral transformations 
that ensure the preservation of consistency properties 
when progressing to the case of inhomogeneous flow characteristics.  
Moreover, we employ a two-scale approach that 
separates a macro scale related to the 
variations of the characteristic flow quantities
from a representative turbulence scale on which the fluctuations are modeled, 
allowing to assess the model properties by asymptotic analysis \wrt the scale ratio.  
In particular, a novel inhomogeneous ergodicity result establishes the 
recovery of the  inhomogeneous characteristic flow quantities 
by means of local averages of a single sample path in time and space.   
\end{abstract}
\maketitle
\noindent
{\sc Keywords.} 
Inhomogeneous turbulence, 
random velocity field, 
stochastic integral representation, 
multiscale model, 
asymptotic analysis, 
spatio-temporal ergodicity \\
{\sc AMS-Classification.} 60G60, 60H05, 76F55, 76M35







\section{Introduction}

The numerical simulation of turbulent velocity fields 
is 
ubiquitous 
in a broad range of scientific applications  
and has advanced significantly 
alongside 
the rapid growth of available computing power. 
While the feasibility of computationally involved 
approaches  based on direct numerical simulation,  
large eddy simulation,    
or machine learning 
is steadily increasing, 
the practical applicability of these techniques still has its limits. 
Statistical turbulence modeling and stochastic simulation methods 
involving 
synthetic 
random 
fields  
therefore 
remain relevant in various contexts. 
Examples include 
the generation of interface data 
and boundary conditions 
for hybrid methods \cite{SSST14, XYD22}, 
the superposition of 
given 
time-averaged flow characteristics by 
artificial spatio-temporal fluctuations 
\cite{Pol+20, Zwi+24},  
and the provision of 
local 
high-resolution information 
on dynamically varying subsets of the computational domain 
as required for 
particle or fiber dynamics \cite{HMRW13,Wie+19}.

In statistical turbulence modeling, the turbulent flow velocity $\u$ is considered as a random field in space and time, usually split into expected value  (mean velocity) $\overline{\u}$ and centered fluctuations $\u'$. 
The fluctuations are characterized by the turbulent kinetic energy $k$ 
and possible anisotropic components of the Reynolds stresses, 
the dissipation rate $\eps$ of turbulent kinetic energy, and the kinematic viscosity $\nu$ of the fluid, among others.
%
%
These characteristic fields and the mean velocity 
are typically non-uniform and 
may be provided by statistical turbulence models based on the 
Reynolds-averaged Navier-Stokes equations (RANS) with constitutive laws,
such as $k$-$\omega$ models, $k$-$\eps$ models 
or Reynolds stress models \cite{Pope00, Wil06, SC22}. 
In applications depending on the fine structure of the flow,  
it is of interest to additionally model the actual turbulent velocity fluctuations $\u'$. 
This has 
led to the development of a variety of 
methods for the 
reconstruction of the fluctuations as a 
synthetic random field 
depending on 
given inhomogeneous flow characterstics. 
%
However, existing 
random field models for inhomogeneous turbulence 
are usually formulated in a discrete framework 
with regard to, e.g., spectral space and random inputs,
which sometimes comes along with a limited analytical foundation 
of the desired model properties, as detailed below. 

The aim of this article is to develop and analyze a flexible random field model for 
the reconstruction of turbulent velocity fluctuations from 
RANS-type data 
that is accessible to a rigorous analytical validation of the model properties 
and can be simulated efficiently. 
Accessibility to analytical investigations
facilitates not only conceptual clarity and novel analytical insights, 
but also the establishment of improved features of the model 
and the avoidance of inconsistencies. 
It 
is achieved through two key factors: 
First, the developed model is fully continuous and based on 
an explicit representation formula in terms of stochastic Fourier-type integrals, 
enabling the use of helpful tools from stochastic analysis. 
This 
allows to separate modeling and analysis on the one hand from 
numerical discretization and simulation on the other hand,   
so that the latter can be treated subsequently 
and 
boils down 
to a quadrature problem for stochastic integrals,  
addressed in the accompanying article \cite{AKLMW25}. 
Second, we employ a two-scale approach based on a suitable non-dimensionalization 
that separates a macro scale related to the variations of the characteristic flow quantities 
from a 
representative 
turbulence scale on which the fluctuations are modeled. 
It allows to assess the inhomogeneous model properties by asymptotic analysis \wrt the scale ratio. 
An exemplary 
simulation result illustrating 
our 
model is presented in Figure~\ref{fig:intro}. 
In the remainder of this introductory section we first discuss the relevant literature 
before describing 
the model and our contributions to the field in more detail.

\begin{figure}
\centerline{\includegraphics[scale=.51]{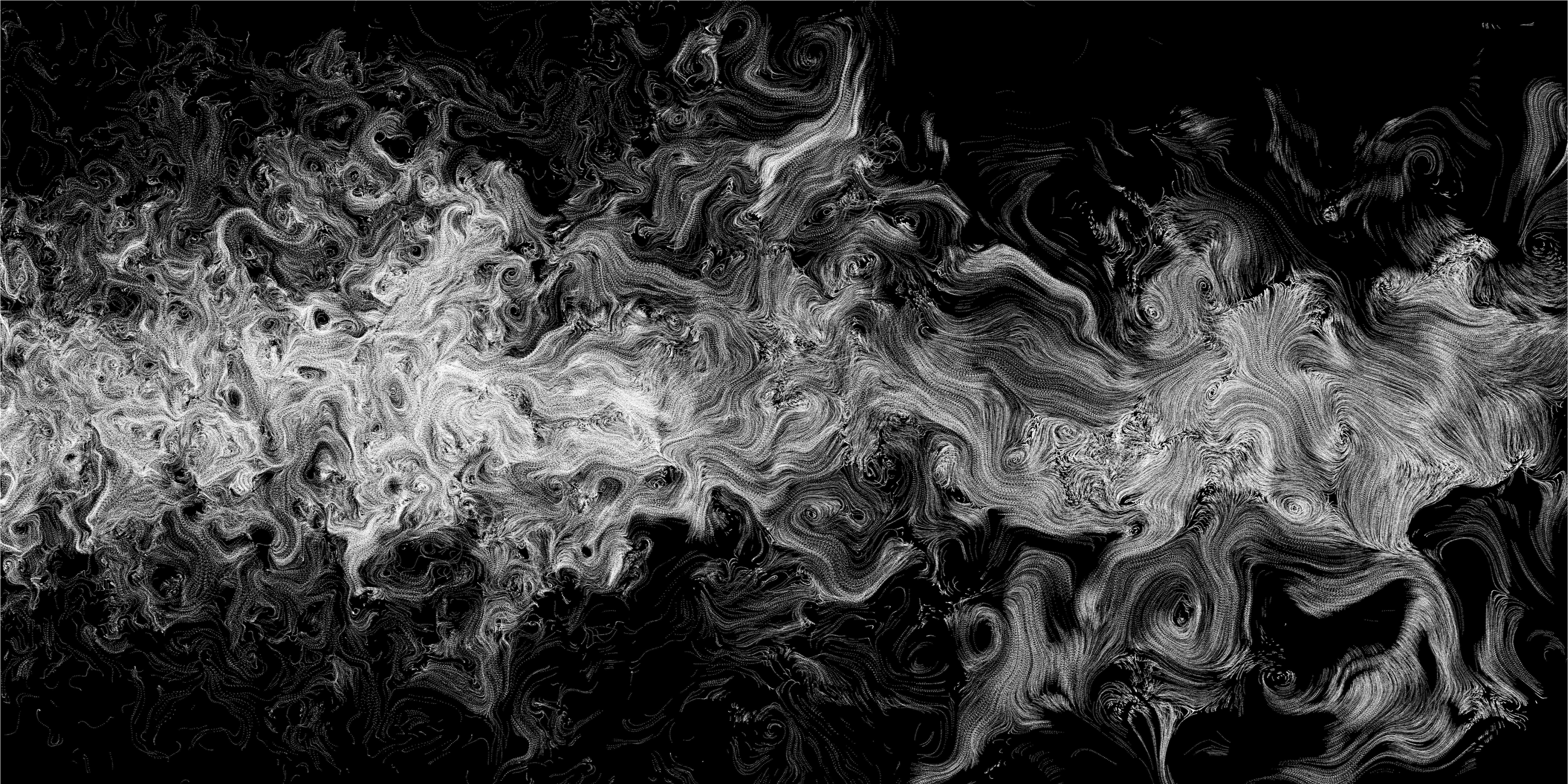}}
\caption{
	Snapshot of 
	a simulated inhomogeneous turbulence field, 
	visualized 
	by means of 
	the intersection points of 
	$30\,000$ 
	streamlines 
	and a thin rectangular domain 
	of proportion $\text{width}:\text{height}:\text{depth}=1:0.5:0.01$. 
	The streamlines are 
	started at 
	random initial positions 
	in a box of proportion $1:0.1:0.5$. 
	The underlying turbulence length scale 
	$k^{3/2}/\eps$ 
	increases from left to right. 
	} 
\label{fig:intro} 
\end{figure}

\subsection{Literature background}

The classical spectral decomposition of homogeneous random fields 
in terms of Fourier integrals \wrt stochastic orthogonal measures 
\cite{GS74,Bre14} 
%
provides a well-established analytical framework and natural representation formulas 
for
the idealized case of homogeneous turbulence  \cite{Bat53,MY75,Yag87}. 
%
Although 
Fourier-type stochastic integral representations for 
stationary stochastic processes and homogeneous random fields 
have been generalized to cover evolutionary spectra and 
certain classes of nonstationary processes and inhomogeneous fields  \cite{Pri65, PT73, Pri88, DS89, LCS07}, 
applications 
of such generalizations 
to the continuous description of inhomogeneous turbulence 
have been of limited extent and 
not related to RANS simulations. 
Examples include 
the use of stochastic integral representations as a starting point for the development of simulation algorithms for 
inhomogeneous scalar-valued
wind fields in one and two spatial dimensions
\cite{Pen+17, SCPSL18, THWZ23} 
or as a modeling tool for three-dimensional atmospheric turbulence with 
an inhomogeneity in 
the vertical direction \cite{Mann94}, compare also \cite{PS95}. 
%
%
By contrast, there exists a large variety of inhomogeneous random field models 
formulated in partially discrete settings, aiming at 
the reconstruction of turbulent velocity fluctuations from 
non-uniform 
characteristic flow quantities 
such as turbulent kinetic energy, dissipation rate, and Reynolds stresses 
that may be provided in the context of RANS computations. 
%
%
%
%
Arguably the largest class of such models is constituted by 
synthetic turbulence fields involving the summation of a finite number of random Fourier modes -- 
an approach initiated 
in 
\cite{Kra70} to investigate single-particle diffusion in 
homogeneous velocity fields. 
Early adaptations to inhomogeneous scenarios 
regarding  
aerosol particle dispersion and aeroacoustics 
can be found in \cite{LA92,LA93,Li+94} and \cite{BBLC94}, respectively. 
%
The 
first-mentioned articles 
restrict inhomogeneity to a non-uniform scaling in velocity only but allow for a specific form of Reynolds stress anisotropy, 
whereas the 
last-mentioned work 
includes non-uniform length and times scales as well 
but 
assumes isotropic one-point correlations.  
The approach has been subsequently extended and modified \wrt various 
aspects such as 
a refined 
treatment of anisotropy 
\cite{SSC01, BGC04, BEDJ04, HLW10, SSST14}  
and incompressibility properties in the presence of inhomogeneity and anisotropy 
\cite{YB14, PR18, ADDK21, ADDK22, GJYZ23}.  
A common conceptual simplification in these works 
is the 
assumption that the random fluctuation field is Gaussian, 
which is 
often not formulated explicitly but implied 
in an approximate sense 
by virtue of the central limit theorem. 
While 
the assumption of Gaussianity 
reflects a pragmatic focus on 
second-order statistics and is adopted in the present article, 
we remark that 
it 
disregards the skewness of small scale velocity increments and related aspects 
such as intermittency and vortex streching 
that have been 
partially 
included in refined non-Gaussian 
models for homogeneous 
fluctuations; 
see, e.g.,  
\cite{Fung+92, 
RM08, 
Laf+14,
PGC16,
VS21, 
FMSWP22}. 
%
Further approaches to 
random field modeling of inhomogeneous turbulence 
besides the random Fourier method,   
such as 
the synthetic 
eddy method 
\cite{Jar+09, Pol+13, KK18, Sch+22} 
or 
moving average techniques (also referred to as 
digital filtering) 
\cite{Ewe07, XC08, CIA12, KCX13},  
typically involve partially discrete settings as well. 
To streamline the  discussion  
we mainly focus on Fourier-based methods,  
as they offer a particularly intrinsic connection to 
classical turbulence theory as well as 
flexibel, grid-independent local evaluations 
and a convenient adaptability to high-resolution scenarios.

Despite the various advancements in random field modeling of inhomogeneous turbulence in the past few decades, 
there remain practically relevant challenges 
and opportunities for improvement, 
even within the simplified Gaussian setting. 
To begin with, 
an essential aspect that tends to be 
neglected 
is the consistency of the modeled fluctuations 
with prescribed flow characteristics involving  
the velocity gradient, such as the dissipation rate. 
The latter is usually 
considered 
only indirectly and abstractly 
as a model parameter, 
but the question whether 
its prescribed local value can be represented 
in terms of the spatial derivatives of  
the constructed inhomogeneous fluctuation field 
is not examined; 
see, e.g., 
\cite{BBLC94, BJ99, BED03, SSST14, ADDK20, PBJ21}. 
%
Addressing this question requires a 
careful 
regularity analysis, 
and its nontriviality  is 
illustrated 
by the fact that a systematic investigation 
of the velocity gradient 
reveals 
a further issue that 
concerns 
a variety of available models, 
namely a lack of invariance \wrt the underlying coordinate frame  
in the sense that the model properties may significantly depend on the choice of the origin. 
This typically occurs 
when an inhomogeneous 
characteristic flow quantity 
influences the scaling 
in the argument of the Fourier mode itself as opposed to the 
Fourier coefficient, compare  
\cite{SSC01, BGC02, BGC04, BED04, BEDJ04, HLW10, HMRW13, YB14, Kar+19, CLWH22, GJYZ23}. 
The resulting effect 
can be tracked analytically, e.g., by calculating the second moment of the velocity derivative,  
which is seen to depend on the difference vector of the considered evaluation point and the origin. 
While there are models that circumvent this problem, its avoidance sometimes 
comes at the cost of other desirable model properties 
such as exact incompressibility under homogeneous but anisotropic flow conditions \cite{SSST14,Bat+15,XYD22,Bos+23,Zwi+24}. 
%
As another aspect related to the dissipation rate, 
we remark that 
the adjustment of the employed model spectrum to the 
prescribed local flow characteristics 
commonly relies on asymptotic arguments for an infinite turbulence Reynolds number   
\cites[App.~A]{BJ99}[p.~3]{BED03} 
and is thus prone to overestimating the width of the inertial subrange. 
A further 
relevant feature consists in 
the advection of the turbulent structures by the mean flow, 
which is often either not addressed 
or restricted to 
a uniform convection velocity 
\cite{BBLC94, SSST14,  ADDK20, PBJ21, XYD22}, 
limiting the model applicability in 
basic scenarios such as free jets.  
Notable exceptions are 
given by 
the approach in \cite{BED03}, 
which  has been 
adapted, e.g., in  
\cite{BED04, CIA12, Bos+23, Zwi+24} 
and utilizes temporal decay  
to include non-uniform mean flow advection 
without distorting the locally prescribed correlation structure, 
by the alternative approach in \cite{Ewe07},  
which has been employed, e.g., in
 \cite{Ewe08, Ewe+11, Ewe16, Ewe+19}
and relies on spatial moving averages \wrt a nonlinearly advected white noise, 
and 
by 
the approach described in \cite{PR18, BPSS20}, 
which is based on 
the superposition of a finite number of homogeneous fields 
by means of spatially localizing weight functions.
%
%
However, the temporal evolution in \cite{BED03} is constructed as 
an autoregressive process in discrete time only, 
without specification of a continuous limit object 
as the discretization is refined. 
It turns out that the time series model converges 
to an Ornstein-Uhlenbeck process as the stepsize decreases, 
so that the limit model inherits the path irregularity of a 
Brownian motion, resulting in  
non-differentiability of the fluctuations 
that may lead to unphysical 
behaviour 
in high-resolution applications.  
The autoregressive construction from \cite{BED03} is included in \cite{Ewe07} as well, 
and the superposition approach in 
\cite{PR18, BPSS20} 
comes with the drawback 
that the weight functions alter 
the two-point correlations 
in such a way that the resulting compound field is not able to recover 
the special case of globally homogeneous turbulence. 
%
Finally, a further essential  question 
that has not been dealt with exhaustively 
in the context of inhomogeneous synthetic fluctuations 
concerns the validity of the ergodic hypothesis, 
lying at the heart of statistical turbulence theory 
and requiring that 
statistical flow properties represented in terms of ensemble averages 
are reflected in 
the respective 
averages of a single sample path in space and time \cite{MY71,Hin75}. 
It is primarily 
addressed 
in terms of numerical experiments and
with regard to temporal averages under stationary flow conditions only,  
but 
less in terms of a thorough mathematical analysis and 
not with regard to local spatial averages involving inhomogeneous directions 
or 
local temporal averages under nonstationary flow conditions 
\cite{SSC01, BGC02, CP13, PBJ21, GJYZ23}.  
%
%
In the overall view, 
the different listed aspects 
tend to be 
related to 
the semi-discrete nature of available models 
and 
to 
a missing separation of questions concerning modeling and analysis 
from questions concerning numerical discretization and simulation. 

\subsection{Contributions and structure of the article} 
\label{subsec:intro_contrib}

The random field model developed in this article addresses the 
challenges above 
within a fully continuous framework that is accessible to 
mathematical analysis 
as well as 
numerical discretization. 
Guided by the idea that the local fine structure of turbulent flows is nearly homogeneous and isotropic 
(Kolmogorov's hypothesis of local isotropy, \cite{Kol41, Kol62}), 
the derivation of the model 
builds on natural Fourier-type stochastic integral representations of homogeneous and isotropic fields, 
which are 
discussed  
in Section~\ref{sec:hom}. 
In order to enable a 
consistent and tractable 
description of 
both 
homogeneous local behaviour 
and 
inhomogeneous large-scale properties,  
we adopt a dimensionless view in combination with a two-scale approach 
as detailed in Section~\ref{sec:scaling}, 
capturing 
the fluctuation field $\u'(\x,t)$ from the perspective of a macro scale 
related to the (outer) geometry of the flow. 
%
%
The macro scale is complemented by a representative turbulence scale 
defined 
in terms of typical values $k_0$, $\eps_0$, and $\nu_0$ 
for turbulent kinetic energy, dissipation rate, and kinematic viscosity,  respectively. 
This turbulence scale serves as a 
basis for the local modeling of the fluctuations 
and plays the role of a micro scale 
in comparison to the macroscopic perspective, characterized by the 
corresponding 
scale ratio $\delta\ll1$. 
Inhomogeneous macroscopic 
variations 
of the 
characteristic flow quantities  
relative to 
the typical values $k_0$, $\eps_0$, $\nu_0$ 
are described 
by the dimensionless 
factors 
$k(\x,t)$, $\eps(\x,t)$, $\nu(\x,t)$.  
In Section~\ref{sec:further_aspects} we combine our scaling approach with 
the previously considered stochastic integral representations of homogeneous fields  
and prepare the transition to the inhomogeneous case by 
applying 
specific stochastic integral transformations in order to derive equivalent but suitably modified representation formulas 
for homogeneous fields 
in Lemmata~\ref{lem:transformation1}--\ref{lem:transformation3}. 
These integral transformations 
constitute a key step in 
the conceptualization of our model 
and provide a general technique to 
ensure the preservation of 
invariance properties 
as well as an adequate 
advection by the mean flow 
when 
progressing to
the case of non-uniform characteristic flow quantities. 
%
The inhomogeneous model resulting from these derivation steps is presented 
in Section~\ref{sec:inhom} and takes the form 
\begin{equation}\label{eq:u'_inhom_intro}
\begin{aligned}
\u'(\x,t)  = \su(\x,t)\:\Re \! & \int_{\R^+ \times S^2 \times \R}  
\Bigl(\frac1{\delta\st(\x,t)}\Bigr)^{1/2}\eta\Bigl(\frac{1}{\delta\st(\x,t)}(t-s)\Bigr) \exp\Bigl\{ i \frac{1}{\delta}\;\! \kappa \;\! \btheta \cdot \flow(s;\x,t) \Bigr\} 
\\ & \vphantom{\int} \hphantom{\int_{\R^+}\times}
\sx^{1/2}(\x,t)\;\! E^{1/2} \bigl(\sx(\x,t)\kappa; \sz(\x,t)z\bigr) \Pt \cdot 
 \bm{L}(\x,t) \cdot \wn(\d\kappa,\d\btheta,\d s),  
\end{aligned}
\end{equation}
combining a moving average representation in time with a Fourier-type representation in space;
compare Model~\ref{model:turb_inhom} and the 
comments thereafter. 
The stochastic integral in \eqref{eq:u'_inhom_intro} 
is $\C^3$-valued and taken \wrt a $\C^3$-valued Gaussian white noise $\bxi$  
on the product space $\R^+ \times S^2 \times \R$ accomodating all combinations of 
wave numbers $\kappa\in\R^+$, 
wave orientation vectors $\btheta\in S^2$, 
and values of the time integration variable $s\in\R$;  
%
the notation  $\Re$ 
stands for 
the real part.   
The time integration kernel 
$\eta(s)$ 
and the 
dimensionless 
energy spectrum function 
$E(\kappa;\zeta)$ 
model the 
temporal decay and the 
wave number distribution 
of the fluctuations 
from the perspective of 
local turbulence scales in time and space, respectively, 
where the second argument of $E$ 
indicates  
the inverse of the 
local 
turbulence Reynolds number, 
with typical value 
$z=\eps_0\nu_0/k_0^2\ll1$. 
The scaling factors 
$\sx(\x,t)=k^{3/2}(\x,t)/\eps(\x,t)$, $\st(\x,t)=k(\x,t)/\eps(\x,t)$, $\su(\x,t)=k^{1/2}(\x,t)$, and $\sz(\x,t)=\eps(\x,t)\nu(\x,t)/k^2(\x,t)$ 
describe the local 
macroscopic 
variations of 
the characteristic values of 
turbulent length, time, velocity, and inverse Reynolds number 
relative to the respective global reference values based on $k_0$, $\eps_0$, $\nu_0$. 
Non-uniform advection of the turbulent structures 
along the mean velocity field $\overline\u(\x,t)$ 
is captured via the mean flow function $\flow(s;\x,t)$, 
describing 
a mean flow pathline that passes through position $\x$ at time $t$. 
The 
distortion of the 
spatial coordinate 
system 
along these pathlines 
is of negligible effect 
due to the localizing influence of the 
time integration kernel $\eta$; 
see Figure~\ref{fig:mean_flow} for an illustration.  
The $3\times 3$ matrices $\Pt$ and $\bm{L}(\x,t)$ 
represent the projector onto the orthogonal complement of 
$\btheta$ 
associated with the condition of incompressibility 
and $(\x,t)$-dependent directional weightings associated with the 
prescribed  
Reynolds stress tensor, respectively.  

The explicit stochastic integral representation \eqref{eq:u'_inhom_intro} 
in combination with the two-scale approach 
allows to 
rigorously 
verify 
the 
desired 
model properties 
in an asymptotic sense 
as the ratio $\delta$ of representative turbulence scale to macro scale tends to zero, 
justifying 
the practical use of the model in scenarios with moderate inhomogeneities. 
In our main analytical 
results,  
Theorem~\ref{thm:kin_diss_div} and its ergodic counterpart Theorem~\ref{thm:ergodicity_inhom},  
we recover the prescribed characteristic flow quantities in terms of 
expected values and local spatio-temporal averages, 
respectively, related to the constructed fluctuation field. 
In particular, Theorem~\ref{thm:kin_diss_div} concerns 
not only the representation of 
turbulent kinetic energy $k(\x,t)$ and Reynolds stresses 
in terms of expected values related to 
the fluctuating velocity 
$\u'(\x,t)$ 
but also the representation of 
the dissipation rate $\eps(\x,t)$ 
in terms of expected values related to the 
velocity gradient $\nabla_{\x}\u'(\x,t)$. 
This is possible due to a 
thorough 
regularity analysis of the inhomogeneous random field 
%
%
and due to the specific structure of \eqref{eq:u'_inhom_intro}, 
preventing 
both 
a lack of invariance \wrt the underlying coordinate frame 
and 
significant distortions of the prescribed spectral structure 
by the 
non-uniform 
mean flow advection,   
while at the same time 
upholding 
the required regularity properties. 
%
If not avoided, these 
pitfalls 
would become particularly apparent 
in the 
analysis of the velocity gradient. 
Incompressibility is shown to hold in an approximate sense as well, 
and it holds exactly in the 
case of homogeneous and possibly anisotropic characteristic flow quantities. 
In addition, 
a 
novel inhomogeneous ergodicity result in Theorem~\ref{thm:ergodicity_inhom} 
essentially 
permits 
to replace the expected values from Theorem~\ref{thm:kin_diss_div} 
by respective local averages in space and time, 
thus providing an analytical verification for the validity of 
the ergodic hypothesis under inhomogeneous and potentially non-stationary flow conditions.
We further point out that the localizing effect of the time integration kernel 
$\eta$ 
in \eqref{eq:u'_inhom_intro} 
basically 
allows for a fully explicit approximation 
of the mean flow function 
$\flow(s;x,t)\approx \x - (t-s) \overline{\u}(\x,t)$ 
in the context of both analysis and simulation, 
facilitating 
efficient and flexible local evaluations of the nonlinearly advected field. 
Moreover, the 
presented 
model is 
customizable 
in that 
it is not only able to incorporate case-dependent RANS data 
%
but also offers a flexible choice of 
the employed model spectrum $E$ and time integration kernel $\eta$ 
%
as long as the respective regularity conditions in Model~\ref{model:turb_inhom} are fulfilled. 
This allows to 
optionally include additional aspects 
from general turbulence theory 
such as 
the $-5/3$ power law 
for the spectrum in the 
inertial range 
or estimates for the integral time scale; 
cf.~Examples \ref{ex:energy_spectrum}, \ref{ex:temporal_cor}, and \ref{ex:energy_spectrum_correlation_function}. 
Furthermore, 
local control of the 
width of the 
inertial range is facilitated 
by the 
employed 
non-dimensionalization 
that separates the 
local 
scales in space, time, and velocity 
from the 
local 
turbulence Reynolds number captured by the 
inverse of the 
second argument of $E$. 
The latter is related to the 
scale ratio between turbulent fine-scale (Kolmogorov) and large-scale structures 
and governs the 
shape of the 
local 
energy 
spectrum without necessarily involving 
infinite 
Reynolds number limits. 
Under the assumption of global homogeneity and isotropy, 
the model is 
perfectly 
compatible with the known spectral representations for homogeneous 
and isotropic 
turbulence used as a starting point. 

Seeking a 
compromise 
between accuracy and practical feasability, 
the scope of the model also involves 
considered limitations,  
such as the simplifying assumption of Gaussianity mentioned above. 
A further common simplification 
lies in the fact that the time integration kernel 
$\eta$ in \eqref{eq:u'_inhom_intro} 
is chosen to be independent of the wavenumber $\kappa$,   
capturing the temporal decay only in a mean sense, averaged over all wavelengths; 
compare, e.g., 
\cite{BBLC94, BED03, Kar+19, Bos+23}.  
It 
constitutes a pragmatic middle ground between 
the widely used but more restrictive 
hypothesis of frozen turbulence 
 \cite{Mann98, SSST14, Pol+20, ADDK20} 
and more 
detailed 
but computationally demanding 
descriptions of the temporal 
correlations \cite{Kan93, Laf+14, WH21, GMPSW22}. 
%


In summary, the article is structured as follows. 
After
revisiting the theory of homogeneous turbulence by 
setting up a general modeling framework for homogeneous fields and recalling 
stochastic Fourier-type integral representations 
in Section~\ref{sec:hom}, 
we specify our approach to non-dimensionalization and scaling 
in the context of inhomogeneous turbulence in Section~\ref{sec:scaling}. This is 
followed by a discussion of the scaling approach in regard to 
homogeneous fields including 
the derivation of alternative integral representation formulas 
in Section~\ref{sec:further_aspects}. 
Based on these preparations, our 
inhomogeneous random field model is 
finally 
introduced and analyzed in Section~\ref{sec:inhom}, 
before Section~\ref{sec:outlook} conludes with a short outlook on the accompanying paper \cite{AKLMW25}. 
The mathematical background of this article is substantiated in three appendices:
Auxiliary results concerning stochastic integration \wrt white noise and ergodicity of random vector fields are 
established 
in Appendices~\ref{app:white_noise} and \ref{sec:ergodicity}, respectively, 
while Appendix~\ref{sec:calc_estimates} contains technical calculations concerning an exemplary closure of our general inhomogeneous model. 

\subsection{General notation and conventions}

Throughout this article we set $\R^+ = (0,\infty)$, $\R^+_0 = [0,\infty)$ and denote by $\Re z$ and $\Im z$ the real and the imaginary part of a  complex number $z\in\C$.
Unless mentioned otherwise, it is assumed that $d$, $\ell$, $m$, $n\in\N$ are arbitrary natural numbers. 
We typically use small bold letters for vectors and capital bold letters for matrices. 
In order to distinguish dimensional variables from dimensionless variables, the former are indicated by the symbol $\vphantom{x}^\star$ (e.g.\ $t^\star$, $\x^\star$, $\u^\star$).
Elementary tensor operations are defined by $\bm{a} \cdot \bm{b} = \sum_{j} a_j b_j$, $\bm{a} \otimes \bm{b} = ( a_j b_k)_{j,k}$, $\bm{A} \cdot \bm{b} = (\sum_{k} A_{j,k} b_k)_j$, $\bm{a} \cdot \bm{B} = (\sum_{j} a_j B_{j,k})_k$, $\bm{A} \cdot \bm{B} = (\sum_{k} A_{j,k} B_{k,l})_{j,l}$, and $\bm{A} : \bm{B} = \sum_{j,k} A_{j,k} B_{j,k}$ for vectors and matrices $\bm{a} = (a_j)_j$, $\bm{b} = (b_j)_j$, $\bm{A} = (A_{j,k})_{j,k}$, $\bm{B} = (B_{j,k})_{j,k}$ of suitable dimensions. 
The Euclidean and Frobenius norm of real or complex vectors and matrices is denoted by $\| \bm{a} \| = (\sum_j |a_j|^2)^{1/2}$ and $\| \bm{A} \| = (\sum_{j,k} |A_{j,k}|^2)^{1/2}$, respectively.
If $\bm{f} \colon \R^n \to \R^d$ is a differentiable function, we write $D \bm{f}(\x)\in\R^{d\times n}$ for the Jacobian matrix of $\bm{f}$ at $\x\in\R^n$ and $\nabla \bm{f}(\x)\in\R^{n\times d}$ for the transpose thereof. 
The supremum norm of a function $\bm{F} \colon \R^n \to \R^{d \times \ell}$ is denoted by $\| \bm{F}\|_\infty = \sup_{\x \in \R^n} \| \bm{F}(\x) \|$. 
We use the notation $B_r^{(n)}(\x) = \{ \y \in \R^n \colon \| \y - \x \| \leq r \}$ for the closed ball in $\R^n$ with radius 
$r \in \R_0^+$ and center $\x \in \R^n$.
If it is clear from the context, we omit the notation of the dimension and write $B_r(\x) = B_r^{(n)}(\x)$. 
By $\B(U)$ we denote the Borel $\sigma$-algebra on a Borel set $U\subset \R^n$ and by $\lambda^n$ the Lebesgue measure on $(\R^n,\B(\R^n))$. 
Given a measure space $(U,\mathcal{A},\mu)$ and a finite-dimensional normed vector space $V$, we write $L^2(\mu;V) = L^2(U,\mu;V)$ for the space of (equivalence classes of) measurable and square-integrable functions from $U$ to $V$. 
Regarding the Fourier transform, we choose the convention
\begin{align*}
[\mathcal{F}(\phi)](\kap) = \frac{1}{(2\pi)^n} \int_{\R^n} e^{-i \kap\cdot\x} \phi(\x) \,\d \x, \quad
[\mathcal{F}^{-1}(\phi)](\x) =  \int_{\R^n} e^{i \kap\cdot\x} \phi(\kap) \,\d \kap
\end{align*}
for integrable functions $\phi\colon\R^n\to\C$ and $\x,\kap \in \R^n$. 
In particular, the same notation $\mathcal{F}$ is used for different dimensions $n$. 
The Fourier-transform of square-integrable functions is defined accordingly, and the Fourier transform of vector- and matrix-valued functions is considered as a component-wise operation. 
Unless mentioned otherwise, all random variables and random fields are assumed to be defined on a common underlying probability space $(\Omega,\mathscr{F},\Pr)$. 
If $V$ is a finite-dimensional normed vector space and $X \colon \Omega \to V$ is an integrable random variable, then  $\E[X] = \int_\Omega X \,\d\Pr$ denotes the expected value of $X$.


\section{Homogeneous turbulence revisited}\label{sec:hom}

In this section we revisit the theory of homogeneous turbulence in preparation of our inhomogeneous model. A suitable framework for random field modeling of homogeneous and spatially isotropic turbulence based on reformulating characteristic flow properties in terms of the energy spectrum is fixed in Subsection~\ref{subsec:revisited1}. Spectral representations of homogeneous turbulence fields involving stochastic Fourier-type integrals are recalled in Subsection~\ref{subsec:revisited2}.

\subsection{Random field modeling of homogeneous turbulence}\label{subsec:revisited1}

In homogeneous turbulence, the flow fields for mean velocity $\overline{\u}^\star$, kinematic viscosity $\nu^\star$, turbulent kinetic energy $k^\star$ and dissipation rate $\eps^\star$ are constant in space and time. The turbulent velocity is considered as a random field in space and time
\begin{align*}
\u^\star(\x^\star, t^\star) = \overline{\u}^\star + \u'^\star(\x^\star, t^\star), \qquad \E\bigl[ \u^\star(\x^\star, t^\star) \bigr]  = \overline{\u}^\star, \quad \E\bigl[ \u'^\star(\x^\star, t^\star) \bigr] = \bm{0},
\end{align*}
where the centered turbulent fluctuations  $\u'^\star$ are characterized by the properties
\begin{gather}\label{eq:kinetic_energy_hom*}
\frac12\,\E \Bigl[  \bigl\| \u'^\star(\x^\star, t^\star) \bigr\|^2  \Bigr] = k^\star,\\
\label{eq:dissipation_hom*}
 \frac12 \, \E \Bigl[  \bigl\| \nabla_{\x^\star} \u'^\star(\x^\star, t^\star) + \bigl(\nabla_{\x^\star}\u'^\star(\x^\star, t^\star)\bigr){\vphantom{)}}^{\!\top} \bigr\|^2  \Bigr] = \frac{\eps^\star}{\nu^\star}
\end{gather}
in terms of turbulent kinetic energy, dissipation rate, and kinematic viscosity.

The turbulent fluctuations are viewed as a stochastic process parametrized in $k^\star$, $\eps^\star$, $\nu^\star$ and in the mean velocity $\overline{\u}^\star$, i.e., $\u'^\star(\x^\star, t^\star)=\u'^\star(\x^\star, t^\star; \overline{\u}^\star, \nu^\star, k^\star,\eps^\star)$, where $\overline{\u}^\star$ enters via the hypothesis of advection-driven turbulence \cite{Tay38, Wil06}. Then, a dimensionless view allows the reduction of the parameter dependencies, yielding 
\begin{align}\label{eq:ansatz_u'_hom*}
\u'^\star(\x^\star, t^\star; \overline{\u}^\star, \nu^\star, k^\star,\eps^\star) = \sqrt{k^\star}\, \v \Bigl( \frac{\eps^\star}{\sqrt{k^\star}^3} \bigl(\x^\star - t^\star\overline{\u}^\star \bigr), \frac{\eps^\star}{k^\star} t^\star;  \frac{\eps^\star\nu^\star}{(k^\star)^2}\Bigr),
\end{align}
when using the typical turbulent length $\sqrt{k^\star}^3/\eps^\star$ and time $k^\star/\eps^\star$ for non-dimensionalization. The remaining single parameter $\zeta=\eps^\star\nu^\star/(k^\star)^2$ is the inverse of the turbulence Reynolds number, \cite[Section 6.3]{Pope00}. It is not only 
proportional to the inverse of the turbulent viscosity ratio, but also indicates the scale ratio between turbulent fine-scale (Kolmogorov) and large-scale structures, $0< \zeta\ll 1$. 
The dimensionless field $\v$ describes the turbulent fluctuations from a mean flow perspective and with respect to the turbulent scales.

Following Kolmogorov's local isotropy hypothesis \cite{Kol41, Kol62}, $\v(\x,t; \zeta)$ is modeled as a homogeneous, spatially isotropic, centered random field in  $(\x,t)\in \R^3\times \R$. As a simplifying assumption common in the context of simulation methods for random velocity fields,  
we additionally choose the random field $\v$ to be Gaussian, reflecting the fact that we focus on the second-order statistics of the turbulent fluctuations.  
Further assuming mean-square continuity, mean-square differentiability in $\x$, and divergence freedom
\begin{align}\label{eq:divergence_hom*}
\nabla_{\x} \cdot \v(\x, t;\zeta) = 0
\end{align}
as well as absolute continuity of the spectral measures of the components $v_j$ \wrt Lebesgue measure, it follows that the matrix-valued spectral measure of $\v$ admits a matrix-valued spectral density that can be expressed in terms of a single scalar-valued 
function $f$. 
%
In particular, the spectral density has the form 
\begin{align}\label{eq:sd_general}
f(\nkap,\omega;\zeta) \Pk, \qquad \P(\kap) = \bm{I}- \frac{\kap}{\nkap}\otimes\frac{\kap}{\nkap},
\end{align}
$(\kap,\omega)\in \R^3\times\R$, where $f(\nkap,\omega;\zeta)$ is non-negative and symmetric in $(\kap,\omega)$ and the projector $\P(\kap)\in \R^{3\times 3}$ is defined as the identity matrix $\bm{I}$ if 
$\kap=\bm{0}$,  cf.\ \cite[Section~12.3]{MY75}, \cite[Section~3.4.3]{Bre14}. 
The inverse Fourier transform of the spectral density represents the covariance function, i.e.,   
\begin{align}\label{eq:v_covariance}
\E \bigl[ \v(\x,t;\zeta) \otimes \v(\tilde\x,\tilde t;\zeta) \bigr] = \int_{\R^3 \times \R} \exp\Bigl\{ i\bigl(\kap\cdot(\x-\tilde\x)  + \omega(t-\tilde t\;\!) \bigr) \Bigr\} f\bigl( \nkap,\omega;\zeta) \Pk \,\d(\kap,\omega).
\end{align}
Choosing a product structure for $f$ of type
\begin{align}\label{eq:sd_product_approach}
f(\kappa,\omega;\zeta) = \fx(\kappa;\zeta) \ft(\omega),
\end{align}
$(\kappa,\omega)\in \R^+_0\times \R$, separates spatial and temporal 
correlation effects; compare, e.g., \cite{Kra70, SSC01, Ewe+11, HMRW13, Zwi+24}. 
It accounts for an initial spatial correlation structure of the turbulence pattern transported unchanged by the mean flow, which is superposed by a natural temporal decay.
In particular, the temporal decay is considered as independent of 
the wavenumber $\kappa$ and 
the inverse turbulent viscosity ratio $\zeta$. 
Balancing model accuracy and numerical feasability, the structural assumption \eqref{eq:sd_product_approach} offers a broader applicability than the widely used frozen turbulence hypothesis, 
formally corresponding to $\ft(\omega)=\delta_0(\omega)$ Dirac distribution, 
but constitutes a 
simplification 
compared to more involved spatio-temporal correlations models; 
see the respective references in Section~\ref{subsec:intro_contrib} above. 

Homogeneous  and spatially isotropic turbulence is usually described in terms of the energy spectrum $E(\kappa;\zeta)$, defined for $\kappa\in\R^+$ as the surface integral of $\fx(\|\kap\|;\zeta)$ over the sphere of all wave vectors $\kap\in\R^3$ with magnitude $\kappa$. Reformulating the characteristic 
flow conditions \eqref{eq:kinetic_energy_hom*}, \eqref{eq:dissipation_hom*} in terms of integral conditions on the energy spectrum, our framework for random field modeling of homogeneous turbulence reads as follows. 
\begin{model}[Homogeneous turbulence field]\label{model:turb_hom*}
Let $\v = (\v(\x,t; \zeta))_{(\x,t) \in \R^3 \times \R}$ be a homogeneous, mean-square continuous, centered, $\R^3$-valued Gaussian random field in $(\x,t)\in \R^3\times \R$ with parameter $\zeta\in \R^+$ and spectral density given by \eqref{eq:sd_general} and \eqref{eq:sd_product_approach}, where
\begin{enumerate}[label= \emph{\textbf{\alph*)}}]
\item the function $\fx(\bdot;\zeta)\colon \R_0^+ \to \R_0^+$ is measurable and the associated energy spectrum $E(\kappa; \zeta) = 4\pi \kappa^2\fx(\kappa;\zeta)$ satisfies
\begin{align}\label{eq:energy_spectrum_1}
\int_0^\infty E(\kappa;\zeta) \,\d\kappa = 1, \qquad \int_0^\infty \kappa^2 E(\kappa;\zeta) \,\d\kappa = \frac{1}{2\zeta},
\end{align} 
\item the function $\ft \colon \R \to \R_0^+$ is measurable, symmetric, and such that
\begin{align}\label{eq:temporal_correlation}
\int_\R \ft(\omega)\,\d \omega = 1.
\end{align}
\end{enumerate}
Then we refer to $\u'^\star$ described via \eqref{eq:ansatz_u'_hom*} as a \emph{homogeneous turbulence field}.
\end{model}

Let $\bm C_{\mathrm{x}}(\bdot;\zeta) \colon\R^3\to\R^{3\times 3}$ and $C_{\mathrm{t}} \colon\R\to\R$ 
denote the correlation functions given via (inverse) Fourier transform as 
\begin{align}\label{eq:correlation_functions}
\bm C_{\mathrm{x}}(\bm{y};\zeta) = \bigl[\mathcal{F}^{-1}\bigl(\fx(\|\!\bdot\!\|;\zeta) \P(\bdot)\bigr)\bigr](\bm{y}),\qquad C_{\mathrm{t}}(s) = \bigl[\mathcal{F}^{-1}\bigl(\ft(\bdot)\bigr)\bigr](s).
\end{align}
The following proposition summarizes the basic properties of Model~\ref{model:turb_hom*}.

\begin{proposition}[Characteristic flow properties]
\label{prop:properties_v_hom}
Assume Model~\ref{model:turb_hom*} to be given, then the random field $\v = (\v(\x,t; \zeta))_{(\x,t) \in \R^3 \times \R}$ is spatially isotropic, continuously mean-square differentiable in $\x$, and the characteristic turbulent flow properties \eqref{eq:kinetic_energy_hom*} and \eqref{eq:dissipation_hom*} as well as the condition of incompressibility \eqref{eq:divergence_hom*} are satisfied. 
In addition, we have
\begin{align}\label{eq:v_covariance_xt} 
\E \bigl[ \v(\x,t;\zeta) \otimes \v(\tilde\x,\tilde t;\zeta) \bigr]
= 
\bm C_{\mathrm{x}}( \x-\tilde\x;\zeta) \,C_{\mathrm{t}}(t-\tilde t\;\!)
\end{align}
with $\bm C_{\mathrm{x}}( \bm0;\zeta)=(2/3)\bm{I}$ and $C_{\mathrm{t}}(0)=1$.
\end{proposition}

\begin{proof}
First note that the correlation formula \eqref{eq:v_covariance_xt} is a direct consequence of \eqref{eq:v_covariance}, \eqref{eq:sd_product_approach}, and \eqref{eq:correlation_functions}.
Spatial isotropy follows from the fact that for every orthogonal matrix  
$\bm{Q} \in \R^{3\times 3}$ and every $\kap \in \R^3$ we have 
$\bm P( \bm{Q} \kap) = \bm{Q} \Pk \bm{Q}^\top\!$, implying 
$\bm{C}_\mathrm{x}( \bm{Q} \y;\zeta) = \bm{Q} \bm{C}_\mathrm{x}( \y;\zeta) \bm{Q}^\top\!$, $\y\in\R^3$. 
Combining this,  \eqref{eq:energy_spectrum_1}, \eqref{eq:temporal_correlation}, and the fact that $\trace\P(\kap)=2$ for $\kap\neq\bm0$ establishes the assertions concerning $\bm C_{\mathrm{x}}( \bm0;\zeta)$ and $C_{\mathrm{t}}(0)$.  Next observe that \eqref{eq:energy_spectrum_1}, \eqref{eq:temporal_correlation}, and the dominated convergence theorem yield that $\bm{y}\mapsto\bm C_{\mathrm{x}}(\bm{y};\zeta)$ is twice continuously differentiable and $C_{\mathrm{t}}$ is continuous. Lemma~\ref{lem:cov_ms-diff} in Appendix~\ref{app:white_noise} thus ensures that $\v$ is partially mean-square differentiable \wrt $\x$. Formula \eqref{eq:derivative_of_cov_function} in Lemma~\ref{lem:cov_ms-diff} further implies the identity 
\begin{align*}
\E\Bigl[\bigl\| \partial_{x_j}\v(\x,t; \zeta) - \partial_{\tilde x_j}\v(\xalt,\talt ;\zeta) \bigr\|^2\Bigr] 
= 
2\,\trace \!\bigl(\partial_{y_j}^2\bm{C} _\mathrm{x}\bigr)(\x-\xalt;\zeta) \, C_\mathrm{t}(t-\talt\;\!) 
-2\, \trace \!\bigl(\partial_{y_j}^2\bm{C} _\mathrm{x}\bigr)(\bm{0};\zeta) \, C_\mathrm{t}(0)
 \end{align*}
and hence the mean-square continuity of $\partial_{x_j}\v$. The kinetic energy property \eqref{eq:kinetic_energy_hom*} is a  consequence of \eqref{eq:ansatz_u'_hom*} and the fact that $\E \bigl[ \|\v(\x,t;\zeta) \|^2 \bigr]  = \trace \bm{C}_\mathrm{x}( \bm{0}; \zeta) C_\mathrm{t}(0)=2$. 
Moreover, note that \eqref{eq:derivative_of_cov_function} in Lemma~\ref{lem:cov_ms-diff} 
and \eqref{eq:energy_spectrum_1} yield the identities 
$\E \bigl[ \| \nabla_{\x} \v(\x,t;\zeta) \|^2 \bigr] = - \trace ( \Delta_\y \bm{C}_\mathrm{x}) ( \bm{0}; \zeta) C_\mathrm{t}(0) = 1/\zeta$ and
\begin{align*}
 \E \Bigl[ \nabla_{\x} \v(\x,t;\zeta)  : \bigl(\nabla_{\x}\v(\x,t;\zeta)\bigr){\vphantom{)}}^{\!\top} \Bigr] &  = - (\nabla_\y \cdot \bm{C}_\mathrm{x} \cdot \nabla_\y)(\bm{0};\zeta) \, C_\mathrm{t}(0) = \E \Bigl[ \bigl|\nabla_\x \cdot \v(\x,t;\zeta)\bigr|^2 \Bigr],
\end{align*}
where the divergence of $\bm{C}_\mathrm{x}$ vanishes since $\kap\cdot\bm{P}(\kap)=\bm{0}$. 
Combining this with \eqref{eq:ansatz_u'_hom*} establishes the dissipation and incompressibility properties \eqref{eq:dissipation_hom*} and \eqref{eq:divergence_hom*}.
\end{proof}

Closing Model~\ref{model:turb_hom*} requires the specific choice of $E(\kappa;\zeta)$ and $\ft(\omega)$ (or alternatively of $C_\mathrm{t} (s)$, cf.~\eqref{eq:correlation_functions}).
The energy spectrum $E$ of isotropic turbulence is well studied in literature by means of theoretical and experimental investigations; see, e.g., \cite{Hin75,Fri95,Pope00} and references within. Of central importance is Kolmogorov's 5/3 law for the inertial subrange. 

\begin{example}[Energy spectrum]\label{ex:energy_spectrum} 
In \cite{MW06} the energy spectrum is modeled as
\begin{equation}\label{eq:def_E}
\begin{aligned}
E(\kappa;\zeta) = C_\mathrm{K}
\begin{cases}
\kappa_1^{-5/3} \sum_{j=4}^6 a_j\bigl( \tfrac{\kappa}{\kappa_1} \bigr)^{j}, & \kappa < \kappa_1,
\\ \kappa^{-5/3}, & \kappa_1 \leq \kappa \leq \kappa_2,
\\ \kappa_2^{-5/3} \sum_{j=7}^9 b_j \bigl(\tfrac{\kappa}{\kappa_2} \bigr)^{-j}, & \kappa_2 < \kappa, 
\end{cases}
\end{aligned}
\end{equation}
compare Figure~\ref{fig:model_spectrum}.
The model obeys Kolmogorov's 5/3 law, and its limit behavior goes with $E(\kappa;\bdot) \sim \kappa^4$ for $\kappa \rightarrow 0$ \cite{BP56} as well as $E(\kappa;\bdot) \sim \kappa^{-7}$ for $\kappa \rightarrow \infty$ \cite{Hei48}.  The Kolmogorov constant $C_\mathrm{K}$ is taken as $C_\mathrm{K}=1.5$,  \cite{YZ97}. The coefficients $a_j$, $b_j$ are chosen as $a_4 = 230/9$, $a_5 = -391/9$, $a_6 = 170/9$, $b_7 = 209/9$, $b_8 = -352/9$, and $b_9 = 152/9$ to ensure $E(\bdot;\zeta)\in {C}^2(\R_0^+)$. The $\zeta$-dependent transition wave numbers $\kappa_1$ and $\kappa_2$ are implicitly given by the integral identities~\eqref{eq:energy_spectrum_1}. 
The condition $0< \kappa_1< \kappa_2 < \infty$ is equivalent to $0< \zeta < \zeta_\mathrm{crit} \approx 0.14$. The bound on $\zeta$ is irrelevant for practical applications, since the general turbulence theory assumes the inverse turbulent viscosity ratio to satisfy $\zeta \ll 1$. Model variants with different regularities are conceivable.
\end{example}

\begin{figure}[h!]
	\includegraphics[width=.4\textwidth]{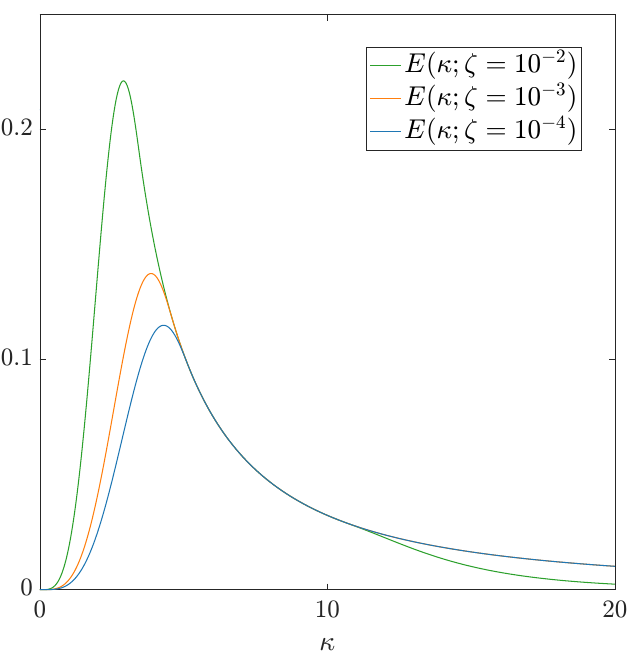}\qquad \quad
	\includegraphics[width=.4\textwidth]{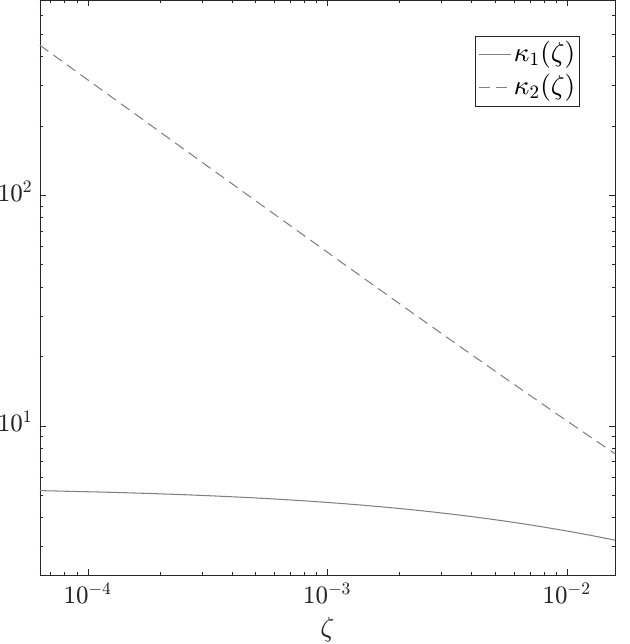}
	\caption{
	Model spectrum from Example~\ref{ex:energy_spectrum} 
	for three different values of the inverse turbulent viscosity ratio $\zeta$ (left)
	and implicitly determined transition wave numbers in dependence on $\zeta$ (right). 
	}
	\label{fig:model_spectrum}
\end{figure}

Alternative models for the energy spectrum can be found in, e.g., \cite[Section~6.5.3]{Pope00}, where the limit behavior for $\kappa \rightarrow \infty$ satisfies an exponential decay. The limit behavior -- indicating the dissipation of the turbulent fine-scale structure to heat -- is undoubtedly subject to the freedoms of modeling. 

The temporal correlation function $C_\mathrm{t}$ can be written as a convolution
$C_{\mathrm{t}}(s) = \int_\R \eta(s-r) \eta(r) \,\d r $ in terms of the symmetric kernel $\eta \colon \R \to \R$  satisfying 
\begin{align}\label{eq:eta}
\eta(s) = \frac{1}{\sqrt{2\pi}} \bigl[\mathcal{F}^{-1}\bigl(\fts(\bdot)\bigr)\bigr](s), \qquad \int_\R \eta^2(s)\,\d s = 1
\end{align}
due to the properties of the spectral density $\ft$. To model the temporal correlations this gives rise to a variety of possible choices for $\eta$.  From a numerical point of view, $\eta \in C^2(\R)$ with compact support is favorable. 
Heuristics exist \wrt the Lagrangian integral time scale $T_{\text{L}}$ given in dimensionless form as 
$T_{\text{L}}(\zeta)=(\E \bigl[ \|\v(\mathbf{0},0;\zeta) \|^2 \bigr])^{-1} \int_0^\infty \E \bigl[ \v(\mathbf{0},0;\zeta) \cdot  \v(\x_s,s;\zeta) \bigr] \, \d s$ 
with 
$\x_s=\int_0^s \v(\x_r,r; \zeta)\, \d r$ 
and its Eulerian counterpart $T_{\text{E}}$, i.e., 
$T_{\text{E}}(\zeta)=(\E \bigl[ \|\v(\mathbf{0},0;\zeta) \|^2 \bigr])^{-1}  \int_0^\infty \E \bigl[ \v(\mathbf{0},0;\zeta) \cdot \v(\mathbf{0},s;\zeta)\bigr]\, \d s$, 
\cite{TL72,Hin75, Lu94}. In particular, estimates for the integral time scales have been deduced and applied that are independent of the turbulent viscosity ratio,  
motivating the simplifying assumption that $\ft$ is independent of $\zeta$, cf.~\eqref{eq:sd_product_approach}, since $T_{\text{E}}= \int_{0}^\infty C_\mathrm{t}(s)\,\d s=\pi \ft (0)$.

\begin{example}[Temporal correlations]\label{ex:temporal_cor} 
Consider the time integration kernel
\begin{align*}
\eta(s) = \begin{cases}
\frac{a}{\sqrt{s_\mathrm{c}}} \, p \bigl( \frac{|s|}{s_\mathrm{c}} \bigr), & |s| \leq s_\mathrm{c},
\\ 0, & |s| > s_\mathrm{c},
\end{cases}
\end{align*}
with $p(z) = 1-10z^3+15z^4-6z^5$, $a = ({231}/{181})^{1/2}$ and $s_\mathrm{c} = ({362}/{231}) T_{\mathrm{E}}$ that satisfies the properties stated  above, 
compare Figure~\ref{fig:temp_corr}. 
In applications focusing on a Lagrangian perspective, the value for $T_{\mathrm{E}}$ may be chosen such that model-based numerical simulations of the Lagrangian integral time scale $T_{\mathrm{L}}$ are in accordance with respective estimates  in literature, i.e., $T_{\mathrm{L}}=0.16$ in \cite[p.\,426] {Hin75} and \cite[Section~3.1]{Lu94}. 
Employing the model spectrum from Example~\ref{ex:energy_spectrum}, 
the simulated values for $T_{\mathrm{L}}$ exhibit only a minor  dependency on the inverse viscosity ratio $\zeta$ 
in the relevant range $\zeta \leq 10^{-3}$
 and indicate $T_{\mathrm{E}}=0.3$ as a suitable choice. 
\end{example}

\begin{figure}[h!]
	\includegraphics[width=.4\textwidth]{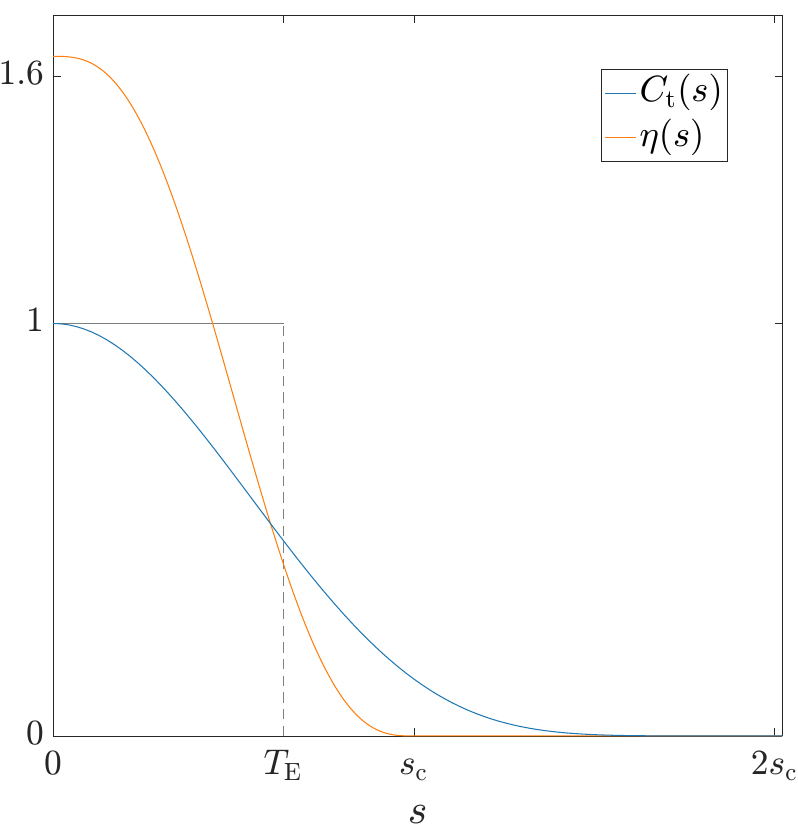} \qquad \quad
	\includegraphics[width=.4\textwidth]{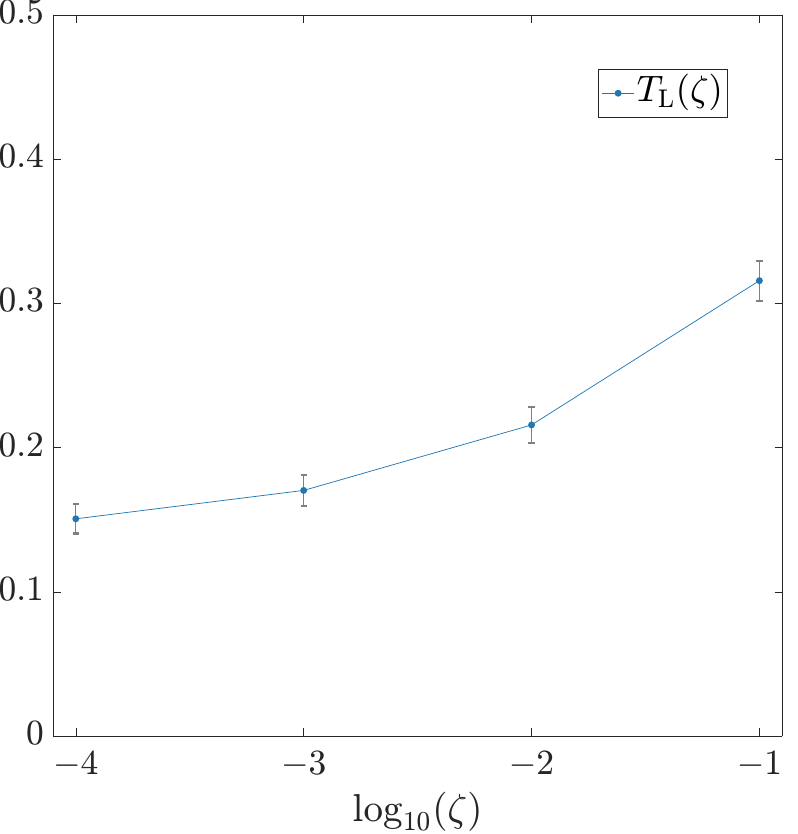}
	\caption{ 
		     Temporal correlation function $C_{\mathrm{t}}$
		     and time integration kernel $\eta$ 
		     from Example~\ref{ex:temporal_cor} (left) 
		     together with associated Monte Carlo estimates and $95\%$ confidence intervals for the Langrangian time scale $T_{\mathrm{L}}$ in dependence on the 
		     inverse turbulent viscosity ratio $\zeta$ 
		     (right). Each estimate is based on $20\,000$ simulated 
		     fluid particle trajectories. 
		     }
	\label{fig:temp_corr}
\end{figure}

Alternative models for the temporal correlations have been employed in,  e.g., 
\cite{Max87, CSS93, BBLC94, Dav07, MW11, Kar+19}. 
The frozen turbulence hypothesis \cite{Tay38}, in which the turbulence patterns are transported by the mean flow without changing their structure over time, assumes $T_{\mathrm{E}}\rightarrow \infty$. This corresponds to $C_{\mathrm{t}}(s) = 1$ and $\ft(\omega) = \delta_0(\omega)$. 

\begin{remark}[Sample path regularity] \label{rem:sample_path_reg}
Model~\ref{model:turb_hom*} and Proposition~\ref{prop:properties_v_hom} involve continuity and differentiability properties of random fields in the mean-square sense, cf.\ Appendix~\ref{subsec:appendix_differentiability}.  
In order to simplify the presentation, we refrain from analyzing pathwise regularity properties of random fields such as pathwise continuity and pathwise differentiability. 
Instead, the focus lies on mean-square regularity properties, and 
in particular all 
derivatives of random fields such as the gradients in \eqref{eq:kinetic_energy_hom*}, \eqref{eq:divergence_hom*} are understood in the mean-square sense. 
It is well-known that under suitable additional assumptions, 
the considered random fields admit modifications with differentiable sample 
paths; see, e.g., \cite{Kru01, AT07, Pot10}. 
In this case the  mean-square derivatives 
coincide with the pathwise derivatives $\Pr$-almost surely.
\end{remark}

\subsection{Stochastic spectral representations}\label{subsec:revisited2}

The theory of stationary random fields provides the means to translate spectral representations of covariance functions of the form \eqref{eq:v_covariance} into spectral representations  of the trajectories of the associated random fields involving stochastic Fourier-type integrals. The corresponding result is sometimes referred to as the Cramér-Khinchin decomposition, and our standard references in this context are \cite[Chapter~3]{Bre14},  \cite[Chapter~IV]{GS74}, and \cite[Chapter~6]{MY75}.  
In Proposition~\ref{prop:sprectral_rep_v_hom} below we present a specific variant of this result which is suitable for our purpose. 
Similar decompositions can be found, e.g., in \cite[Proposition~11.19]{Kal97} and \cite[Section~2.1]{KS06}. 
The employed concepts of complex-valued white noise and the associated stochastic integral for deterministic integrands are recalled and specified in Appendix~\ref{app:white_noise}. 

\begin{proposition}[Spectral representation]\label{prop:sprectral_rep_v_hom}
Assume the random field $\v$ from Model~\ref{model:turb_hom*} to be given. Then there exists a $\C^3$-valued Gaussian white noise $\wn$ on $\R^3 \times \R$ with structural measure $2 \lambda^3 \otimes \lambda^1$ 
in the sense of Definition \ref{def:gaussian_white_noise} 
such that for all $(\x,t) \in \R^3 \times \R$ it holds $\Pr$-almost surely that
\begin{align}\label{eq:v_integral_representation_1}
\v(\x,t;\zeta) & = \Re \!\int_{\R^3\times\R} \exp\bigl\{ i(\kap\cdot\x + \omega t)\bigr\} \fxs\bigl(\nkap;\zeta\bigr)\fts(\omega) \Pk \cdot \wn(\d\kap,\d\omega).
\end{align}
\end{proposition}

\begin{proof}
Recall that $\v$ is defined on the underlying probability space $(\Omega,\mathscr F,\Pr)$. 
Let $\widetilde{\wn}$ be a $\C^3$-valued Gaussian white noise on $\R^3 \times \R$ with structural measure $2 \lambda^3 \otimes \lambda^1$, possibly defined on a different probability space $(\widetilde\Omega,\widetilde{\mathscr F},\widetilde{\Pr})$, and let $\widetilde\v= (\widetilde\v(\x,t; \zeta))_{(\x,t) \in \R^3 \times \R}$ be defined by  \eqref{eq:v_integral_representation_1} with $\widetilde\v$ and $\widetilde\wn$ in place of $\v$ and $\wn$. 
Then $\widetilde\v$ is a centered Gaussian random field and an application of Corollary~\ref{cor:covmain} 
implies that the covariance functions of $\widetilde\v$ and $\v$ coincide. 
It follows that  $\widetilde\v$ and $\v$ are identically distributed. 
Combining this and, e.g., \cite[Theorem~5.10]{Kal97} ensures that -- after possibly extending the probability space $(\Omega,\mathscr F,\Pr)$ -- there exists a $\C^3$-valued white noise $\wn$ on $\R^3 \times \R$ defined on $(\Omega,\mathscr F,\Pr)$ such that the tuples $(\widetilde\v,\widetilde\wn)$ and $(\v,\wn)$ have the same distribution. 
The construction of $\widetilde\v$ via $\widetilde\wn$ thus implies the representation formula \eqref{eq:v_integral_representation_1}.
\end{proof}

We remark that the change of variables formula for spherical coordinates established in Corollary~\ref{cor:polar} 
implies that the spectral decomposition \eqref{eq:v_integral_representation_1} can be rewritten in the form
\begin{align}\label{eq:v_integral_representation_2}
\v(\x,t;\zeta) & = \Re\!\int_{\R^+ \times S^2\times\R} \exp\bigl\{ i(\kappa \;\!\btheta \cdot\x + \omega t)\bigr\} E^{1/2}(\kappa; \zeta)\fts(\omega) \P( \btheta) \cdot  \wn(\d\kappa,\d\btheta,\d\omega), 
\end{align}
which particularly highlights the isotropic structure of the random field. 
Here, with a slight abuse of notation, $\wn(\d\kappa,\d\btheta,\d\omega)$ denotes 
a $\C^3$-valued Gaussian white noise on $\R^+\times S^2\times\R$ with structural measure $2\lambda^1|_{\R^+}\otimes U_{S^2}\otimes\lambda^1$, 
obtained by transformation of 
$\wn(\d\kap,\d\omega)$ in \eqref{eq:v_integral_representation_1}.  
The notations $S^2$ and $U_{S^2}$ refer to the unit $2$-sphere and the uniform distribution on the unit $2$-sphere (i.e., the normalized surface measure), respectively.

\section{Non-dimensionalization and scaling}\label{sec:scaling}

With the idea of formulating an inhomogeneous turbulence field in terms of Fourier-type stochastic integrals, this section presents a non-dimensionalization and a specific scaling that will allow the assessment of the inhomogeneous model and its properties by asymptotic analysis. In inhomogeneous turbulence the flow fields for mean velocity $\overline{\u}^\star(\x^\star, t^\star)$, kinematic viscosity $\nu^\star(\x^\star, t^\star)$, turbulent kinetic energy $k^\star(\x^\star, t^\star)$ and dissipation rate $\eps^\star(\x^\star, t^\star)$ are functions of space $\x^\star$ and time $t^\star$. The turbulent velocity is composed of
\begin{align*}
\u^\star(\x^\star, t^\star) = \overline{\u}^\star(\x^\star, t^\star) + \u'^\star(\x^\star, t^\star),
\qquad \E\bigl[ \u^\star(\x^\star, t^\star) \bigr]  = \overline{\u}^\star(\x^\star, t^\star) , \quad \E\bigl[ \u'^\star(\x^\star, t^\star) \bigr] = \bm{0},
\end{align*}
where we consider the turbulent fluctuations  $\u'^\star$ as a centered random field depending on the above flow quantities.

For non-dimensionalization we introduce four reference values: 
\begin{align*}
x_0, \qquad \nu_0, \qquad k_0,\qquad \eps_0,
\end{align*}
i.e., a macro length $x_0$ coming from the (outer) geometry of the flow problem as well as typical values for kinematic viscosity $\nu_0$, kinetic energy $k_0$, and dissipation rate $\eps_0$ characterizing the turbulence. Depending on them, we consider 
\begin{align*}
u_0 = \sqrt{k_0}, \qquad t_0 = x_0/\sqrt{k_0}, \qquad x_\mu = \sqrt{k_0}^3/\eps_0, \qquad t_\mu = k_0/\eps_0.
\end{align*}
The quantities $x_\mu$ and $t_\mu$ indicate a micro scale; they are the typically turbulent length and time known from the homogeneous turbulence theory \eqref{eq:ansatz_u'_hom*}. Our treatment of inhomogeneous turbulence is based on introducing a macro scale, by using a macro length $x_0$ and respective time $t_0$. The velocity value $u_0$ we choose satisfies $u_0 = x_0/t_0 = x_\mu/t_\mu $. It yields a balancing between mean flow and fluctuations. From the perspective of the macro scale, we introduce the dimensionless functions for the flow quantities as
\begin{equation}\label{eq:scaling}
\begin{aligned}
u_0 \u(\x,t) = \u^\star(x_0 \x,t_0 t), \qquad u_0 \overline{\u}(\x,t) & = \overline{\u}^\star(x_0\x,t_0 t), \qquad u_0 \u'(\x,t) = \u'^\star(x_0 \x,t_0 t),\\
\nu_0 \nu(\x,t) = \nu^\star(x_0 \x,t_0 t), \,\, \qquad k_0 k(\x,t) & = k^\star(x_0 \x,t_0 t), \,\,\,\, \qquad \eps_0 \eps(\x,t) = \eps^\star(x_0 \x,t_0 t).
\end{aligned}
\end{equation}
where the dimensionless turbulent velocity satisfies
\begin{align*}
\u(\x,t) = \overline{\u}(\x,t) + \u'(\x,t) 
\end{align*} 
The turbulent flow is thus characterized by two dimensionless numbers, i.e., 
\begin{align}\label{eq:z_delta}
z = \frac{\eps_0\nu_0}{k_0^2},\qquad \delta = \frac{x_\mu}{x_0} = \frac{t_\mu}{t_0} =\frac{\sqrt{k_0}^3}{\eps_0 x_0},
\end{align}
satisfying $z,\delta\ll 1$ in practical applications.
The inverse turbulent viscosity ratio (inverse of the turbulence Reynolds number) $z$ is already known from the homogeneous theory. It indicates the scale ratio between the turbulent fine-scale structure (Kolmogorov scales of small vortices dissipating into heat) and the turbulent large-scale structure (large energy-bearing vortices). In addition, we face here a turbulence scale ratio $\delta$ that prescribes the ratio between the turbulent scale and the macro scale coming from the (outer) geometry of the flow problem. 
Moreover, we have the following space- and time-dependent scaling functions
\begin{equation}\label{eq:sigma_x_t_u_z}
\begin{aligned}
\sx(\x,t) = \!\frac{\sqrt{k(\x,t)}^{3}}{\eps(\x,t)}, \quad \st(\x,t) = \frac{k(\x,t)}{\eps(\x,t)}, \quad \su(\x,t) = \sqrt{k(\x,t)}, \quad \sz(\x,t) = \frac{\eps(\x,t)\nu(\x,t)}{k(\x,t)^2}.
\end{aligned}
\end{equation}

The following inhomogeneous turbulence reconstruction is based on a two-scale consideration. The micro scale $x_\mu$, $t_\mu$ shows the characteristics of locally homogeneous isotropic turbulence that are superposed on the macro scale $x_0$, $t_0$.
The turbulence scale ratio  $\delta\ll 1$ plays a crucial role in the asymptotic analysis of the model, as we will show that the characteristic flow properties of the turbulent field concerning dissipation rate and incompressibility are satisfied as $\delta \rightarrow 0$. Note that the Kolmogorov scale does not provide any additional information for the asymptotics in this context, since it can be interpreted as rescaling of the micro scale using the inverse turbulent viscosity ratio $z$, i.e., for the Kolomogorov length $x_\mathrm{K}$ and time $t_\mathrm{K}$ 
we have 
$x_\mathrm{K}=z^{3/4} x_\mu$ and $t_\mathrm{K}=z^{1/2} t_\mu$.


\section{Further aspects of homogeneous turbulence}
\label{sec:further_aspects}

This section continues the discussion of homogeneous turbulence initiated in Section~\ref{sec:hom}. 
The two-scale perspective introduced in Section~\ref{sec:scaling} is taken into account in Subsection~\ref{subsec:hom_scaling}, including the investigation of ergodicity properties in Proposition~\ref{prop:ergodicity_hom}. 
Subsection~\ref{subsec:hom_alternative} is concerned with the derivation of modified stochastic integral representation formulas in Lemmata~\ref{lem:transformation1}--\ref{lem:transformation3} that are of key importance for the conceptualization of our inhomogeneous model.

\subsection{The homogeneous model in view of scaling}\label{subsec:hom_scaling}

In the following, we present the homogeneous turbulence model of Section~\ref{sec:hom} from the two-scale perspective. Using the scaling introduced in Section~\ref{sec:scaling}, the turbulent velocity fluctuations \eqref{eq:ansatz_u'_hom*} are given in dimensionless form as
\begin{align}\label{eq:ansatz_u'_hom}
\u'(\x,t) = \su\, \v \Bigl( \frac{1}{\delta\sx} \bigl(\x - t\overline{\u} \bigr), \frac{1}{\delta\st} \,t ;  \sz z \Bigr).
\end{align} 
The characteristic turbulent flow properties \eqref{eq:kinetic_energy_hom*} and \eqref{eq:dissipation_hom*} as well as the condition of incompressibility \eqref{eq:divergence_hom*} become
\begin{align}\label{eq:kinetic_evergy_hom}
\frac 12\,\E \Bigl[ \bigl\| \u'(\x,t) \bigr\|^2 \Bigr] & = k \,= \su^2,\\
\label{eq:dissipation_hom}
\frac12 \,\delta^2 z\,\E \Bigl[ \bigl\| \nabla_{\x} \u'(\x,t) + \bigl( \nabla_{\x}\u'(\x,t) \bigr){\vphantom{)}}^{\!\top} \bigr\|^2 \Bigr] & = \frac{\eps}{\nu} = \frac{\su^2}{\sx^2 \sz},\\
\label{eq:divergence_hom}
\delta \,\nabla_{\x} \cdot \u'(\x, t) & = 0.
\end{align}
Note that the condition of incompressibility originally carries the scaling factor $ {\sx}/{\su}$, which is of order one. Since it plays no role in the further asymptotic analysis with respect to the small parameter $\delta \ll 1$, it is omitted in \eqref{eq:divergence_hom}.

Combining \eqref{eq:ansatz_u'_hom} with the covariance formula~\eqref{eq:v_covariance_xt}, the correlations of the fluctuations read
\begin{align}\label{eq:cov_u'_hom}
\E \bigl[ \u'(\x,t) \otimes \u'(\tilde\x,\tilde t\;\!) \bigr] 
= \su^2\, \bm C_{\mathrm{x}}\Bigl( \frac{1}{\delta \sx}\bigl(\x-\tilde\x- (t-\tilde t\;\!)\overline{\u}\bigr);  \sz z \Bigr) C_{\mathrm{t}} \Bigl( \frac{1}{\delta\st} (t-\tilde t\;\!)\Bigr), 
\end{align}
which simplifies to $(2/3) k \bm{I}$ in the case $(\x,t)=(\tilde\x,\tilde t)$. 
In view of the stochastic integral representation~\eqref{eq:v_integral_representation_1}, the spectral decomposition of $\u'$
is given by
\begin{align}\label{eq:u'_integral_representation_hom_1}
\u'(\x,t) & = \su \:\Re \!\int_{\R^3\times\R} \exp\Bigl\{ i\frac{1}{\delta}\Bigl[ \frac{1}{\sx} \kap\cdot(\x - t\overline{\u}) + \frac{1}{\st} \omega t \Bigr]\Bigr\} \fxs\bigl(\nkap;\sz z\bigr)  \fts(\omega) \Pk \cdot \wn(\d\kap,\d\omega),
\end{align}
where $\wn$ is a $\C^3$-valued Gaussian white noise on $\R^3 \times \R$ with structural measure $2 \lambda^3 \otimes \lambda^1$.
The spectral decomposition can be rewritten in terms of spherical coordinates in analogy to \eqref{eq:v_integral_representation_2}.

\begin{remark}
The use of constant scaling functions seems superfluous. But, with respect to our goal of considering inhomogeneous turbulence, it is important to be able to treat different homogeneous situations with the same reference values. However, in case of considering a single homogeneous flow field, the constant flow parameters are naturally chosen as reference values for the non-dimensionalization, i.e., $\nu_0=\nu^\star$, $k_0=k^\star$, $\eps_0=\eps^\star$, as done in Section~\ref{sec:hom}. In this case $\sx = \st = \su = \sz = 1$ and likewise $k=\eps=\nu=1$ hold, and the characteristic dimensionless numbers are $z=\zeta= {\eps^\star \nu^\star}/{(k^\star)^2}$ and $\delta = {\sqrt{k^\star}^3}/{(\eps^\star x_0)}$.
\end{remark}

The explicit separation of micro and macro scale specified by the ratio $\delta$ is not only pivotal for the justification and analysis of our inhomogeneous model, but also allows for a natural and practical description of ergodicity properties of homogeneous turbulence. 
The following result establishes the approximation of characteristic values of the fluctuation field -- interpreted as expectations \wrt~the underlying probability distribution as in \eqref{eq:kinetic_evergy_hom}, \eqref{eq:dissipation_hom} -- by means of averages in time and space. 
Thanks to the two-scale approach, the averages are not required to be taken \wrt asymptotically infinite domains but may be local averages, given that the turbulence scale ratio $\delta$ tends to zero. 
This aspect will be crucial for the derivation of similar ergodicity results in the inhomogeneous case.

\begin{proposition}[Ergodicity]\label{prop:ergodicity_hom}
Let $\u' = (\u'(\x,t))_{(\x,t) \in \R^3 \times \R}$ be a homogeneous turbulence field in the sense of Model~\ref{model:turb_hom*}, given in the dimensionless form \eqref{eq:ansatz_u'_hom}, and let $L,L' \in [0,\infty]$ be such that $L = \infty \,\vee \, L' = \infty$. 
Then, considering radii of integration $R,R' \in \R_0^+$ and the turbulence scale ratio $\delta \in \R^+$, for every $(\x,t)\in\R^3\times\R$ we have the following convergences in the mean-square sense as $(R/\delta,R'/\delta)\to(L,L')$ in $\overline{\R}\times\overline{\R}$:
\begin{align}
& \mint{-}_{B_R^{(1)}(t)}\mint{-}_{B_{R'}^{(3)}(\x+(s-t)\overline\u)} \u'(\y,s) \,\d\y\,\d s 
\;\; \xrightarrow{\;L^2\;}\;\; \mathbf 0 
\label{eq:ergodicity_prop_hom_A}\\
\frac12\,& \mint{-}_{B_R^{(1)}(t)}\mint{-}_{B_{R'}^{(3)}(\x+(s-t)\overline\u)} \bigl\| \u'(\y,s) \bigr\|^2 \,\d\y\,\d s 
\;\; \xrightarrow{\;L^2\;}\;\; k 
\label{eq:ergodicity_prop_hom_B}\\
\frac12 \, \delta^2z\, & \mint{-}_{B_R^{(1)}(t)}\mint{-}_{B_{R'}^{(3)}(\x+(s-t)\overline\u)} \bigl\| \nabla_\y\u'(\y,s)+ \bigl( \nabla_\y\u'(\y,s)\bigr){\vphantom{)}}^{\!\top} \bigr\|^2 \,\d\y\,\d s 
\;\; \xrightarrow{\;L^2\;}\;\; \frac{\eps}{\nu} 
\label{eq:ergodicity_prop_hom_C}
\end{align}
Here 
the notation 
$\mint{-}$ indicates an 
average value, i.e., 
$\mint{-}_{B_r^{(n)}\!(\y_0)} g(\y) \,\d \y$
is defined as  
$(\lambda^n(B_r^{(n)}(\y_0)))^{-1}\linebreak \int_{B_r^{(n)}\!(\y_0)} g(\y) \, \d \y$ if $r\in\R^+$ 
and as 
$g(\y_0)$ if $r=0$. 
The involved integrals are understood as Bochner integrals with values in $L^2(\Pr;\R^3)$ or $L^2(\Pr;\R)$. 
\end{proposition}

If the random field 
$\u'$ admits a modification with sufficiently regular sample paths,
then the integrals in Proposition~\ref{prop:ergodicity_hom} can be interpreted as pathwise integrals 
which coincide with the abstract Bochner integrals $\Pr$-almost surely, cf.~Remark~\ref{rem:sample_path_reg}. 
For the sake of clarity, we 
further note that the convergence assertions in Proposition \ref{prop:ergodicity_hom} are interpreted as follows: For every open neighborhood $V$ of the constant random variable on the right-hand side of \eqref{eq:ergodicity_prop_hom_A}, \eqref{eq:ergodicity_prop_hom_B}, or \eqref{eq:ergodicity_prop_hom_C} in $L^2(\Pr;\R^3)$ or $L^2(\Pr;\R)$, there exists an open neighborhood $U$ of $(L,L')$ in $\overline{\R}\times\overline{\R}$ such that the random variable given by the left-hand side of \eqref{eq:ergodicity_prop_hom_A}, \eqref{eq:ergodicity_prop_hom_B}, or \eqref{eq:ergodicity_prop_hom_C} lies in $V$ whenever $(R/\delta, R'/\delta)$ lies in $U\cap(\R\times\R)$. Here $\overline{\R}\times\overline{\R}$ is endowed with the product topology induced by the standard topology on $\overline{\R}$.

\begin{proof}[Proof of Proposition \ref{prop:ergodicity_hom}]
We first focus on the assertions \eqref{eq:ergodicity_prop_hom_A} and \eqref{eq:ergodicity_prop_hom_B}. 
For the purpose of adopting a unified perspective, we represent the average integrals in the form $\mint{-}\mint{-} \phi(\u'(\y,s)) \,\d\y\,\d s$  with suitably chosen functions $\phi \colon \R^3 \to \R$, taking $\phi(\u) = u_l$ with $l \in \{1,2,3\}$ in case \eqref{eq:ergodicity_prop_hom_A} and $\phi(\u) = \|\u\|^2/2$ in case $\eqref{eq:ergodicity_prop_hom_B}$, $\u = ( u_1,u_2,u_3) \in \R^3$. The main idea of the proof is to apply the ergodicity result from Proposition \ref{prop:mean_ergo_thm} in Appendix \ref{sec:ergodicity} to the random field $\w = ( \w(\x,t) )_{(\x,t) \in \R^3 \times \R}$ defined by
\begin{equation}\label{eq:ergo_hom_def_w}
\w(\x,t) = \su \v \Bigl( \frac{1}{\sx} \x, \frac{1}{\st} t; \sz z \Bigr)
\end{equation}
and to choose the functional $\varphi$ in Proposition~\ref{prop:mean_ergo_thm} as the mapping that assigns to each function $\bom\colon\R^3 \times \R\to\R^3$ the value $\varphi(\bom):=\phi(\bom(\bm0,0))$. 
Observe that the assumptions in Model \ref{model:turb_hom*} ensure that $\w$ satisfies the conditions in Proposition \ref{prop:mean_ergo_thm}, implying in particular that the mapping $\R^3 \times \R \ni (\y,s) \mapsto \phi(\w(\y,s)) \in L^2(\Pr;\R)$ is continuous whenever 
$\phi \in L^2(\Pr_{\w(\bm{0},0)};\R)$, where $\Pr_{\w(\bm{0},0)}$ denotes the probability distribution of the random vector $\w(\bm{0},0)$. 
This and the identity $\u'(\y,s) = \w((\y-s\overline{\u})/\delta, s/\delta)$ yield that all appearing Bochner integrals are defined. Next fix $(\x,t) \in \R^3 \times \R$ and note that the change of variables formula establishes that
\begin{align*}
\mint{-}_{B_{R}^{(1)}(t)} &\mint{-}_{B_{R'}^{(3)}(\x + (s-t)\overline{\u})} \phi\bigl( \u'(\y,s) \bigr) \,\d \y \, \d s
\\ &  = \mint{-}_{B_{R/\delta}^{(1)}(0)} \mint{-}_{B_{R'/\delta}^{(3)}(\bm{0})} \phi\bigl( \u'(\x + \delta\y + \delta s \overline{\u}, t+\delta s) \bigr) \,\d \y \, \d s
\\ & = \mint{-}_{B_{R/\delta}^{(1)}(0)} \mint{-}_{B_{R'/\delta}^{(3)}(\bm{0})} \phi\Bigl( \w\Bigl( \frac{1}{\delta}(\x - t \overline{\u}) + \y, \frac{1}{\delta} t + s\Bigr) \Bigr) \,\d \y \, \d s.
\end{align*}
Considering arbitrary sequences $(R_j)_{j\in\N}, (R_j')_{j\in\N} \subset \R_0^+$, $(\delta_j)_{j\in\N}\subset \R^+$ such that $\lim_{j\to\infty} R_j/\delta_j = L$ and $\lim_{j\to\infty} R'_j/\delta_j = L'$, the convergences \eqref{eq:ergodicity_prop_hom_A} and \eqref{eq:ergodicity_prop_hom_B} thus follow by applying Proposition\ref{prop:mean_ergo_thm} (with $\mathcal{R}_j = R_j/\delta_j$, $\mathcal{R}'_j = R'_j/\delta_j$, $\x_j = (\x-t\,\overline{\u})/\delta_j$, and $t_j = t/\delta_j$ in the notation of Proposition~\ref{prop:mean_ergo_thm}) and taking into account that $\E[\w(\bm{0},0)] = \bm{0}$ and $\E[ \|\w(\bm{0},0) \|^2 ]/2 = k$.

The strategy of the proof of \eqref{eq:ergodicity_prop_hom_A} and \eqref{eq:ergodicity_prop_hom_B} can be adapted in order to verify assertion \eqref{eq:ergodicity_prop_hom_C}. Replacing the definition of the random field $\w = ( \w(\y,s) )_{(\y,s) \in \R^3 \times \R}$ in \eqref{eq:ergo_hom_def_w} by $\w(\y,s) = (\su/\sx) ( \nabla_\y \v)(\y/\sx, s/\st ;\sz z)$, employing the identity $\delta \nabla_\y \u'(\y,s) = \w((\y-s \overline{\u})/\delta, s/\delta)$, and using a suitable identification mapping $\bm{\pi} \colon \R^{3\times 3} \to \R^9$, an application of Proposition \ref{prop:mean_ergo_thm} with $\phi \colon \R^9 \to \R$, $\phi(\u) = \| \bm{\pi}^{-1}(\u) + (\bm{\pi}^{-1}(\u))^\top\|^2$, establishes \eqref{eq:ergodicity_prop_hom_C} by reasoning along the lines of the first part of the proof.
\end{proof}

\subsection{Alternative representation formulas for the homogeneous model}\label{subsec:hom_alternative}

The spectral decomposition formula \eqref{eq:u'_integral_representation_hom_1} above represents the homogeneous turbulence field in terms of stochastic Fourier-type integrals \wrt an underlying Gaussian white noise.  
It might be tempting to set up an inhomogeneous model capturing spatio-temporal variations of the characteristic flow quantities by simply allowing the scaling factors $\sx$, $\st$, $\su$, $\sz$ in \eqref{eq:u'_integral_representation_hom_1} to vary in $(\x,t)$ and by directly replacing the linear transport term $x-t\overline\u$ in \eqref{eq:u'_integral_representation_hom_1} with a nonlinear counterpart. 
However, this is not a viable approach as it leads to two major problems. 
First, the presence of the scaling factors $\sx$, $\st$ in the argument of the complex exponential function in \eqref{eq:u'_integral_representation_hom_1} would lead to a loss of invariance of the model \wrt translations of the underlying spatio-temporal coordinate system if $(\x,t)$-dependence of these factors was permitted.  
Second, substituting transport by the mean flow in \eqref{eq:u'_integral_representation_hom_1} with an inhomogeneous analogue would typically distort the spatial oscillations specified by the wave vectors $\kap$ in a significant way, especially if the inhomogeneous flow was considered over large time intervals. In order to avoid these pitfalls, we derive suitable alternative stochastic integral representations for the homogeneous model that allow for a proper transition to the inhomogeneous case. 
We proceed in three steps in Lemmata~\ref{lem:transformation1}--\ref{lem:transformation3} below, each of which involves a specific transformation and employs the auxiliary results on integration \wrt white noise presented in Appendix~\ref{app:white_noise}.

The first step consists in a coordinate transformation and ensures that the scaling factors $\sx$, $\st$ no longer appear in the argument of the complex exponential function but are shifted to the Fourier coefficient. 
This addresses the first problem mentioned above and contributes to maintaining basic invariance properties after progressing to the inhomogeneous case. 
The parameter $\delta$ is included in the transformation regarding $\st$ in order to simplify the subsequent steps. 
We recall that $\wn$ in \eqref{eq:u'_integral_representation_hom_1} is a $\C^3$-valued Gaussian white noise on $\R^3 \times \R$ with structural measure $2 \lambda^3 \otimes \lambda^1$ 
and employ the notation  $\B_0(\R^3 \times \R)$ for the class of Borel sets 
in $\R^3 \times \R$ 
of finite Lebesgue measure. 
\wen

\begin{lemma}\label{lem:transformation1}
Let $\u' = (\u'(\x,t))_{(\x,t) \in \R^3 \times \R}$ be a homogeneous turbulence field in the sense of Model~\ref{model:turb_hom*}, given in dimensionless form by means of the stochastic spectral representation \eqref{eq:u'_integral_representation_hom_1}. 
Then for all $(\x,t)\in\R^3\times \R$ it holds $\Pr$-almost surely that
\begin{equation}\label{eq:alternative_rep_hom_1}
\begin{aligned}
\u'(\x,t) = \su\:\Re\! &  \int_{\R^3\times\R} \exp\Bigl\{ i \Bigl[\frac{1}{\delta}\, \kap\cdot(\x - t\overline{\u}) +  \omega t \Bigr]\Bigr\} 
\\ &  \vphantom{\int}
\, \sx^{3/2} \;\! \fxs\bigl(\sx\nkap; \sz z\bigr) \;\! (\delta\st)^{1/2} \fts\bigl(\delta\st\,\omega\bigr) \;\! \Pk \cdot \wt{\wn}(\d\kap,\d\omega).
\end{aligned}
\end{equation}
Here $\wt{\wn}$ is the $\C^3$-valued Gaussian white noise on $\R^3 \times \R$ with structural measure $2 \lambda^3 \otimes \lambda^1$ defined via transformation of $\wn$ by
\begin{align*}
\wt{\wn}(B) = 
\Bigl(\frac{1}{\sx^{3} \,\delta\st} \Bigr)^{1/2}
\wn \bigl(\bm{\phi}^{-1}(B)\bigr),\qquad \bm{\phi}(\kap,\omega) = \Bigl(\frac{1}{\sx} \kap, \frac{1}{\delta\st} \omega\Bigr),
\end{align*}
$B \in\B_0(\R^3 \times \R)$, 
with coordinate transformation mapping $\bm{\phi} \colon \R^3 \times \R \to \R^3 \times \R$.  
\end{lemma}

\begin{proof}
In view of the fact that $\bm{\phi}$ is a $C^1$-diffeomorphism with constant Jacobian determinant $\det(D\bm\phi)(\x,t)=1/(\sx^3\;\!\delta\st)$, the assertion is an immediate consequence of the change of variables formula for white noise integrals in Proposition \ref{prop:pushforward} b). 
\end{proof}

The second step translates integration in Fourier space \wrt the temporal frequency $\omega$ into integration on the time axis \wrt an auxiliary time variable $s$. This transformation addresses the second major problem mentioned above as it paves the way for localizing in time the effect of advection by the mean flow. Interpreting the integral in \eqref{eq:alternative_rep_hom_1} as the inverse Fourier transform of the product of a regular deterministic function and the generalized random field $\widetilde\wn$, the employed transformation reflects the well-known convolution theorem for Fourier transforms. 

\begin{lemma}\label{lem:transformation2}
Let $\u' = (\u'(\x,t))_{(\x,t) \in \R^3 \times \R}$ be a homogeneous turbulence field in the sense of Model~\ref{model:turb_hom*}, given in dimensionless form by means of the transformed representation \eqref{eq:alternative_rep_hom_1}. 
Then for all $(\x,t)\in\R^3\times \R$ it holds $\Pr$-almost surely that
\begin{equation}\label{eq:alternative_rep_hom_2}
\begin{aligned}
\u'(\x,t)  = \frac{\su}{(2\pi)^{1/2}} \,\Re\! \int_{\R^3\times\R}  \Bigl( & \frac1{\delta\st} \Bigr)^{1/2} \eta \Bigl(\frac{1}{\delta\st}(t-s)\Bigr) \;\!
\exp\Bigl\{ i \frac{1}{\delta}\, \kap\cdot(\x - t \overline{\u}) \Bigr\} 
\\ & \vphantom{\int} 
\sx^{3/2} \;\! \fxs \bigl(\sx\nkap;\sz z\bigr) \;\! \Pk \cdot \bigl[(\mathcal{I} \otimes \mathcal{F}^{-1})(\wt{\wn})\bigr](\d\kap,\d s).
\end{aligned}
\end{equation} 
Here $\eta=(2\pi)^{-1/2}\mathcal F^{-1}(\fts)$ is defined as in \eqref{eq:eta},  
and $(\mathcal{I} \otimes \mathcal{F}^{-1})(\wt{\wn})$ is the $\C^3$-valued Gaussian white noise on $\R^3 \times \R$ with structural measure $4\pi \lambda^3 \otimes \lambda^1$ 
specified 
via transformation of $\wt\wn$ by
\begin{align*}
\bigl[(\mathcal{I} \otimes \mathcal{F}^{-1})(\wt{\wn})\bigr](B) = \int_{\R^3\times\R} \bigl[(\mathcal{I} \otimes \mathcal{F}^{-1})(\indicator{B})\bigr](\kap,s) \,\wt{\wn}(\d\kap,\d s), 
\end{align*}
$B \in\B_0(\R^3 \times \R)$, with partial inverse Fourier transform $\mathcal{I} \otimes\mathcal{F}^{-1} \colon L^2(\lambda^3 \otimes \lambda^1;\C) \to L^2(\lambda^3 \otimes \lambda^1;\C)$ given as the tensor product of the identity mapping $\mathcal{I}\colon L^2(\lambda^3;\C) \to L^2(\lambda^3;\C)$ and the inverse Fourier transform $\mathcal{F}^{-1}\colon L^2(\lambda^1;\C) \to L^2(\lambda^1;\C)$.
\end{lemma}

\begin{proof}
Let us first remark that the definition of $(\mathcal{I} \otimes \mathcal{F}^{-1})(\wt{\wn})$ above corresponds to the application of the adjoint operator $(\mathcal{I} \otimes \mathcal{F}^{-1})'$ to the generalized random field $\wt{\wn}$, cf.~Appendix~\ref{subsec:appendix_transformation_results}.   
In this sense it is consistent with the usual definition of Fourier transforms in the theory of  generalized functions. 
Moreover, note that $(2\pi)^{-1/2}\mathcal{I} \otimes \mathcal{F}^{-1}$ is an isometric operator on $ L^2(\lambda^3 \otimes \lambda^1;\C)$ with inverse $(2\pi)^{1/2}\mathcal{I} \otimes \mathcal{F}$. 
The fact that $\wt\wn$ is a Gaussian white noise on $\R^3\times\R$ with structural measure $2\lambda^3\otimes\lambda^1$ and an application of Proposition~\ref{prop:linear_isometry_on_white_noise_integral} therefore imply that $(\mathcal{I} \otimes \mathcal{F}^{-1})(\wt{\wn})$ is a Gaussian white noise on $\R^3\times\R$ with structural measure $4\pi\lambda^3\otimes\lambda^1$. 
Employing this and the identity 
$(\mathcal{I} \otimes \mathcal{F})((\mathcal{I} \otimes \mathcal{F}^{-1})(\wt{\wn}))= \wt{\wn}$ 
in combination with the representation formula \eqref{eq:alternative_rep_hom_1} and a further application of Proposition~\ref{prop:linear_isometry_on_white_noise_integral}, we obtain that
\begin{align*}
\u'(\x,t)  = \su \:\Re \!\int_{\R^3\times\R}  \mathcal{F} \Bigl( & \exp\bigl\{ i (\bdot) t\bigr\} \, (\delta\st)^{1/2} \fts(\delta\st \bdot) \Bigr)(s) \;\!
\exp\Bigl\{ i \frac{1}{\delta} \kap\cdot(\x - t\overline{\u}) \Bigr\}  
\\ & \vphantom{\int}
\quad \sx^{3/2} \;\! \fxs \bigl(\sx\nkap; \sz z\bigr) \;\! \Pk \cdot \bigl[(\mathcal{I} \otimes\mathcal{F}^{-1})(\wt{\wn})\bigr](\d\kap,\d s).
\end{align*}
The identity in \eqref{eq:alternative_rep_hom_2} then follows from 
the observation that
\begin{align*}
\mathcal{F} \Bigl(  \exp\bigl\{ i (\bdot) t\bigr\} \, (\delta\st)^{1/2} \fts(\delta\st \bdot) \Bigr)(s) & = \lim_{R\to\infty} \frac{1}{{2\pi}} \int_{-R}^R \exp\bigl\{ i \omega ( t - s) \bigr\} (\delta\st)^{1/2} \fts(\delta\st \omega)\, \d \omega
\\ & = \frac{1}{{2\pi}}\Bigl(\frac{1}{\delta\st}\Bigr)^{1/2} \lim_{R\to \infty} \int_{-R}^R \exp\Bigl\{ i \omega \frac{1}{\delta\st}(t-s)\Bigr\} \fts( \omega) \,\d \omega
\\ & = \frac{1}{{2\pi}}\Bigl(\frac{1}{\delta\st}\Bigr)^{1/2} \Bigl[  \mathcal{F}^{-1}\bigl(\fts\,\bigr)\Bigr] \Bigl(\frac{1}{\delta\st}(t-s)\Bigr),
\end{align*}
where the limit is taken in $L^2(\d s;\C)$. 
\end{proof}
 
In the third step we synchronize the time variables appearing in \eqref{eq:alternative_rep_hom_2} in the argument of the time integration kernel  $\eta$ and in the traveling plane wave $\exp\{i\delta^{-1}\kap\cdot(\x-t\overline\u)\}$. 
In view of the fact that the function $\eta$ typically attains its maximal absolute value at zero and is rapidly decaying, this transformation establishes the localization in time prepared in the previous step.  We refer to Figure~\ref{fig:mean_flow} and Remark~\ref{rem:model_turb_inhom} below for a discussion of the localization aspect in the inhomogeneous context.   
To simplify the resulting representation formula, 
we 
additionally substitute 
the transformed white noise $(2\pi)^{-1/2}(\mathcal{I} \otimes \mathcal{F}^{-1})(\wt{\wn})$ 
involved in \eqref{eq:alternative_rep_hom_2} with the original white noise $\wn$ sharing the same structural 
measure.

\begin{lemma}\label{lem:transformation3}
Let $\u' = (\u'(\x,t))_{(\x,t) \in \R^3 \times \R}$ be a homogeneous turbulence field in the sense of Model~\ref{model:turb_hom*}, given in dimensionless form by means of the transformed representation 
\eqref{eq:alternative_rep_hom_2}, and let $\widetilde{\u}' = (\widetilde{\u}'(\x,t))_{(\x,t) \in \R^3 \times \R}$ be the 
random field given by 
\begin{equation}\label{eq:u'_integral_representation_hom_5}
\begin{aligned}
\widetilde{\u}'(\x,t) = \su\: \Re\!  \int_{\R^3\times\R} \Bigl(\frac1{\delta\st}\Bigr)^{1/2} \eta & \Bigl( \frac{1}{\delta\st}(t-s)\Bigr) 
\exp\Bigl\{ i \frac{1}{\delta}\, \kap\cdot\bigl(\x - (t-s)\overline{\u}\bigr) \Bigr\} 
\\ & \vphantom{\int}
\quad 
\sx^{3/2} \;\! \fxs \bigl(\sx\nkap; \sz z\bigr) \;\! \Pk \cdot \wn(\d\kap,\d s).
\end{aligned}
\end{equation}
Then $\u'$ and $\widetilde{\u}'$ have the same finite-dimensional distributions. 
\end{lemma}

\begin{proof}
Observe that the transformed 
white noise $(2\pi)^{-1/2}(\mathcal{I} \otimes \mathcal{F}^{-1})(\wt{\wn})$ 
in \eqref{eq:alternative_rep_hom_2} and the original $\wn$ 
are both $\C^3$-valued Gaussian white noises with structural measure $2\lambda^3\otimes\lambda^1$. 
Moreover, 
if we identify $(2\pi)^{-1/2}(\mathcal{I} \otimes \mathcal{F}^{-1})(\wt{\wn})$  in \eqref{eq:alternative_rep_hom_2} with $\wn$ in \eqref{eq:u'_integral_representation_hom_5}, 
then the right-hand side of \eqref{eq:u'_integral_representation_hom_5} differs from 
the right-hand side of \eqref{eq:alternative_rep_hom_2} only in the argument of the exponential function, 
resulting in an additional factor $\exp\{ i \delta^{-1} \kap \cdot s\overline{\u} \}$ in the integrand. 
As this factor does not depend on $(\x,t)$ and has absolute value one, 
the covariance formula in Lemma~\ref{lem:covmain} ensures that 
the random fields $\u'$ and $\widetilde{\u}'$ 
have 
the same covariance structure. 
Since both random fields are centered and Gaussian, this implies the claim. 
\end{proof}

The results above reveal that the representation formulas \eqref{eq:u'_integral_representation_hom_1} and \eqref{eq:u'_integral_representation_hom_5} are equivalent in the sense 
that the distributions of the random fields $\u'$ and $\widetilde{\u}'$ coincide. 
In what follows, the representation \eqref{eq:u'_integral_representation_hom_5} plays a key role as its specific structure permits the inclusion of spatio-temporal variations of the characteristic flow quantities.


\section{Inhomogeneous turbulence}\label{sec:inhom}

Building on the foregoing preparations, in this section we present our random field model for the reconstruction of inhomogeneous turbulence. The model is introduced and discussed in Subsection~\ref{subsec:inhom_model}, 
while 
Subsection~\ref{subsec:inhom_analysis} deals with its mathematical analysis with regard to characteristic turbulent flow properties and spatio-temporal ergodicity. 
The
respective results are given in Theorem~\ref{thm:kin_diss_div} and Theorem~\ref{thm:ergodicity_inhom} below.

\subsection{Random field modeling of inhomogeneous turbulence}\label{subsec:inhom_model}

The starting point for our inhomogeneous turbulence model is the representation formula \eqref{eq:u'_integral_representation_hom_5} for the homogeneous and spatially isotropic model above. 
Its specific form allows for the following extensions. 
\smallskip \\
\noindent
\textbf{Inhomogeneous scaling:} 
In order to account for the macroscopic variability of the 
turbulence scales and the turbulence Reynolds number, 
we replace the constant scaling factors $\sx$, $\st$, $\su$, and $\sz$ in \eqref{eq:u'_integral_representation_hom_5} by the respective $(\x,t)$-dependent flow quantities $\sx(\x,t)$, $\st(\x,t)$, $\su(\x,t)$, and $\sz(\x,t)$. This guarantees an adequate  $(\x,t)$-dependent scaling of the spatial energy spectrum and the temporal correlations.
\smallskip \\
\noindent
\textbf{Non-uniform advection:} 
To capture possibly nonlinear advection by the mean flow from an Eulerian perspective, we further replace the linear transport term
$\x - (t-s)\overline{\u} = \x + (s-t) \overline{\u}$ 
appearing in \eqref{eq:u'_integral_representation_hom_5} by the solution $\flow(s;\x,t)$ to the integral equation
\begin{equation}\label{eq:integral_equation_xi}
\flow(s;\x,t) = \x + \int_t^s\overline{\u} \bigl( \flow(r;\x,t), r \bigr) \,\d r, \quad s \in \R,
\end{equation}
describing for every $(\x,t) \in \R^3 \times \R$  a trajectory with value $\x$ at time $t$ that evolves along the mean velocity field $\overline{\u}$. 
We remark that the transformation $\x\mapsto \flow(s;\x,t)$ is the inverse mapping of  $\x\mapsto \flow(t;\x,s)$ and refer to the illustration in Figure~\ref{fig:mean_flow} below.
\smallskip \\
\noindent
\textbf{Directional weightings:} 
While the extensions above constitute the core of our approach to inhomogeneous turbulence, 
we also allow for 
anisotropic one-point velocity correlations by 
incorporating a linear transformation on the noise vector $\wn(\d\kap,\d s)$ in \eqref{eq:u'_integral_representation_hom_5} 
in terms of 
a matrix $\bm{L}(\x,t)$ with $\|\bm{L}(\x,t)\|^2=3$. This 
anisotropy factor 
provides control of directional weightings and covers the isotropic case for $\bm{L}(\x,t)=\bm I$ being the identity matrix.
\smallskip 

We summarize and specify our framework for random field modeling of inhomogeneous turbulence as follows. 
As discussed in Remark \ref{rem:model_turb_inhom} below, the two-scale perspective reflected in the presence of the turbulence scale ratio $\delta \ll 1$ is crucial for the justification and applicability of our model.

\begin{model}[Inhomogeneous turbulence field]\label{model:turb_inhom}
Let $\u' = (\u'(\x,t))_{(\x,t) \in \R^3 \times \R}$ be a centered, $\R^3$-valued Gaussian random field of the form
\begin{equation}\label{eq:u'_inhom}
\begin{aligned}
\u'(\x,t) = \su(\x,t)\:\Re\!  \int_{\R^3\times\R}&  \Bigl(\frac1{\delta\st(\x,t)}\Bigr)^{1/2}\eta\Bigl(\frac{1}{\delta\st(\x,t)}(t-s)\Bigr) 
\exp\Bigl\{ i \frac{1}{\delta}\, \kap\cdot \flow(s;\x,t) \Bigr\} 
\\ & \vphantom{\int} \sx^{3/2}(\x,t) \;\! \fxs \bigl(\sx(\x,t)\nkap; \sz(\x,t)z\bigr) \;\!  
\Pk \cdot \bm{L}(\x,t) \cdot \wn(\d\kap,\d s),
\end{aligned}
\end{equation}
where the characteristic numbers $\delta$, $z \in \R^+$ and the scaling functions $\sx$, $\st$, $\su$, $\sz \colon \R^3 \times \R \to \R^+$ are given by \eqref{eq:z_delta} and \eqref{eq:sigma_x_t_u_z}, and the following is assumed:
\begin{enumerate}[label= \emph{\textbf{\alph*)}}]
\item The function 
$\R^+\times\R^+ \ni (\kappa,\zeta) \mapsto \fx(\kappa;\zeta) \in \R_0^+$ 
is  continuously differentiable and the associated energy spectrum $E(\kappa;\zeta) = 4 \pi \kappa^2 \fx(\kappa;\zeta)$ fulfills the integral conditions \eqref{eq:energy_spectrum_1}, i.e.,
\begin{align}\label{eq:basic_integral_prop_fx}
\int_{\R^3} \fx(\nkap;\zeta) \,\d\kap = 1, \quad \int_{\R^3} \nkap^2 \fx(\nkap;\zeta) \, \d\kap = \frac{1}{2 \zeta}.
\end{align}
\item The time integration kernel $\eta \colon \R \to \R$ is continuously differentiable and satisfies
\begin{align}\label{eq:basic_integral_prop_eta}
\int_\R \eta^2(s) \,\d s = 1.
\end{align}
\item For every $(\x,t) \in \R^3 \times \R$ the mean flow function $\flow(\bdot;\x,t) \colon \R \to \R^3$ is continuous and solves the integral equation \eqref{eq:integral_equation_xi}.
\vspace{1mm}
\item The flow quantities $k,\eps,\nu \colon \R^3 \times \R \to \R^+$ and $\overline{\u}\colon \R^3 \times \R \to \R^3$ -- introduced in \eqref{eq:scaling} and entering the definitions of $\sx$, $\st$, $\su$, $\sz$, and $\flow$ via \eqref{eq:sigma_x_t_u_z} and \eqref{eq:integral_equation_xi} -- are continuous in $(\x,t)$, differentiable in $\x$, and such that 
$\nabla_\x \overline{\u}$, $\nabla_\x k$, $\nabla_\x \eps$, $\nabla_\x \nu$ are continuous in $(\x,t)$. 
The mean velocity gradient $\nabla_\x \overline{\u} \colon \R^3 \times \R \to \R^{3 \times 3}$ is bounded. 
\vspace{1mm}
\item 
The anisotropy function  $\bm{L}\colon \R^3 \times \R \to \R^{3\times 3}$ is continuous in $(\x,t)$, 
differentiable in $\x$, and such that $\nabla_{\x}\bm{L}$ is continuous in $(\x,t)$. Moreover, it satisfies 
$\|\bm{L}(\x,t)\|^2=3$.
\vspace{1mm}
\item $\wn$ is a $\C^3$-valued Gaussian white noise on $\R^3 \times \R$ with structural measure $2 \lambda^3 \otimes \lambda^1$ in the sense of Definition \ref{def:gaussian_white_noise}.
\end{enumerate}
\vspace{1mm}
As before, $\Pk \in \R^{3 \times 3}$ in \eqref{eq:u'_inhom} denotes the projector onto the orthogonal complement of $\kap$. 
In addition to the assumptions above, we suppose that the following technical integrability conditions related to the spatial mean-square differentiability of $\u'$ are fulfilled:
\vspace{1mm}
\begin{enumerate}
\item[\emph{\textbf{g)}}] For every $(\x,t) \in \R^3 \times \R$ there exists 
a ball $B_r(\x,t)$ of radius $r=r(\x,t)\in \R^+$
such that
\begin{equation}\label{eq:assumption_eta}
\int_\R s^2 \etad^2(s)\, \d s < \infty, \qquad  \int_{\R} s^2 \exp\Bigl\{ 2 \| \nabla_\x \overline{\u} \|_\infty \delta \st(\x,t) |s| \Bigr\} \eta^2(s) \,\d s < \infty,
\end{equation}
and
\begin{equation}\label{eq:assumption_deriv_fxsbb}
\int_{\R^3} \sup_{(\y,s) \in B_r(\x,t)} \Bigl\| \nabla_{\y} \Bigl( \sx^{3/2}(\y,s) \fxs\bigl(\sx(\y,s) \nkap; \sz(\y,s)z \bigr) \Bigr) \Bigr\|^2 \d \kap < \infty.
\end{equation}
\end{enumerate}
\vspace{.5mm}
In the described situation, we refer to $\u'$ as an \emph{inhomogeneous turbulence field}.
\end{model}
\vspace{.5mm}

Note that an application of the change of variables formula 
for spherical coordinates 
in Corollary~\ref{cor:polar} to the stochastic integral \eqref{eq:u'_inhom} in Model~\ref{model:turb_inhom} 
implies the alternative representation \eqref{eq:u'_inhom_intro} 
given in 
the introductory section. 
Just as in 
the homogeneous case \eqref{eq:v_integral_representation_2}, 
the underlying white noise 
$\wn(\d\kappa,\d\btheta,\d\omega)$ 
in \eqref{eq:u'_inhom_intro} 
is obtained by transformation of 
$\wn(\d\kap,\d\omega)$ 
in \eqref{eq:u'_inhom}.  
Moreover, observe that 
the time integration kernel $\eta$ described in item b) in Model~\ref{model:turb_inhom} is not necessarily given in terms of a spectral density $\ft$ as in the homogeneous case~\eqref{eq:eta}. While the assumptions in the homogeneous model imply that $\eta$ in the representation formula \eqref{eq:u'_integral_representation_hom_5} is symmetric and satisfies $\int_\R \eta^2(s) \,\d s = 1$, only the latter property is imposed in Model~\ref{model:turb_inhom}. 

For the sake of presentation and to shorten the proofs we introduce the abbreviations 
\begin{equation}\label{eq:def_etabb_fxbb}
\begin{aligned}
\etabb(s;\x,t)  &= \Bigl(\frac1{\delta\st(\x,t)}\Bigr)^{1/2}\eta\Bigl(\frac{1}{\delta\st(\x,t)}s\Bigr),  \\
\fxsbb(\kap;\x,t)  &= \sx^{3/2}(\x,t)  \;\!  \fxs \bigl(\sx(\x,t) \nkap; \sz(\x,t)z \bigr),  \\
\Pbb(\kap;\x,t) &= \su(\x,t) \;\! \Pk \cdot \bm{L}(\x,t),
\end{aligned}
\end{equation}
so that \eqref{eq:u'_inhom} in Model \ref{model:turb_inhom} can be rewritten as
\begin{equation}\label{eq:u'_inhom_bb_eta_f}
\begin{aligned}
\u'(\x,t) & = 
\Re \!\int_{\R^3\times\R} \etabb(t-s;\x,t) \exp\Bigl\{ i \frac{1}{\delta}\, \kap\cdot \flow(s;\x,t) \Bigr\} 
\;\! \fxsbb(\kap;\x,t) \,\Pbb(\kap;\x,t) \cdot \wn(\d\kap,\d s).
\end{aligned}
\end{equation}
According to Corollary \ref{cor:covmain} the covariance function of $\u'$ is given by
\begin{equation}\label{eq:cov_inhom} 
\begin{aligned}
\E \bigl[ \u'(\x,t) \otimes \u'(\xalt,\talt\;\!) \bigr] 
 = \int_{\R^3\times\R} \etabb(t-s;\x,t) & \;\! \etabb(\talt - s;\xalt,\talt\;\!) 
\cos\Bigl\{ \frac{1}{\delta} \kap \cdot \bigl( \flow(s;\x,t)  - \flow(s;\xalt,\talt\;\!) \bigr)\Bigr\}\;\! \\ 
& \vphantom{\int}
\fxsbb(\kap;\x,t) \,\fxsbb(\kap;\xalt,\talt\;\!)\, \Pbb(\kap;\x,t)\cdot\Pbb(\kap;\xalt,\talt\;\!){\vphantom{)}}^{\!\top} \d(\kap,s).
\end{aligned}
\end{equation}
It can usually not be factorized into the product of a spatial and a temporal covariance function, in contrast to the homogeneous case \eqref{eq:cov_u'_hom}.

The specific structure of Model~\ref{model:turb_inhom} is further discussed in the following remark. 

\begin{figure}
\centerline{\includegraphics[trim={0cm 0cm 0cm .3cm}, clip, scale=1.6]{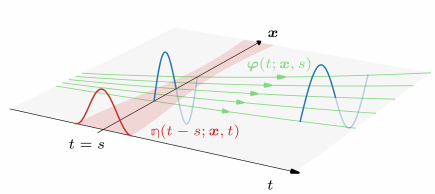}}
\label{fig:mean_flow}
\caption{Illustration of the advection of a plane wave of type $\x\mapsto \exp\{i \frac1\delta \kap\cdot \x\}$ \linebreak 
  (blue curve at $t=s$) by the mean flow (green) for a fixed value of the time integration variable $s$ appearing in \eqref{eq:u'_inhom}, \eqref{eq:u'_inhom_bb_eta_f}: Transport along the pathlines of the mean flow results in the transformed wave 
  $\x\mapsto \exp\{i \frac1\delta \kap\cdot \flow(s;\x,t)\}$, 
  where $\x\mapsto \flow(s;\x,t)$ represents the inverse of the mean flow transformation $\x\mapsto \flow(t;\x,s)$, and may alter the effective spatial frequency. The distortive effect is negligible due to the localizing influence of the time integration kernel $\bbeta(t-s;\x,s)$, 
 cf.~\eqref{eq:def_etabb_fxbb}. }
\end{figure}

\begin{remark}\label{rem:model_turb_inhom}
In order to elucidate the interplay between the single components of Model~\ref{model:turb_inhom}, it is useful to consider a Riemann sum approximation of the stochastic integral in \eqref{eq:u'_inhom} \wrt the time integration variable $s$. To this end, let $s_j = j \Delta s$, $j \in \mathbb{Z}$, be equidistant points with stepsize $\Delta s > 0$, let $\wn_j $ be the $\C^3$-valued Gaussian white noise on $\R^3$ defined by 
$\wn_j(B) = \wn( B \times (s_j - \frac{\Delta s}{2}, s_j + \frac{\Delta s}{2}])$, $B \in \B_0(\R^3)$, and set
\begin{align}\label{eq:u'_j}
\u'_j(\x,t) = 
\Re \!\int_{\R^3} \exp\Bigl\{ i \frac{1}{\delta} \kap \cdot \flow(s_j;\x,t) \Bigr\} \;\!\fxsbb(\kap;\x,t)  
\,\Pbb(\kappa;\x,t) \cdot \wn_j(\d\kap).
\end{align}
Assuming that $\eta^2(s)$ decays sufficiently fast as $|s| \to \infty$, we obtain the approximation
\begin{align}\label{eq:approximation_of_u'}
\u'(\x,t) \approx \sum_{j \in \mathbb{Z}} \etabb(t-s_j;\x,t) \;\!\u'_j(\x,t)
\end{align}
in the sense that the sum on the right-hand side is defined in $L^2(\Pr;\R^3)$ and as $\Delta s \to 0$ its value converges in $L^2(\Pr;\R^3)$ to $\u'(\x,t)$, compare \eqref{eq:u'_inhom_bb_eta_f}.  
We record the following observations:
\begin{enumerate}
\item[\emph{\textbf{i)}}]
In view of \eqref{eq:approximation_of_u'}, the inhomogeneous turbulence field $\u'$ can be interpreted as a weighted sum of independent random fields $\u'_j$, each of which essentially represents a frozen turbulence model that allows for varying flow quantities on the macro scale via the $(\x,t)$-dependence of $\fxsbb$ and $\Pbb$ in \eqref{eq:u'_j}. Transport by the mean flow is described from an Eulerian perspective via the mean flow function $\flow$. 
The fields $\u'_j$ are superposed by means of a convolution with the weight function $\etabb$. The latter includes macroscopically varying flow quantities too, see~\eqref{eq:def_etabb_fxbb}. 
The functions  $\eta$ and $\fx$ in \eqref{eq:def_etabb_fxbb} model the temporal decay and the spatial spectral properties from a micro scale perspective, while the factor $\bm{L}$ allows for directional weightings. 
\vspace{.5mm}
\item[\emph{\textbf{ii)}}] 
In contrast to the homogeneous case, the transformation $ \x\mapsto\flow(s_j;\x,t)$ in \eqref{eq:u'_j} typically has an unwanted distorting effect on the influence of each wave vector $\kap$, and thus on the spatial spectral properties of $\u'_j$ on the micro scale; 
compare Figure~\ref{fig:mean_flow}.
However, for fixed $j$ this distorting effect does not occur at time $t=s_j$, and it is negligible for time points $t$ sufficiently close to $s_j$ on the macro scale. 
In this regard, the weight function $\etabb$ in \eqref{eq:approximation_of_u'} ensures that the superposition of the fluctuations $\u'_j$ is properly localized in time, so that only those fields $\u'_j$ with $s_j$ close to the current time $t$ have a relevant influence on $\u'(x,t)$. This localization is particularly effective due to the presence of the turbulence scale ratio $\delta \ll 1$  in \eqref{eq:def_etabb_fxbb},  which reflects the assumption that the typical decay times of the turbulent structures -- modeled via the function $\eta$ on the micro scale -- are small compared to the macro scale.
\end{enumerate}
\end{remark}

In analogy to the homogeneous case, closing Model~\ref{model:turb_inhom} requires the specific choice of the energy spectrum $E(\kappa;\zeta)$ and the time integration kernel $\eta(s)$.

\begin{example}
\label{ex:energy_spectrum_correlation_function}
Consider the energy spectrum $E$ from Example ~\ref{ex:energy_spectrum} as well as the time integration kernel $\eta$ 
from Example~\ref{ex:temporal_cor}.
We assume that $\sz(\x,t)z<1$ for all $(\x,t)\in\R^3\times\R$, so that the fact that $E(\kappa;\zeta)$ in Example~\ref{ex:energy_spectrum} has been introduced as a function on $\R^+\times (0,\zeta_{\mathrm{crit}})$ instead of $\R^+\times\R^+$ is irrelevant. 
In this situation, the differentiability and integrability conditions in items a), b) and g) of Model~\ref{model:turb_inhom} are satisfied, given that the assumptions on the flow quantities in item d) are fulfilled. 

Indeed, observe that $\eta$ is continuously differentiable with $\int_\R \eta^2(s)\,\d s=1$ and trivially satisfies the integrability assumptions \eqref{eq:assumption_eta}.
For the verification of the continuous differentiability of $E$ we refer to Appendix \ref{sec:calc_estimates}, while the integral conditions in \eqref{eq:basic_integral_prop_fx} are fulfilled by construction of $E$. Regarding the integrability assumption $ \int_{\R^3} \sup_{(\y,s) \in B_r(\x,t)} \| \nabla_{\y} \fxsbb(\kap;\y,s) \|^2 \, \d \kap < \infty$ in \eqref{eq:assumption_deriv_fxsbb}, observe that for every $(\x,t) \in \R^3 \times \R$ the estimate \eqref{eq:estimate_property_fxsbb} in Appendix~\ref{sec:calc_estimates} implies the existence of a radius $r \in \R^+$ such that
\begin{align}\label{eq:estimate_grad_fxsbb_to_prove_in_app}
\sup_{(\y,s) \in B_r(\x,t)} \bigl\| \nabla_{\y} \fxsbb(\kap;\y,s) \bigr\|^2 & \leq\frac{1}{4\pi \nkap^2} 
\begin{cases}
C_1 \nkap^4, & \nkap < \kappa_1^-,
\\ C_2, & \kappa_1^- \leq \nkap \leq \kappa_2^+,
\\  C_3 \nkap^{-7}, & \nkap > \kappa_2^+,
\end{cases}
\end{align}
where $C_1$, $C_2$, $C_3$, $\kappa_1^-$, and $\kappa_2^+$ are non-negative constants depending only on $\x$, $t$, and $r$. By \eqref{eq:estimate_grad_fxsbb_to_prove_in_app} and spherical coordinate transformations we obtain
\begin{align*}
\int_{\R^3} \sup_{(\y,s) \in B_r(\x,t)} \bigl\| \nabla_{\y} \fxsbb(\kap;\y,s) \bigr\|^2 \d \kap \leq C_1
 \int_0^{\kappa_1^-} \kappa^4 \,\d \kappa + C_2\,(\kappa_2^+ - \kappa_1^-) + C_3 \int_{\kappa_2^+}^{\infty} \kappa^{-7} \,\d \kappa  < \infty,
\end{align*}
which proves \eqref{eq:assumption_deriv_fxsbb}.
\end{example}


\subsection{Analysis of the inhomogeneous model}\label{subsec:inhom_analysis}

In this subsection we show in Theorem~\ref{thm:kin_diss_div} below that the characteristic turbulent flow properties \eqref{eq:kinetic_evergy_hom} and \eqref{eq:dissipation_hom} as well as the incompressibility condition \eqref{eq:divergence_hom}, originally discussed in the homogeneous case, remain valid for our inhomogeneous turbulence model in an approximate sense, i.e.,
\begin{align*}
 \frac{1}{2}\,\E \Bigl[ \bigl\| \u'(\x,t) \bigr\|^2 \Bigr] & = k(\x,t) = \su^2(\x,t), \\ 
 \frac12 \,\delta^2 z  \, \E \Bigl[ \bigl\| \nabla_{\x} \u'(\x,t) + \bigl( \nabla_{\x}\u'(\x,t) \bigr){\vphantom{)}}^{\!\top} \bigr\|^2 \Bigr] & \approx \frac{\eps(\x,t)}{\nu(\x,t)} = \frac{\su^2(\x,t)}{\sx^2(\x,t) \sz(\x,t)},\\ 
\delta \, \nabla_{\x} \cdot\u'(\x,t) & \approx 0.
\end{align*}
The precise statements are given in the form of asymptotic results for the turbulence scale ratio $\delta \to 0$. 
As an extension of the first identity above we also verify a related assertion for the 
one-point velocity correlations (Reynolds stresses).
Moreover, in Theorem \ref{thm:ergodicity_inhom} below we present an extension of the ergodicity result from Proposition~\ref{prop:ergodicity_hom} to the case of inhomogeneous turbulence.

In preparation for our analysis we investigate the mean-square regularity of 
the inhomogeneous fluctuation field.

\begin{lemma}[Mean-square regularity]
\label{lem:ms_diff_u'}
Let $\u' = (\u'(\x,t))_{(\x,t) \in \R^3 \times \R}$ be an inhomogeneous turbulence field in the sense of Model \ref{model:turb_inhom}. Then $\u'$ is mean-square differentiable in $\x$, and for every $(\x,t) \in \R^3 \times \R$, $j \in \{1,2,3\}$ the mean-square partial derivative \wrt the $j$-th coordinate of $\x$ satisfies
\begin{align*}\label{eq:partial_u'}
\partial_{x_j} \u'(\x,t) = \Re \!\int_{\R^3 \times \R}  \partial_{x_j} \Bigl( \etabb(t-s;\x,t) \;\! \exp\Bigl\{ i \frac{1}{\delta}\, \kap\cdot\flow(s;\x,t) \Bigr\} \;\! \fxsbb(\kap;\x,t) \;\! \Pbb(\kap;\x,t)  \Bigr) \!\cdot \wn(\d\kap,\d s),
\end{align*}
where 
$\etabb(s;\x,t)$, $\fxsbb(\kap;\x,t)$, and $\Pbb(\kap;\x,t)$
are given by \eqref{eq:def_etabb_fxbb}. Moreover, both $\u'$ and $\partial_{x_j} \u'$ are mean-square continuous in $(\x,t)$.
\end{lemma}

\begin{proof}
Throughout this proof we assume \Wlog that 
$\su(\x,t) = 1$ and $\bm{L}(\x,t) = \bm{I}$, which 
is justified since products of partially differentiable deterministic functions 
and partially mean-square differentiable random fields 
are again partially mean-square differentiable 
and satisfy the usual product rule for derivatives. 
Moreover, we 
make use of 
the fact that the assumptions on $\flow$ and $\overline{\u}$ 
in Model~\ref{model:turb_inhom} 
and standard regularity results for ordinary differential equations 
ensure that the mapping $(s,\x,t) \mapsto \flow(s;\x,t)$ is 
continuous, continuously partially differentiable \wrt $\x$, and such that 
\begin{equation} \label{eq:partial_deriv_bxi}
\partial_{x_j}  \flow(s;\x,t) = \bm{e}_j + \int_t^s ( D_\x \overline{\u}) \bigl( \flow(r;\x,t), r \bigr)\cdot \partial_{x_j} \flow(r;\x,t) \,\d r;
\end{equation}
see, e.g., \cite[Lemma~4.8]{JLP19}. 
Here $(\bm{e}_j)_{j=1}^3$ denotes the standard basis in $\R^3$, 
and for every $(\y,r) \in \R^3 \times \R$ the notation $(D_\x \overline{\u})(\y,r)$ 
refers to the Jacobian matrix \wrt the spatial argument of $\overline{\u}$, evaluated at $(\y,r)$.

We first verify the mean-square continuity of $\u'$ in $(\x,t)$. 
To this end, consider the integrand 
\begin{equation}\label{eq:integrand_G}
\bm{G}(\kap,s;\x,t) :=  \etabb(t-s;\x,t) \exp\Bigl\{ i \frac{1}{\delta}\, \kap\cdot\flow(s;\x,t) \Bigr\}  \;\!\fxsbb(\kap;\x,t)  \Pk
\end{equation}
appearing in the white noise integral \eqref{eq:u'_inhom_bb_eta_f} defining $\u'$, 
and note that the isometry property 
\eqref{eq:isometric_property_real_part} 
of the stochastic integral 
implies that it is sufficient to verify the continuity of  $(\x,t) \mapsto \bm{G}(\kap,s;\x,t)$  in $L^2(\d\kap,\d s;\C^{3\times 3})$. 
This continuity, in turn, is guaranteed by the product structure of $\bm{G}$ in combination with \wen 
the boundedness and continuity of $(\kap,s;x,t) \mapsto \exp\{ i \delta^{-1}\kap\cdot \flow(s;\x,t) \}$, the continuity of $(\x,t) \mapsto \etabb(t-s;\x,t)$ in $L^2(\d s;\R)$, and the continuity of $(\x,t) \mapsto \fxsbb(\kap;\x,t)$ in $L^2(\d\kap;\R)$. 
The latter two 
properties follow from 
the continuity of $(\x,t)\mapsto\etabb(t-s;\x,t)$ and $(\x,t)\mapsto\fxsbb(\kap;\x,t)$ for fixed 
$s\in\R$, $\kap\in\R^3$ and 
the fact that $\| \etabb(t-s;\x,t) \|_{L^2(\d s;\R)} = \| \fxsbb(\kap;\x,t) \|_{L^2(\d\kap;\R)} = 1$ for all $(\x,t)\in\R^3\times\R$ 
according to \eqref{eq:basic_integral_prop_fx}, \eqref{eq:basic_integral_prop_eta}; 
compare, e.g., \cite[Theorem~12.10]{Sch05}. 

Next we demonstrate that $\x \mapsto \u'(\x,t)$ is partially mean-square differentiable 
and establish the formula for $\partial_{x_j} \u'(\x,t) $. 
By Lemma \ref{lem:ms-diff} 
in Appendix~\ref{app:white_noise} 
it is sufficient to 
show that for all $(\x,t) \in \R^3 \times \R$,  $j\in\{1,2,3\}$   
the difference quotient $(\bm{G}(\kap,s;\x + h\bm{e}_j,t) - \bm{G}(\kap,s;\x,t))/h$ converges to $\partial_{x_j} \bm{G}(\kap,s;\x,t)$ in $L^2(\d\kap,\d s;\C^{3\times 3})$ as $h \to 0$. 
Combining this, the product structure of $\bm{G}$, 
the boundedness of 
$(\kap,s) \mapsto \exp\{ i \delta^{-1} \kap\cdot\flow(s;\x,t) \}$ and $\kap \mapsto \bm{P}(\kap)$, 
and the continuity of $\x \mapsto \etabb(s-t;\x,t)$ in $L^2(\d s;\R)$ 
further 
reveals that it is sufficient to prove that
\wen 
\begin{align}\nonumber
& \lim_{h\to 0} \frac{1}{h}\Bigl( \exp\Bigl\{ i \frac{1}{\delta}\, \kap\cdot\flow(s;\x + h\bm{e}_j,t) \Bigr\} - \exp\Bigl\{ i \frac{1}{\delta}\, \kap\cdot\flow(s;\x,t) \Bigr\}\Bigr)\etabb(t-s;\x,t) \;\!\fxsbb(\kap;\x,t) 
\\ \label{eq:diffquotient_exp}
 & \qquad \qquad = \partial_{x_j}\Bigl(\exp\Bigl\{ i \frac{1}{\delta}\, \kap\cdot\flow(s;\x,t) \Bigr\} \Bigr) \etabb(t-s;\x,t) \;\! \fxsbb(\kap;\x,t) \qquad \text{in } L^2(\d\kap,\d s;\C),
\\ \label{eq:diffquotient_eta}
 & \lim_{h\to 0} \frac{1}{h}\bigl( \etabb(t-s;\x + h\bm{e}_j,t) - \etabb(t-s;\x,t) \bigr)  = \partial_{x_j} \etabb(t-s;\x,t) \qquad \text{in } L^2(\d s;\R),
\\ \label{eq:diffquotient_fxs}
 & \lim_{h\to 0} \frac{1}{h}\bigl( \fxsbb(\kap;\x + h\bm{e}_j,t) - \fxsbb(\kap;\x,t) \bigr) = \partial_{x_j} \fxsbb(\kap;\x,t) \qquad \text{in } L^2(\d \kap;\R).
\end{align}
To verify the convergence in \eqref{eq:diffquotient_exp},
we are going to apply the dominated convergence theorem. 
For this 
note that \eqref{eq:partial_deriv_bxi} and 
Gronwall's inequality yield
\begin{align}\label{eq:estimate_partial_bxi}
\bigl\| \partial_{x_j} \flow(s;\x,t) \bigr\| \leq \exp\Bigl\{ \| \nabla_\x \overline{\u} \|_\infty |t-s| \Bigr\},
\end{align}
where we recall that $\nabla_\x \overline{\u}$ is the transpose of the Jacobian matrix $D_\x \overline{\u}$. 
Moreover, observe that \eqref{eq:estimate_partial_bxi} and \eqref{eq:assumption_eta} imply 
\begin{align*}
& \Bigl\| \partial_{x_j}\Bigl(\exp\Bigl\{ i \frac{1}{\delta}\, \kap\cdot\flow(s;\x,t) \Bigr\} \Bigr) \Bigr\| \leq \frac{1}{\delta} \nkap \exp\Bigl\{ \| \nabla_\x \overline{\u} \|_\infty |t-s| \Bigr\},
\\ &\Bigl\| \frac{1}{h}\Bigl( \exp\Bigl\{ i \frac{1}{\delta}\, \kap\cdot \flow(s;\x + h\bm{e}_j,t) \Bigr\} - \exp\Bigl\{ i \frac{1}{\delta}\, \kap\cdot \flow(s;\x,t) \Bigr\}\Bigr) \Bigr\| \leq \frac{1}{\delta} \nkap \exp\Bigl\{ \| \nabla_\x \overline{\u} \|_\infty |t-s| \Bigr\},
\end{align*}
and
\begin{align*}
\int_{\R} \exp\Bigl\{ 2 \| \nabla_\x \overline{\u} \|_\infty |t-s| \Bigr\} \etabb^2(t-s;\x,t) \,\d s = \int_{\R} \exp\Bigl\{ 2 \| \nabla_\x \overline{\u} \|_\infty \delta \st(\x,t) |s| \Bigr\} \eta^2(s) \,\d s < \infty.
\end{align*} 
This and the fact that $\int_{\R^3} \nkap^2 \;\!\fxbb(\kap;\x,t) \,\d\kap < \infty$ ensure that the convergence in \eqref{eq:diffquotient_exp} follows from the corresponding pointwise convergence and the dominated convergence theorem.
Regarding \eqref{eq:diffquotient_eta} observe that for $j = 1$ we have 
\begin{align*}
& \frac{1}{h}\bigl( \etabb(t-s;\x + h\bm{e}_1,t) - \etabb(t-s;\x,t) \bigr)  - \partial_{x_1} \etabb(t-s;\x,t) 
\\ & \qquad 
= \frac{1}{h} \int_{x_1}^{x_1+h} \partial_{x_1} \etabb\bigl(t-s;(y,x_2,x_3),t\bigr) - \partial_{x_1} \etabb(t-s;\x,t) \,\d y,
\end{align*}
and similarly for $j \in \{2,3\}$. The term on the right hand side converges in $L^2(\d s;\R)$ as $h \to 0$ given that $\x \mapsto \partial_{x_j} \etabb(t-s;\x,t)$ is continuous in $L^2(\d s;\R)$. This, in turn, 
is implied by the continuity of 
$\x \mapsto \partial_{x_j} \etabb(t-s;\x,t)$ for fixed $s \in \R$  and the continuity of $\x \mapsto \| \partial_{x_j} \etabb(t-s;\x,t) \|_{L^2(\d s;\R)}$.  
The latter properties follow from items b), d) and \eqref{eq:assumption_eta} in Model \ref{model:turb_inhom}, 
noting that 
\begin{align}\label{eq:norm_partial_deriv_etabb}
\bigl\| \partial_{x_j} \etabb(t-s;\x,t) \bigr\|_{L^2(\d s;\R)}^2 & = \Bigl| \frac{\partial_{x_j} \st(\x,t)}{\st(\x,t)} \Bigr|^2 \int_\R \Bigl| \frac{1}{2}\eta(s) + s \etad(s) \Bigr|^2 \d s.
\end{align}  
The 
convergence in \eqref{eq:diffquotient_fxs} is a consequence of the corresponding pointwise convergence in combination with 
the integrability assumption 
\eqref{eq:assumption_deriv_fxsbb} and the dominated convergence theorem.

Finally, we note that the mean-square continuity of $\partial_{x_j} \u'$ in $(\x,t)$ is established using analogous arguments to those employed for the mean-square continuity $\u'$.  
\end{proof}

We are now able to specify and verify the characteristic turbulent flow and incompressibility properties of our inhomogeneous model. 

\begin{theorem}[Characteristic flow properties]\label{thm:kin_diss_div}
Let $\u' = (\u'(\x,t))_{(\x,t) \in \R^3 \times \R}$ be an inhomogeneous turbulence field in the sense of Model~\ref{model:turb_inhom}. 
Then for all $(\x,t) \in \R^3 \times \R$ it holds that
\begin{align}\label{eq:kinetic_energy_inhom}
\frac{1}{2}\,\E \Bigl[ \bigl\| \u'(\x,t) \bigr\|^2 \Bigr] & = k(\x,t) = \su^2(\x,t),
\\ \label{eq:dissipation_inhom}
\lim_{\delta \to 0} \frac12 \,\delta ^2 z  \,\E \Bigl[ \bigl\| \nabla_{\x} \u'(\x,t) + \bigl( \nabla_{\x} \u'(\x,t)\bigr){\vphantom{)}}^{\!\top} \bigr\|^2 \Bigr] & = \frac{\eps(\x,t)}{\nu(\x,t)} = \frac{\su^2(\x,t)}{\sx^2(\x,t) \sz(\x,t)},
\\ \label{eq:divergence_inhom}
\lim_{\delta \to 0}  \E \Bigl[ \bigl| \delta \, \nabla_{\x} \cdot\u'(\x,t) \bigr|^2 \Bigr] & = 0.
\end{align} 
In addition and consistent with \eqref{eq:kinetic_energy_inhom}, 
the one-point velocity correlations satisfy 
\begin{equation}\label{eq:one-point_correlation}
\begin{aligned}
\E\bigl[\u'(\x,t)\otimes\u'(\x,t)\bigr] 
= k(\x,t)\Bigl[\frac{7}{15}\,\bm{L}(\x,t)\cdot\bm{L}(\x,t){\vphantom{)}}^{\!\top} + \frac{1}{5}\,\bm{I}\Bigr]. 
\end{aligned} 
\end{equation}
\end{theorem}
\medskip

We remark that the supplementary 
assertion \eqref{eq:one-point_correlation} 
allows to 
encode additionally given information of the 
Reynolds stress tensor $\bm{R}(\x,t)$ 
fulfilling the compatibility condition 
$\trace\bm{R}(\x,t)=2k(\x,t)$ \wrt the turbulent kinetic energy $k(\x,t)$  
into the model. 
Indeed, 
adapting a strategy from \cite{ADDK21,ADDK22}, 
one may choose the anisotropy factor $\bm{L}(\x,t)$ 
such that 
$\bm{L}(\x,t)\cdot\bm{L}(\x,t){\vphantom{)}}^{\!\top}= (15/7)\, k(\x,t)^{-1}\bm{R}(\x,t)- (3/7)\,\bm{I}$ 
in order to ensure that 
the right hand side of \eqref{eq:one-point_correlation} amounts to 
a prescribed value for
$\bm{R}(\x,t)$. 
In the isotropic case, $\bm{R}(\x,t)=(2/3)\, k(\x,t)\, \bm{I}$ and $\bm{L}(\x,t)=\bm{I}$ apply in particular. 

\begin{proof}[Proof of Theorem~\ref{thm:kin_diss_div}]
The four assertions are verified one by one. 
As before, we employ the shorthand notation 
$\etabb$, $\fxsbb$, and $\Pbb$
introduced in \eqref{eq:def_etabb_fxbb} to simplify the presentation. 
With regard to  
the first assertion \eqref{eq:kinetic_energy_inhom} 
and the supplementary statement \eqref{eq:one-point_correlation},  
note that the covariance formula \eqref{eq:cov_inhom} in combination with 
the integral assumptions \eqref{eq:basic_integral_prop_fx}, \eqref{eq:basic_integral_prop_eta} and a spherical coordinate transformation 
ensures that
\begin{equation}\label{eq:one-point_correlation:proof}
\E\bigl[\u'(\x,t) \otimes \u'(\x,t)\bigr]  = \int_{S^2} \Pbb(\btheta;\x,t) \cdot \Pbb(\btheta;\x,t)^\top U_{S^2}(\d \btheta),
\end{equation} 
where $U_{S^2}$ denotes the uniform distribution on the unit sphere $S^2$. 
Next observe that the projector $\Pt$ 
enjoys specific properties when integrated over the unit sphere. 
In particular, for any symmetric matrix $\bm{S} \in \R^{3\times 3}$ careful calculations result in 
the identity 
$\int_{S^2} \Pt \cdot \bm{S} \cdot \Pt  \, U_{S^2}(\d \btheta) = (7/15)\:\bm{S} + (1/15)\:\trace\!(\bm{S})\:\bm{I}$. 
Choosing 
$\bm{S}=k(\x,t)\: \bm{L}(\x,t)\cdot\bm{L}(\x,t){\vphantom{)}}^{\!\top}$ 
and taking into account 
the fact 
that $\trace(\bm{L}(\x,t)\cdot\bm{L}(\x,t){\vphantom{)}}^{\!\top})=\|\bm{L}(\x,t)\|^2=3$ 
establishes 
\eqref{eq:one-point_correlation} as a consequence of \eqref{eq:one-point_correlation:proof}. 
Taking 
the trace of both sides in \eqref{eq:one-point_correlation} readily implies \eqref{eq:kinetic_energy_inhom}.

In order to prove the second assertion \eqref{eq:dissipation_inhom}, we are going to show that
\begin{align}\label{eq:norm_gradient_u'}
\lim_{\delta \to 0} \delta^2 z \,\E \Bigl[ \bigl\| \nabla_\x \u'(\x,t) \bigr\|^2 \Bigr] = \lim_{\delta \to 0} \delta^2 z\, \E \Bigl[ \bigl\| \bigl( \nabla_\x \u'(\x,t) \bigr){\vphantom{)}}^{\!\top} \bigr\|^2 \Bigr] = \frac{\eps(\x,t)}{\nu(\x,t)}
\end{align}
and
\begin{align}\label{eq:scalar_product_gradient_u'}
\lim_{\delta \to 0} \delta^2 z \,\E \bigl[ \nabla_\x \u'(\x,t) : \bigl( \nabla_\x \u'(\x,t) \bigr){\vphantom{)}}^{\!\top}\bigr] = 0.
\end{align} 
Regarding \eqref{eq:norm_gradient_u'}, first note that Lemma \ref{lem:ms_diff_u'} and the isometric property \eqref{eq:isometric_property_real_part} of the stochastic integral ensure that $\u'$ is mean-square differentiable in $\x$ and the mean-square partial derivatives satisfy
\begin{align*}
\E \Bigl[ \bigl\| \partial_{x_j} \u'(\x,t) \bigr\|^2 \Bigr] & = \int_{\R^3 \times \R}  \biggl\|\partial_{x_j} \Bigl( \etabb(t-s;\x,t) \exp\Bigl\{ i \frac{1}{\delta}\, \kap\cdot \flow(s;\x,t) \Bigr\} \fxsbb(\kap;\x,t) \, \Pbb(\kap;\x,t) \Bigr)\biggr\|^2 \d(\kap, s),
\end{align*}
which leads to
\begin{equation}\label{eq:exp_value_nabla_u'}
\begin{aligned}
 & \delta^2 \E \Bigl[ \bigl\| \nabla_\x \u'(\x,t) \bigr\|^2 \Bigr] =  \sum_{j=1}^3 \delta^2 \E \Bigl[ \bigl\| \partial_{x_j} \u'(\x,t) \bigr\|^2 \Bigr]
\\ & \qquad = \int_{\R^3 \times \R} \bigl| \etabb(t-s;\x,t) \;\!\fxsbb(\kap;\x,t) \bigr|^2 \bigl\| \Pbb(\kap;\x,t)\bigr\|^2 \bigl\| \nabla_\x  \flow(s;\x,t) \cdot \kap \bigr\|^2  \d(\kap,s)
\\ &  \qquad \qquad +  \delta^2 \int_{\R^3 \times \R} \Bigl\| \nabla_\x \bigl(  \etabb(t-s;\x,t) \;\!\fxsbb(\kap;\x,t)\, \Pbb(\kap;\x,t) \bigr) \Bigr\|^2 \d(\kap,s)
\\ & \qquad =: I_1 + \delta^2 I_2.
\end{aligned}
\end{equation}
In view of 
the term $I_1$ in \eqref{eq:exp_value_nabla_u'}, observe that the identity \eqref{eq:partial_deriv_bxi} in the proof of Lemma~\ref{lem:ms_diff_u'} above implies that
\begin{align}\label{eq:gradient_bxi}
\nabla_\x \flow(s;\x,t) \cdot \kap = \kap + \int_t^s \nabla_\x \flow(r;\x,t) \cdot (\nabla_\x \overline{\u}) \bigl( \flow(r;\x,t),r\bigr) \cdot \kap \, \d r
\end{align}
and consequently
\begin{equation}\label{eq:I_1}
\begin{aligned}
 I_1 &= \int_{\R^3 \times \R}   \bigl| \etabb(t-s;\x,t) \;\!\fxsbb(\kap;\x,t) \bigr|^2  \bigl\| \Pbb(\kap;\x,t)\bigr\|^2 
 \\& \qquad\qquad
 \biggl\{ \nkap^2  
 + 2  \biggl[ \int_t^s \! \kap \cdot  \nabla_\x \flow(r;\x,t) \cdot (\nabla_\x \overline{\u}) \bigl( \flow(r;\x,t),r\bigr) \cdot \kap \, \d r \biggr] \\
 & \qquad \qquad\; + \biggl\| \int_t^s \nabla_\x \flow(r;\x,t) \cdot (\nabla_\x \overline{\u}) \bigl( \flow(r;\x,t),r\bigr) \cdot \kap \, \d r \biggr\|^2 \, \biggr\} \;\d(\kap,s) 
\\ & =: I_{1,1} + I_{1,2} + I_{1,3}.
\end{aligned}
\end{equation}
Next note that 
an application of a spherical coordinate transformation 
in combination with the second identity in \eqref{eq:basic_integral_prop_fx} and
the fact that $\int_{S^2} \| \Pt \cdot \bm{L}(\x,t) \|^2 \,U_{S^2}(\d\btheta) = 2$ 
yields
\begin{align}\label{eq:identity_L_and_fxbb}
\int_{\R^3} \nkap^2 \;\! \fxbb(\kap;\x,t) \;\! \bigl\| \Pbb(\kap;\x,t)\bigr\|^2 \,\d\kap = \frac{\su^2(\x,t)}{\sx^2(\x,t) \sz(\x,t) z} =  \frac{\eps(\x,t)}{\nu(\x,t) z}.
\end{align}
This and \eqref{eq:basic_integral_prop_eta} ensure that
the first summand on the right-hand side of \eqref{eq:I_1} equals
\begin{equation}\label{eq:I_1_1}
I_{1,1} = \int_{\R} \etabb^2(t-s;\x,t) \,\d s \, \int_{\R^3} \nkap^2 \;\!\fxbb(\kap;\x,t)\;\! \bigl\| \Pbb(\kap;\x,t)\bigr\|^2 \,\d\kap  = \frac{\eps(\x,t)}{\nu(\x,t) z}.
\end{equation}
Regarding the term $I_{1,2}$ in \eqref{eq:I_1}, we employ the Gronwall estimate \eqref{eq:estimate_partial_bxi} in the proof of Lemma~\ref{lem:ms_diff_u'} above to obtain that
\begin{align*}
\biggl| \int_t^s \kap \! \cdot  \nabla_\x \flow(r;\x,t) \cdot (\nabla_\x \overline{\u}) \bigl( \flow(r;\x,t),r\bigr) \cdot \kap \, \d r \biggr| 
& \leq 
 |t-s| \;\!\nkap^2 \;\! \sqrt{3}\;\!   \exp\Bigl\{ \| \nabla_\x \overline{\u} \|_\infty |t-s| \Bigr\} \| \nabla_\x \overline{\u} \|_\infty.
\end{align*}
Using this, 
the definition of $\etabb$ in 
\eqref{eq:def_etabb_fxbb}, and \eqref{eq:identity_L_and_fxbb} we arrive at 
\begin{align}
I_{1,2} & \leq  4\;\!  \| \nabla_\x \overline{\u} \|_\infty \! \int_{\R} |t-s| \exp\Bigl\{ \| \nabla_\x \overline{\u} \|_\infty \! |t-s| \Bigr\} \;\!\etabb^2(t-s;\x,t) \, \d s    
\int_{\R^3} \! \nkap^2 \;\!\fxbb(\kap;\x,t) \bigl\| \Pbb(\kap;\x,t)\bigr\|^2 \d\kap
\notag
\\ & =   4 \;\!   \delta \;\!  \st(\x,t) \frac{\eps(\x,t)}{\nu(\x,t)z} \;\! \| \nabla_\x \overline{\u} \|_\infty  \int_{\R} |s| \exp\Bigl\{ \| \nabla_\x \overline{\u} \|_\infty \delta \st(\x,t) |s| \Bigr\}\;\! \eta^2(s) \, \d s.
\label{eq:I_1_2}   
\end{align}
The term $I_{1,3}$ in \eqref{eq:I_1} can be handled analogously to $I_{1,2}$, resulting in the estimate  
\begin{equation*}
I_{1,3}  \leq  3\;\!  \delta^2 \st^2(\x,t) \frac{\eps(\x,t)}{\nu(\x,t)z} \;\!   \| \nabla_\x \overline{\u} \|_\infty^2  \int_{\R} s^2 \exp\Bigl\{ 2 \| \nabla_\x \overline{\u} \|_\infty \delta \st(\x,t) |s| \Bigr\} \;\! \eta^2(s) \, \d s. 
\end{equation*} 
Combining this with \eqref{eq:I_1}, \eqref{eq:I_1_1}, \eqref{eq:I_1_2}, and 
the integrability assumption 
\eqref{eq:assumption_eta} establishes that
\begin{align}\label{eq:I_1_limit}
\lim_{\delta \to 0} I_1 = \frac{\eps(\x,t)}{\nu(\x,t)z} .
\end{align}
In view of  
the integral $I_2$ in \eqref{eq:exp_value_nabla_u'}, note that 
the integral identities 
\eqref{eq:basic_integral_prop_fx}, \eqref{eq:basic_integral_prop_eta} and the fact
that $\int_{S^2} \| \Pt \cdot \bm{L}(\x,t) \|^2 \,U_{S^2}(\d\btheta) = 2$ and $\|\Pk\|^2 = 2$ for $\kap \neq \bm{0}$ 
imply 
\begin{equation}\label{eq:I_2_estimate_1} 
I_2 \leq  6 \;\! k(\x,t) \biggl[ \int_{\R} \bigl\| \nabla_\x \etabb(t-s;\x,t) \bigr\|^2 \d s +  \int_{\R^3} \bigl\| \nabla_\x \fxsbb(\kap;\x,t) \bigr\|^2 \d\kap \biggr] + 6 \;\! \Bigl\| \nabla_\x \bigl(\su(\x,t)\bm{L}(\x,t)\bigr) \Bigr\|^2.
\end{equation} 
The identity \eqref{eq:norm_partial_deriv_etabb} in the proof  of Lemma \ref{lem:ms_diff_u'} above
and the integrability assumptions \eqref{eq:assumption_eta}, \eqref{eq:assumption_deriv_fxsbb} 
therefore ensure that 
$\sup_{\delta > 0} I_2 < \infty$. 
Combining this with \eqref{eq:exp_value_nabla_u'} and \eqref{eq:I_1_limit}  
establishes
the assertion \eqref{eq:norm_gradient_u'}. 
Regarding \eqref{eq:scalar_product_gradient_u'}, observe that Lemma \ref{lem:ms_diff_u'},
a summand-wise application of Lemma \ref{lem:covmain}, and the Cauchy-Schwarz inequality
yield
\begin{align}\label{eq:equality_of_scalar_product_gradient_u'}
& \delta^2 \Bigl| \E \bigl[ \nabla_\x \u'(\x,t) : \bigl( \nabla_\x \u'(\x,t) \bigr){\vphantom{)}}^{\!\top}\bigr] \Bigr| 
= \delta^2 \Bigl| \sum_{j,l=1}^3 \E \bigl[ \partial_{x_j}  u'_l(\x,t) \;  \partial_{x_l} u'_j(\x,t) \bigr] \Bigr| 
\\ & \quad \leq \delta^2  \int_{\R^3 \times \R} \Bigl\|  \Pk \cdot \nabla_\x \Bigl(  \su(\x,t) \etabb(t-s;\x,t) \exp\Bigl\{ i \frac{1}{\delta} \kap \cdot \flow(s;\x,t)\Bigr\} \fxsbb(\kap;\x,t)  \bm{L}(\x,t) \Bigr) \Bigr\|^2 \d(\kap,s). \notag
\end{align}
Taking into account that $\Pk \cdot \kap = \bm{0}$, 
the upper bound in \eqref{eq:equality_of_scalar_product_gradient_u'} can be estimated by 
techniques similar to those employed before, 
leading to a term of order $\delta^2$. This implies \eqref{eq:scalar_product_gradient_u'}.

It remains to verify the third assertion \eqref{eq:divergence_inhom}. 
To this end, note that Lemma \ref{lem:ms_diff_u'} 
and the isometric property \eqref{eq:isometric_property_real_part} of the stochastic integral ensure 
that the divergence of $\u'$ in the mean-square sense satisfies 
\begin{equation*}
\E \Bigl[ \bigl| \delta \,\nabla_\x \cdot \u'(\x,t) \bigr|^2 \Bigr]
= \delta^2 \int_{\R^3 \times \R} \Bigl\|  \nabla_\x \cdot \Bigl( \etabb(t-s;\x,t) \exp\Bigl\{ i \frac{1}{\delta} \kap \cdot \flow(s;\x,t)\Bigr\} \fxsbb(\kap;\x,t) \;\! \Pbb(\kap;\x,t) \Bigr) \Bigr\|^2 \d(\kap,s).
\end{equation*} 
As 
the term on the right-hand side 
is less than or equal to 
three times the upper bound in 
\eqref{eq:equality_of_scalar_product_gradient_u'}, 
it follows that \eqref{eq:divergence_inhom} is established in the same way as \eqref{eq:scalar_product_gradient_u'}. 
\end{proof}

In extension of the 
homogeneous ergodicity result
in Proposition~\ref{prop:ergodicity_hom}, the following theorem shows that local characteristic values of the inhomogeneous fluctuation field can be recovered from suitable local averages in time and space. Loosely speaking, the asymptotic statement requires that both the turbulence scale ratio $\delta$ and the size of the domain of integration tend to zero, but the latter with a slower rate than the former. We note that the 
definition and the formal interpretation of the average integrals 
and the convergence assertions in Theorem~\ref{thm:ergodicity_inhom} 
are completely analogous to Proposition~\ref{prop:ergodicity_hom}.

\begin{theorem}[Ergodicity]\label{thm:ergodicity_inhom}
Let $\u' = (\u'(\x,t))_{(\x,t) \in \R^3 \times \R}$ be an inhomogeneous turbulence field in the sense of Model \ref{model:turb_inhom} and let $L,L' \in [0,\infty]$ be such that $L = \infty \,\vee \, L' = \infty$.
Then, considering radii of integration $R,R' \in \R_0^+$ and the turbulence scale ratio $\delta \in \R^+$, for every $(\x,t)\in\R^3\times\R$ we have the following convergences in the mean-square sense as $(R/\delta,R'/\delta,R,R')\to(L,L',0,0)$ in $\overline{\R}^4$: 
\begin{align}
& \mint{-}_{B_R^{(1)}(t)}\mint{-}_{B_{R'}^{(3)}(\flow(s;\x,t))} \u'(\y,s) \,\d\y\,\d s 
\;\; \xrightarrow{\;L^2\;}\;\; \mathbf 0 
\label{eq:ergodicity_thm_inhom_A}\\
\frac12\,&  \mint{-}_{B_R^{(1)}(t)}\mint{-}_{B_{R'}^{(3)}(\flow(s;\x,t))} \bigl\|\u'(\y,s) \bigr\|^2 \,\d\y\,\d s 
\;\; \xrightarrow{\;L^2\;}\;\; k(\x,t)  
\label{eq:ergodicity_thm_inhom_B}\\
\frac12 \, \delta^2z\,& \mint{-}_{B_R^{(1)}(t)}\mint{-}_{B_{R'}^{(3)}(\flow(s;\x,t))} \bigl\| \nabla_\y\u'(\y,s) + \bigl(\nabla_\y\u'(\y,s)\bigr){\vphantom{)}}^{\!\top}\bigr\|^2 \,\d\y\,\d s 
\;\; \xrightarrow{\;L^2\;}\;\; \frac{\eps(\x,t)}{\nu(\x,t)} 
\label{eq:ergodicity_thm_inhom_C}
\end{align}
In addition and consistent with \eqref{eq:ergodicity_thm_inhom_B}, 
the average integral obtained by replacing 
the integrand 
$\|\u'(\y,s)\|^2/2$ 
on the left-hand side of \eqref{eq:ergodicity_thm_inhom_B} 
with $\u'(\y,s) \otimes \u'(\y,s)$ 
converges to the 
one-point velocity correlation tensor 
$k(\x,t)\,[(7/15)\,\bm{L}(\x,t)\cdot\bm{L}(\x,t){\vphantom{)}}^{\!\top}\! + (1/5)\;\!\bm{I}]$.   
\end{theorem}

\begin{proof} 
We first focus on the assertions \eqref{eq:ergodicity_thm_inhom_A},
\eqref{eq:ergodicity_thm_inhom_B} and 
the last statement 
concerning the one point correlations.
As in the proof of Proposition \ref{prop:ergodicity_hom} we adopt a unified perspective and represent the average integrals in the form $\mint{-}\mint{-} \phi(\u'(\y,s)) \,\d\y\,\d s$  with suitably chosen functions $\phi \colon \R^3 \to \R$, taking $\phi(\u) = u_l$ with $l \in \{1,2,3\}$ in case \eqref{eq:ergodicity_thm_inhom_A},
$ \phi(\u) = \|\u\|^2/2 $ in case $\eqref{eq:ergodicity_thm_inhom_B}$,
and $\phi(\u) = u_j u_l$ with $j,l \in \{1,2,3\}$ 
for the one-point correlations,
$\u = ( u_1,u_2,u_3) \in \R^3$.
Note that the mean-square continuity of $\u'$ established in Lemma \ref{lem:ms_diff_u'} entails the continuity of $\R^3 \times \R \ni (\y,s) \mapsto \phi(\u'(\y,s)) \in L^2(\Pr;\R)$, so that all appearing Bochner integrals are defined. 
In the non-trivial cases $\phi(\u) = \|\u\|^2/2$ 
and $\phi(\u) = u_j u_l$,
this can be seen, e.g., by employing the Kahane-Khinchin inequality for Gaussian random variables, cf.~\cite[Theorem~6.2.6]{HvNVW17}. The main idea of the proof is to reformulate and simplify the claimed convergences in a suitable way in order to be able to apply the ergodicity result from Corollary~\ref{cor:ergo} in Appendix~\ref{sec:ergodicity} to the random field $\w = ( \w(\x,t,\xbb,\tbb))_{(\x,t),(\xbb,\tbb) \in \R^3 \times \R}$ defined by
\begin{equation}\label{eq:ergo_inhom_def_w}
\begin{aligned}
\w(\x,t,\xbb,\tbb)  = 
\Re\! &\int_{\R^3 \times \R} \Bigl(\frac1{\st(\xbb,\tbb)}\Bigr)^{1/2}\eta\Bigl(\frac{1}{\st(\xbb,\tbb)} ( t - r )\Bigr) 
\exp\bigl\{ i \, \kap\cdot \x \bigr\} \,\fxsbb(\kap; \xbb,\tbb)\, \Pbb(\kap;\xbb,\tbb) \cdot \wn(\d\kap,\d r)
\end{aligned}
\end{equation}
with $\fxsbb$ and $\Pbb$ as introduced in \eqref{eq:def_etabb_fxbb}. For this purpose, fix $(\x,t) \in \R^3 \times \R$ and note that the turbulence field $\u' = (\u'(\y,s))_{(\y,s) \in \R^3 \times \R}$ is approximated near $(\x,t)$ by the field $\u'_{\textrm{app}} = (\u'_{\textrm{app}}(\y,s))_{(\y,s) \in \R^3 \times \R}$ given by
\begin{equation}\label{eq:def_u_app}
\begin{aligned}
\u'_{\textrm{app}}(\y,s) = 
\Re\! & \int_{\R^3 \times \R}\Bigl(\frac1{\delta\st(\y,s)}\Bigr)^{1/2}\eta\Bigl(\frac{1}{\delta\st(\y,s)}(s-r)\Bigr) 
\\ & \vphantom{\int} \exp\Bigl\{ i \frac{1}{\delta}\, \kap\cdot \Bigl( \y + \int_s^r \overline{\u}\bigl( \flow(\tau;\x,t),\tau \bigr) \,\d \tau  \Bigr) \Bigr\} 
\,\fxsbb(\kap; \y, s) \, \Pbb(\kap;\y,s) \cdot \wn(\d\kap,\d r)
\end{aligned}
\end{equation}  
in the sense that
\begin{equation}\label{eq:ergo_approx_u_app}
\begin{aligned}
\lim_{\delta,R,R' \to 0} 
\sup\Bigl\{ \bigl\|\u'(\y,s)-\u'_{\textrm{app}}(\y,s) \bigr\|_{L^2(\Pr;\R^3)}\colon s\in B_R^{(1)}(t) ,\, \y\in B_{R'}^{(3)}\bigl( \flow(s;\x,t) \bigr)\Bigr\}
=0.
\end{aligned}
\end{equation}
We remark that the term to the right of the integral sign in the first line of \eqref{eq:def_u_app} is nothing but $\etabb(s-r;\y,s)$ as introduced in \eqref{eq:def_etabb_fxbb}, written out in full for the sake of comparability with \eqref{eq:def_u_mod} below. The approximation \eqref{eq:ergo_approx_u_app} is essentially due to the regularity properties of the flow function $\flow$, in particular the continuity of $ (r,\y,s) \mapsto \flow(r;\y,s)$ from $\R \times \R^3 \times \R$ to $\R^3$, which can be exploited, e.g., by expanding the error $\| \u'(\y,s) - \u'_{\textrm{app}}(\y,s) \|_{L^2(\Pr;\R^3)}$ by means of a suitable telescopic sum involving a truncation of the integrand $\eta$ of the form $\ind_{[-N,N]}\;\!\eta$, where $N \in \N$. Moreover, observe that \eqref{eq:ergo_approx_u_app} and the  Kahane-Khinchin inequality for Gaussian random variables imply an analogous approximation of $\phi(\u'(\y,s))$ by $\phi(\u'_{\textrm{app}}(\y,s))$ in $L^2(\Pr;\R)$. As a first simplification, we thus obtain that it is sufficient to verify \eqref{eq:ergodicity_thm_inhom_A}, \eqref{eq:ergodicity_thm_inhom_B},
and the supplementary assertion
with $\u'_\textrm{app}(\y,s)$ in place of $\u'(\y,s)$. Next observe that the covariance formula in Corollary \ref{cor:covmain} in combination with the change of variables formula yields that, for fixed $(\x,t)$ and $\delta$, the distribution of the random field $\u'_{\textrm{app}}$ is identical to the distribution of the modified field $\u'_{\textrm{mod}} = (\u'_{\textrm{mod}}(\y,s))_{(\y,s) \in \R^3 \times \R}$ defined by
\begin{equation}\label{eq:def_u_mod}
\begin{aligned}
\u'_{\textrm{mod}}(\y,s)  = 
\Re\! &\int_{\R^3 \times \R} \Bigl(\frac1{\st(\y,s)}\Bigr)^{1/2}\eta\Bigl(\frac{1}{\st(\y,s)} \Bigl( \frac{s}{\delta} - r \Bigr)\Bigr) 
\\ & \vphantom{\int} \exp\Bigl\{ i \frac{1}{\delta}\, \kap\cdot \Bigl( \y + \int_s^t \overline{\u}\bigl( \flow(\tau;\x,t),\tau\bigr)\,\d\tau \Bigr) \Bigr\} 
\,\fxsbb(\kap; \y,s) \, \Pbb(\kap;\y,s) \cdot \wn(\d\kap,\d r)
\end{aligned}
\end{equation}
As a second simplification, the fact that the limit values in
the claimed convergences
are non-random thus ensures that it is sufficient to verify 
the
assertions with $\u'_{\textrm{mod}}(\y,s)$ in place of $\u'(\y,s)$. To this end, note that the change of variables formula implies that
\begin{equation}\label{eq:ergo_inhom_transf}
\begin{aligned}
\mint{-}_{B_{R}^{(1)}(t)} & \mint{-}_{B_{R'}^{(3)}( \flow(s;\x,t))} \phi\bigl( \u'_{\textrm{mod}}(\y,s) \bigr) \,\d \y \, \d s
\\ &  = \mint{-}_{B_{R/\delta}^{(1)}(0)} \mint{-}_{B_{R'/\delta}^{(3)}(\bm{0})} \phi\bigl( \u'_{\textrm{mod}}( \flow(t+\delta s;\x,t) + \delta \y,\;\! t+\delta s) \bigr) \,\d \y \, \d s.
\end{aligned}
\end{equation}
In view of the random field $\w$ introduced in \eqref{eq:ergo_inhom_def_w} above and the fact that $\flow(t+\delta s;\x,t) = \x + \int_t^{t+\delta s} \overline{\u}(\flow(\tau;\x,t),\tau) \,\d \tau$, the argument of $\phi$ appearing in the integrand on the right-hand side of \eqref{eq:ergo_inhom_transf} can be rewritten as
\begin{equation*}
\u'_{\textrm{mod}}\bigl( \flow(t+\delta s;\x,t) + \delta \y , \;\! t+\delta s\bigr) = \w\Bigl( \frac{1}{\delta} \x + \y ,\;\! \frac{t}{\delta} + s ,\;\! \flow(t+\delta s;\x,t) + \delta \y ,\;\! t+\delta s \Bigr).
\end{equation*}
Moreover, an application of the transformation formula in Proposition \ref{prop:linear_isometry_on_white_noise_integral} similar to the reasoning in the proof of Lemma \ref{lem:transformation2} in combination with the covariance formula in Lemma \ref{lem:covmain} show that $\w$ satisfies the covariance condition \eqref{eq:cov_pert_ergo} in Corollary \ref{cor:ergo}. In addition, arguments similar to the ones in the proof of Lemma \ref{lem:ms_diff_u'} and the Kahane-Khinchin inequality for Gaussian random variables yield that $\w$ and $\phi$ fulfill the mean-square continuity assumptions in Corollary \ref{cor:ergo}. 
Considering arbitrary sequences $(R_j)_{j\in\N}, (R_j')_{j\in\N} \subset \R_0^+$, $(\delta_j)_{j\in\N}\subset \R^+$ such that $\lim_{j\to\infty} R_j = \lim_{j\to\infty} R_j' = 0$, $\lim_{j\to\infty} R_j/\delta_j = L$, and $\lim_{j\to\infty} R'_j/\delta_j = L'$, 
the convergences 
in \eqref{eq:ergodicity_thm_inhom_A}, \eqref{eq:ergodicity_thm_inhom_B}, 
and the supplementary assertion 
thus follow by applying Corollary \ref{cor:ergo} (with $\mathcal{R}_j = R_j/\delta_j$, $\mathcal{R}'_j = R'_j/\delta_j$, $\x_j = \x/\delta_j$, $t_j = t/\delta_j$, 
$\xbb=\x$, $\tbb=t$, 
$\bm{\alpha}_j(\y,s) = \frac{1}{\delta_j} \int_t^{t+\delta_j s} \overline{\u}(\flow(\tau;\x,t),\tau) \,\d \tau+\y$,
and $\beta_j(\y,s) = s$ for $(\y,s) \in \R^3 \times \R$ in the notation of Corollary \ref{cor:ergo}) and taking into account that $\E[\w(\bm{0},0,\x,t)] = \bm{0}$, $\E\|\w(\bm{0},0,\x,t)\|^2/2 = k(\x,t)$ and 
$\E[ \w(\bm{0},0,\x,t) \otimes \w(\bm{0},0,\x,t) ] 
= k(\x,t)\,[(7/15)\,\bm{L}(\x,t)\cdot\bm{L}(\x,t){\vphantom{)}}^{\!\top} + (1/5)\;\!\bm{I}]$.

The strategy of the proof of the previous three statements based on Corollary \ref{cor:ergo} can be adapted in order to establish assertion \eqref{eq:ergodicity_thm_inhom_C}. A suitable substitute for the approximation step \eqref{eq:ergo_approx_u_app} above is obtained by employing the field $(\delta \nabla_\y \u')_{\textrm{app}} = ( (\delta \nabla_\y \u')_{\textrm{app}}(\y,s))_{(\y,s) \in \R^3 \times \R}$ defined by
\begin{equation*}\label{eq:def_delta_partial_u_app}
\begin{aligned}
(\delta \nabla_\y \u')_{\textrm{app}}(\y,s) = 
\Re\! &\int_{\R^3 \times \R}\Bigl(\frac1{\delta\st(\y,s)}\Bigr)^{1/2}\eta\Bigl(\frac{1}{\delta\st(\y,s)}(s-r)\Bigr) \,\fxsbb(\kap; \y, s)
\\ & \vphantom{\int} i \exp\Bigl\{ i \frac{1}{\delta}\, \kap\cdot\Bigl( \y + \int_s^r \overline{\u}\bigl( \flow(\tau;\x,t),\tau \bigr) \,\d \tau \Bigr) \Bigr\} 
 \, \kap \otimes \Pbb(\kap;\y,s) \cdot \wn(\d\kap,\d r)
\end{aligned}
\end{equation*}
as an approximation of $\delta \nabla_\y \u'$ near $(\x,t)$. Choosing a suitable $\R^{3\times 3}$-valued random field $\w = ( \w(\x,t,\xbb,\tbb))_{(\x,t),(\xbb,\tbb) \in \R^3 \times \R}$ and using an identification mapping $\bm{\pi}\colon \R^{3\times 3} \to \R^9$, an application of Corollary \ref{cor:ergo} with $\phi \colon \R^9 \to \R$, $\phi(\u) = \| \bm{\pi}^{-1}(\u) + (\bm{\pi}^{-1}(\u))^\top\|^2$, establishes the convergence in \eqref{eq:ergodicity_thm_inhom_C} similarly to the first part of the proof.
\end{proof}


\section{Outlook}\label{sec:outlook}

The 
analytical derivations and results 
presented 
in this article are 
complemented 
in 
the accompanying paper \cite{AKLMW25}, 
which concerns the numerical approximation of Model~\ref{model:turb_inhom} 
and discusses 
the key model features 
by means of a variety of 
simulation results. 
In particular, we 
establish a suitable discretization scheme 
combining a randomized quadrature method for stochastic integrals 
with a local linearization of the  non-uniform mean flow advection. 
Convergence of the scheme towards the continuous model ist proven, 
and an efficient algorithmic implementation is 
described 
that allows for a flexible localized synthesis of the inhomogeneous field. 
The presented numerical simulations 
are based on 
the model spectrum and the time integration kernel from Examples \ref{ex:energy_spectrum}, \ref{ex:temporal_cor}, and \ref{ex:energy_spectrum_correlation_function} 
and illustrate, among others, 
the influence of the inhomogeneous scaling factors $\su$, $\sx$, $\st$, $\sz$ and the mean flow function $\flow$ 
on the generated fluctuations, 
the characteristic flow and ergodicity properties established in 
Theorems~\ref{thm:kin_diss_div} and \ref{thm:ergodicity_inhom},  
and the 
effect of the local turbulence Reynolds number on 
the validity of Komogorov's two-thirds law. 



\appendix
\renewcommand{\theequation}{\thesection.\arabic{equation}}

\section{Auxiliary results on integration with respect to white noise}\label{app:white_noise}

Here we collect various auxiliary results on stochastic integration \wrt Gaussian white noise. 
Starting with a recap of  the construction of stochastic integrals in Subsection \ref{subsec:appendix_white_noise}, Subsections~\ref{subsec:appendix_covariance_formulas}, \ref{subsec:appendix_transformation_results}, and \ref{subsec:appendix_differentiability} concern covariance formulas, integral transformations, and differentiability properties, respectively. 

\subsection{Complex-valued white noise and associated stochastic integral}
\label{subsec:appendix_white_noise}

We give a definition of com\-plex-valued Gaussian white noise and recall the construction of the associated stochastic integral for deterministic integrands.

\begin{definition}[Gaussian white noise]\label{def:gaussian_white_noise}
Let $U \subset \R^m$ be a Borel set, let $\mu \colon \B(U) \to [0,\infty]$ be a $\sigma$-finite measure, and 
set $\B_0(U) := \{ A \in \B(U) \colon \mu(A) < \infty\}$. 
\begin{enumerate}[label= \emph{\textbf{\alph*)}}]
\item A mapping $\wnr \colon \B_0(U) \to L^2(\Pr;\R)$ is called a  \emph{(real-valued) Gaussian white noise on $U$ with structural measure $\mu$} if \vspace{1mm}
\hspace{-5mm}
\begin{itemize}[leftmargin=8mm]
\item for all $A \in \B_0(U)$ the random variable $\wnr(A)$ has a (possibly degenerate) Gaussian distribution with mean zero and variance $\mu(A)$; \vspace{1mm}
\item for all disjoint sets $A,B \in \B_0(U)$ the random variables $\wnr(A)$, $\wnr(B)$ are independent and satisfy 
$\wnr(A \cup B) = \wnr(A) + \wnr(B)$ $\Pr$-almost surely.
\end{itemize}
\vspace{1mm}
\item A mapping $\wnr \colon \B_0(U) \to L^2(\Pr;\C)$ is called a \emph{(complex-valued) Gaussian white noise on $U$ with structural measure $\mu$} if $\Re\;\!\wnr$ and $\Im \;\! \wnr$ are independent real-valued Gaussian white noises on $U$ with structural measure $\mu/2$. A mapping $\wn = (\wnr_{1},\ldots,\wnr_{\ell}) \colon \B_0(U) \to L^2(\Pr;\C^\ell)$ is called a 
\emph{($\C^\ell$-valued) Gaussian white noise on $U$ with structural measure $\mu$} if $\wnr_{1},\ldots,\wnr_{\ell}$ are independent complex-valued Gaussian white noises on $U$ with structural measure $\mu$.
\end{enumerate}
\end{definition}

The above notion of Gaussian white noise is a special instance of more general concepts known under various names such as \emph{stochastic orthogonal measures} or  \emph{random pseudo-measures} \cite{GS74,Bre14}. 
Note that the assumption that the real and imaginary parts of a complex-valued Gaussian white noise are independent and share the same structural measure $\mu/2$ is consistent with the definition of complex-valued Gaussian random variables in the literature \cite{Hida80,LPS14}.

The stochastic integral \wrt a complex-valued Gaussian white noise $\wnr$ on a Borel set $U \subset \R^m$ with structural measure $\mu$ is constructed as follows. For simple functions $g$, the integral is defined as
\begin{equation*}
\int_{U} g(\x) \,\wnr(\d \x) := \sum_{j=1}^n a_j \;\!\wnr(A_j), \qquad g(\x) = \sum_{j=1}^n a_j \;\! \indicator{A_j}(\x), 
\end{equation*} 
where $a_1,\ldots,a_n \in \C$ and $A_1,\ldots,A_n \in \B_0(U)$. It is straightforward to check that this integral satisfies the isometric property
\begin{align}\label{eq:isometric_property_0}
\E \biggl[ \Bigl| \int_U g(\x) \,\wnr(\d\x) \Bigr|^2 \biggr]  = \int_U \bigl| g(\x) \bigr|^2 \mu(\d\x).
\end{align}
As a consequence, the integral mapping $g\mapsto \int_U g(\x) \,\wnr(\d\x)$ can be extended linearly and continuously to an isometric mapping from $L^2(\mu;\C)$ to  $L^2(\Pr;\C)$; cf., e.g., \cite[Section~3.2.3]{Bre14}. The resulting real-valued random variables $\Re\!\!\: \int_U g(\x) \,\wnr(\d\x)$, $\Im\!\!\: \int_U g(\x) \,\wnr(\d\x)$,  $g\in L^2(\mu;\C)$, are centered and constitute a jointly Gaussian family. 
 
This construction is further extended to the multidimensional case in a canonical way. 
In particular, if $\wn = (\wnr_{1},\ldots,\wnr_{\ell})$ is a $\C^\ell$-valued Gaussian white noise on $U$ with structural measure $\mu$, then an isometric property analogous to \eqref{eq:isometric_property_0} holds true for $\C^{d\times \ell}$-valued integrands $\bm{G} \in L^2(\mu;\C^{d\times \ell})$, namely
\begin{align}\label{eq:isometric_property_1}
\E \biggl[ \Bigl\| \int_U \bm{G}(\x) \cdot \wn(\d\x) \Bigr\|^2 \biggr] = \int_U \bigl\| \bm{G}(\x) \bigr\|^2 \mu(\d\x).
\end{align}
While the specific form of the multidimensional isometry property \eqref{eq:isometric_property_1} is due to the independence of the components $\wnr_{1},\ldots,\wnr_{\ell}$ of $\wn$, the additional assumption in Definition~\ref{def:gaussian_white_noise} of independence and identical distributions of the real and imaginary parts entails the further isometric property
\begin{align}\label{eq:isometric_property_real_part}
\E \biggl[ \Bigl\| \, \Re \! \int_U \bm{G}(\x) \cdot \wn(\d\x) \Bigr\|^2 \biggr] 
=  \E \biggl[ \Bigl\| \, \Im \! \int_U \bm{G}(\x) \cdot \wn(\d\x) \Bigr\|^2 \biggr] 
=  \frac{1}{2}\int_U \bigl\| \bm{G}(\x) \bigr\|^2 \mu(\d\x).
\end{align}
The identities in \eqref{eq:isometric_property_real_part} are verified in the same way as \eqref{eq:isometric_property_0}, \eqref{eq:isometric_property_1} by following the single steps of the construction of the stochastic integral.

\subsection{Covariance formulas}\label{subsec:appendix_covariance_formulas}

The following covariance formula for white noise integrals is frequently used throughout this article. 

\begin{lemma}\label{lem:covmain}
Let $U \subset \R^m$ be a Borel set, let $\mu \colon \B(U) \to [0,\infty]$ be a $\sigma$-finite measure, let $\wn$ be a $\C^\ell$-valued Gaussian white noise on $U$ with structural measure $\mu$, and let $\bm{G}, \bm{H} \in L^2(\mu;\C^{d\times \ell})$. Then it holds that
\begin{align*}
\E\Bigl[ \;\!\Re\! \int_{U} \bm{G}(\y) \cdot \wn(\d \y) \otimes \Re\! \int_{U} \bm{H}(\y) \cdot \wn(\d \y) \Bigr] &  = \frac{1}{2} \, \Re\! \int_U \bm{G}(\y) \cdot \overline{\bm{H}(\y)}^\top \mu(\d \y). 
\end{align*}
\end{lemma}

\begin{proof}
The expected value of the $\R^{\ell \times \ell}$-valued random variable on the left-hand side will be calculated in a componentwise manner.
To abbreviate calculations, let $\bm{g}_j = (G_{j,l})_{l = 1}^{\ell} \in \C^\ell$ and $\bm{h}_j = (H_{j,l})_{l = 1}^{\ell} \in \C^\ell$ be the $j$th row of $\bm{G}$ and $\bm{H}$, respectively, for any $j \in \{1,\ldots,d\}$. Using the isometric property \eqref{eq:isometric_property_real_part}, 
we find for any $j,k \in \{1,\ldots,\ell\}$ that
\begin{align*}
& \E\biggl[ \Bigl( \Re\! \int_{U} \bm{G}(\y) \cdot \wn(\d \y) \Bigr)_j  \, \Bigl( \Re\! \int_{U} \bm{H}(\y) \cdot \wn(\d \y) \Bigr)_k \,\biggr]\\ 
& \;\; = \frac{1}{2} \biggl( \E\biggl[ \Bigl| \Re\! \int_{U} \bigl(\bm{g}_j(\y) + \bm{h}_k(\y) \bigr) \cdot \wn(\d \y) \Bigr|^2 \biggr] -\E\biggl[ \Bigl| \Re\! \int_{U} \bm{g}_j(\y) \cdot \wn(\d \y) \Bigr|^2 \biggr] 
- \E\biggl[ \Bigl| \Re\! \int_{U} \bm{h}_k(\y) \cdot \wn(\d \y) \Bigr|^2\biggr] \biggr)\\ 
& \;\; = \frac{1}{4} \int_U \bigl\| \bm{g}_j(\y) + \bm{h}_k(\y) \bigr\|^2 -  \bigl\| \bm{g}_j(\y) \bigr\|^2 -  \bigl\| \bm{h}_k(\y) \bigr\|^2 \mu(\d \y)
\;=\; 
\frac{1}{2} \, \Re\! \int_U \bm{g}_j(\y) \cdot \overline{\bm{h}_k(\y)} \,\mu(\d \y),  
\end{align*}
which concludes the proof.
\end{proof}

As a direct consequence of Lemma \ref{lem:covmain} we obtain the following result concerning stochastic Fourier-type integrals.

\begin{corollary}\label{cor:covmain}
In the setting of Lemma \ref{lem:covmain} we consider the case
\begin{align*}
\bm{G}(\y) = \exp\bigl\{ i\phi(\y) \bigr\} \bm{R}(\y), \qquad \bm{H}(\y) = \exp\bigl\{ i\psi(\y) \bigr\} \bm{Q}(\y), 
\end{align*}
$\y \in U$, for measurable functions $\phi, \psi \colon U \to \R$ and $\bm{R},\bm{Q} \in L^2(\mu;\R^{d\times \ell})$. Then
\begin{align*}
\E\Bigl[ \Re\! \int_{U} \bm{G}(\y) \cdot \wn(\d \y) \otimes \Re\! \int_{U} \bm{H}(\y) \cdot \wn(\d \y) \Bigr] = \frac{1}{2}\int_{U} \cos\bigl( \phi(\y) - \psi(\y) \bigr) \bm{R}(\y) \cdot \bm{Q}(\y)^\top \mu(\d \y).
\end{align*}
\end{corollary}

\smallskip

\subsection{Transformation and representation results}\label{subsec:appendix_transformation_results}

This subsection is devoted to the impact of different kinds of transformations on Gaussian white noises and stochastic integrals.

We start with a simple result regarding transformations  involving weight functions. 

\begin{lemma}\label{lem:multiplication}
Let $U \subset \R^m$ be a Borel set, let $\mu \colon \B(U) \to [0,\infty]$ be a $\sigma$-finite measure, let $\wn$ be a $\C^\ell$-valued Gaussian white noise on $U$ with structural measure $\mu$, and let $\rho\colon U\to \R$ be a measurable function such that the measure $\rho^2 \mu\colon \B(U) \to [0,\infty]$ given by $(\rho^2\mu)(A):=\int_A \rho^2(\x)\,\mu(\d\x)$ is $\sigma$-finite. 
Then the mapping $\rho\;\!\wn$ that assigns to every $A\in\B(U)$ with $(\rho^2\mu)(A)<\infty$ the random vector
\begin{align*}
 (\rho \;\! \wn)(A) := \int_A \rho(\x) \,\wn(\d \x)
 \end{align*}
is a $\C^\ell$-valued Gaussian white noise on $U$ with structural measure $\rho^2 \mu$. 
Moreover, for every $\bm{G} \in L^2(\rho^2\mu ;\C^{d\times \ell})$ it holds that
\begin{align*}
\int_U \bm{G}(\x) \cdot (\rho\;\!\wn)(\d \x) = \int_U \rho(\x) \bm{G}(\x) \cdot \wn(\d \x).
\end{align*}
\end{lemma}

\begin{proof}
Obviously, $\rho\:\!\wn$ is a $\C^\ell$-valued Gaussian white noise with structural measure $\rho^2\mu$ and the integral identity is trivial for simple functions $\bm{G}$. The general case is a consequence of the isometric property \eqref{eq:isometric_property_1} of the stochastic integral and the fact that $\bm{G}_n \to \bm{G}$ in $L^2(\rho^2\mu;\C^{d\times \ell})$ if, and only if, $\rho\;\!\bm{G}_n \to  \rho\;\!\bm{G}$ in $L^2(\mu;\C^{d\times \ell})$.
\end{proof}

In classical integration theory the change of variables formula is an effective evaluation method. The next result presents an analogue for white noise integrals.

\begin{proposition}[Change of variables]\label{prop:pushforward}
Let $U \subset \R^m$ and $V \subset \R^n$ be Borel sets, let $\mu \colon \B(U) \to [0,\infty]$ be a $\sigma$-finite measure, let $\wn$ be a $\C^\ell$-valued Gaussian white noise on $U$ with structural measure $\mu$, and let  $\bm{\phi} \colon U \to V$ be a measurable function.  
Then the following assertions hold:
\begin{enumerate}[label= \emph{\textbf{\alph*)}}] 
\item 
The mapping $\wn \circ \bm{\phi}^{-1}$ that assigns to every $A\in\B(V)$ with $\mu(\bm{\phi}^{-1}(A))<\infty$ the random vector $\wn(\bm{\phi}^{-1}(A))$ is a $\C^\ell$-valued Gaussian white noise on $V$ with structural measure $\mu \circ \bm{\phi}^{-1}$. Moreover, for every $\bm{G} \in L^2(\mu \circ \bm{\phi}^{-1};\C^{d\times \ell})$ we have  that
\begin{align*}
\int_V \bm{G}(\y) \cdot (\wn \circ \bm{\phi}^{-1})(\d \y)=\int_{U} \bm{G}(\bm{\phi}(\x)) \cdot \wn(\d \x).
\end{align*}
\item 
Additionally assume that $m=n$, that $\bm{\phi}$ is a $C^1$ diffeomorphism, and that the structural measure of $\wn$ is given by $\mu=\alpha\lambda^m|_U$, where  $\alpha\in\R^+$ is a constant factor and $\lambda^m|_U$ denotes Lebesgue measure on $U$. 
Then the mapping $\wt\wn$ that assigns to every $A\in\B(V)$ with $\lambda^m(A)<\infty$ the random vector 
\begin{align*}
\wt\wn(A) = \int_A \left| \det (D\bm{\phi})(\bm{\phi}^{-1}(\y)) \right|^{1/2}(\wn \circ \bm{\phi}^{-1})(\d \y)
\end{align*} 
is a $\C^\ell$-valued Gaussian white noise on $V$ with structural measure $\alpha\lambda^m|_V$. Moreover, for every $\bm{G} \in L^2(\lambda^m|_V;\C^{d\times \ell})$ we have that  
\begin{align*}
\int_{V} \bm{G}(\y)  \cdot \wt\wn(\d \y)
= \int_U \bigl| \det (D\bm{\phi})(\x)\bigr|^{1/2} \bm{G}(\bm{\phi}(\x))  \cdot \wn(\d \x).
\end{align*}
\end{enumerate}
\end{proposition}

\begin{proof}
Part a) is considered first. By definition, $\wn \circ \bm{\phi}^{-1} \colon \B_0(V) \to L^2(\Pr;\C^d)$ is a centered Gaussian process and its components $\wnr_1 \circ \bm{\phi}^{-1},\ldots,\wnr_\ell\circ \bm{\phi}^{-1}$ are obviously independent with independent real and imaginary parts. The additivity property of the mappings $\Re\;\!\wn \circ \bm{\phi}^{-1}$ and $\Im\;\!\wn \circ \bm{\phi}^{-1}$ follows from the set properties $\bm{\phi}^{-1}(A) \cap \bm{\phi}^{-1}(B) = \bm{\phi}^{-1}(A\cap B)$ and $\bm{\phi}^{-1}(A) \cup \bm{\phi}^{-1}(B) = \bm{\phi}^{-1}(A\cup B)$ for any Borel sets $A,B \subset V$.
For the second claim in part a) we first consider an indicator function $g = \indicator{A}$ for a Borel set $A \subset V$ with $\mu(\bm{\phi}^{-1}(A)) < \infty$. Then it holds that
\begin{align*}
\int_V g(\y) (\wn \circ \bm{\phi}^{-1})(\d \y) & = \wn(\bm{\phi}^{-1}(A)) = \int_U  \indicator{\bm{\phi}^{-1}(A)}(\x) \,\wn(\d \x) =  \int_U g(\bm{\phi}(\x)) \,\wn(\d \x).
\end{align*}
Using the linearity of the integral, this identity follows immediately for simple functions. Since $\bm{G}_n \to \bm{G}$ in $L^2(\mu \circ \bm{\phi}^{-1};\C^{d\times \ell})$ if, and only if, $\bm{G}_n \circ \bm{\phi} \to  \bm{G} \circ \bm{\phi}$ in $L^2(\mu;\C^{d\times \ell})$,
the general case follows by approximation using \eqref{eq:isometric_property_1}.

Regarding part b) observe that
\begin{align*}
\int_A \bigl| \det (D\bm{\phi})(\bm{\phi}^{-1}(\y)) \bigr|  (\mu \circ \bm{\phi}^{-1})(\d \y) = \alpha \int_{\bm{\phi}^{-1}(A)} \bigl| \det (D\bm{\phi})(\x) \bigr| \,\d\x = \alpha \int_A 1 \,\d\y =  \alpha \lambda^m(A),
\end{align*}
where we used the classical change of variables formula on the second equality. In particular, this ensures the existence of $\wt\wn(A)$ for every $A \in \B(V)$ with $\lambda^m(A)<\infty$. The remaining assertions of part b) follow from part a), the classical change of variables formula, and Lemma \ref{lem:multiplication}.
\end{proof}

The following corollary shows the impact of spherical coordinate transformations.

\begin{corollary}[Spherical coordinates]\label{cor:polar}
Let $\wn$ be a $\C^\ell$-valued Gaussian white noise on $\R^3\times\R$ with  structural measure $\lambda^3 \otimes \lambda^1$, 
let $\bm{\phi}_{\mathrm{sph}} \colon (\R^3\setminus\{0\}) \times \R \to \R^+\times S^2 \times \R$ be the spherical coordinate transformation given by 
\begin{align*}
\bm{\phi}_{\mathrm{sph}}(\kap,\omega) = (\kappa, \btheta ,\omega) = \Bigl(\|\kap\|, \frac{\kap}{\|\kap\|},\,\omega\Bigr),
\end{align*}
and let $U_{S^2}$ be the uniform distribution on the $2$-sphere (i.e., the normalized surface measure). 
Then the mapping $\wn_{\mathrm{sph}}$ that assigns to every 
$A\in \B( \R^{+} \times S^2 \times \R)$ with $(\lambda^1|_{\R^+} \otimes U_{S^2} \otimes \lambda^1)(A)<\infty$ 
the random vector 
\begin{align*}
\wn_{\mathrm{sph}}(A) = \int_A \frac{1}{ \sqrt{4\pi \kappa^2}} \,\bigl(\wn \circ \bm{\phi}_{\mathrm{sph}}^{-1}\bigr)(\d\kappa, \d\btheta,\d\omega)
\end{align*} 
is a $\C^\ell$-valued Gaussian white noise with structural measure $\lambda^1|_{\R^+} \otimes U_{S^2} \otimes \lambda^1$. 
Moreover, for every $\bm{G} \in L^2(\lambda^3 \otimes \lambda^1;\C^{d\times \ell})$ we have that 
\begin{align*}
\int_{\R^3\times \R} \bm{G}(\kap,\omega) \cdot \wn(\d\kap,\d\omega) = \int_{\R^+ \times S^2\times \R} \sqrt{4\pi\kappa^2} \, \bm{G}(\kappa\;\! \btheta,\omega) \cdot \wn_{\mathrm{sph}}(\d\kappa, \d\btheta,\d\omega).
\end{align*}
\end{corollary}

\begin{proof}
Observe that $\bm{\phi}_{\mathrm{sph}}^{-1} \colon \R^+ \times S^2 \times \R \to (\R^3\setminus\{0\})\times \R$ is given by $\bm{\phi}_{\mathrm{sph}}^{-1}(\kappa, \btheta,\omega) = (\kappa \;\! \btheta,\omega)$. By elementary analytic calculations using the classical change of variables formula it can be proven that the measure
\begin{align*}
\mu_{\mathrm{sph}}(A) = \int_A \frac{1}{4 \pi \kappa^2} \bigl( (\lambda^3 \otimes \lambda^1) \circ \bm{\phi}_{\mathrm{sph}}^{-1}\bigr)(\d \kappa, \d\btheta,\d \omega),
\end{align*}
for $A\in \B( \R^{+} \times S^2 \times \R)$ with $(\lambda^1|_{\R^+} \otimes U_{S^2} \otimes \lambda^1)(A)<\infty$, satisfies $\mu_{\mathrm{sph}} = \lambda^1|_{\R^+} \otimes U_{S^2} \otimes \lambda^1$. 
Thus, by Lemma \ref{lem:multiplication} and Proposition \ref{prop:pushforward} a) $\wn_{\mathrm{sph}}$ is a $\C^\ell$-valued Gaussian white noise with structural measure $\lambda^1|_{\R^+} \otimes U_{S^2} \otimes \lambda^1$, which satisfies the integral identity
\begin{align*}
\int_{\R^3 \times \R} \bm{G}(\kap,\omega) \cdot \wn(\d\kap,\d\omega) & = \int_{\R^+ \times S^2 \times \R} \bm{G}\bigl( \bm{\phi}_{\mathrm{sph}}^{-1}(\kappa,\btheta,\omega)\bigr) \cdot (\wn \circ \bm{\phi}_{\mathrm{sph}}^{-1})(\d \kappa, \d\btheta, \d\omega)
\\ & = \int_{\R^+ \times S^2\times \R} \sqrt{4\pi \kappa^2} \, \bm{G}(\kappa\;\!  \btheta ,\omega) \cdot  \wn_{\mathrm{sph}}(\d \kappa, \d\btheta,\d\omega).\qedhere
\end{align*}
\end{proof}

In Proposition~\ref{prop:linear_isometry_on_white_noise_integral} below we analyze a class of transformations induced by isometric linear operators. As a preparation, we first specify the action of a bounded linear operator on a white noise from the perspective of generalized random fields. To begin with, consider a $\C$-valued Gaussian white noise $\wnr$ on a Borel set $U \subset \R^m$ with structural measure $\mu \colon \B(U) \to [0,\infty]$. 
Observe that $\wnr$ can be interpreted as a generalized random field acting on test functions $g\in L^2(\mu;\C)$ in the sense that $\wnr$ may be identified with the bounded linear mapping $\wnr \colon L^2(\mu;\C) \to L^2(\Pr;\C)$ defined by 
\begin{equation*}
\wnr(g) := \int_U g(\x)\,\wnr(\d\x).
\end{equation*}
Let $V \subset \R^n$ be a further Borel set, let $\nu \colon \B(V) \to [0,\infty]$ be a $\sigma$-finite measure, and let the transformation  $\mathcal{T}' \colon [L^2(\mu;\C)]' \to [L^2(\nu;\C)]'$ be given as the adjoint of a bounded linear operator $\mathcal{T} \colon L^2(\nu;\C) \to L^2(\mu;\C)$. 
In this situation, the action of  $\mathcal{T}' $ is naturally extended to $\wnr$ by defining the generalized random field  $\mathcal{T}'(\wnr) \colon L^2(\nu;\C) \to L^2(\Pr;\C)$ through
\begin{equation*}
\bigl[\mathcal{T}'(\wnr) \bigr](g) : = \wnr\bigl(\mathcal{T}(g)\bigr) = \int_U [\mathcal{T}(g)](\x)\, \wnr(\d\x).
\end{equation*} 
In order to ensure consistency with the classical theory of generalized functions, we do \emph{not} identify the topologigcal dual spaces $[L^2(\mu;\C)]'$, $ [L^2(\nu;\C)]'$ with the spaces $L^2(\mu;\C)$, $L^2(\nu;\C)$ via the Riesz isomorphism at this point. 
The above definitions are extended to the multidimensional case in a componentwise sense. For instance, if  $\wn=(\wnr_1,\ldots,\wnr_\ell)$ is a $\C^\ell$-valued Gaussian white noise on $U$ with structural measure $\mu$, then 
\begin{equation*}
\mathcal{T}'(\wn) :=  (\mathcal{T}'(\wnr_1) ,\ldots,\mathcal{T}'(\wnr_\ell) ).
\end{equation*}

Using  these concepts, we obtain that the transformed generalized random field $\mathcal{T}'(\wn)$ is again a white noise in the sense of Definition~\ref{def:gaussian_white_noise} provided that the linear operator $\mathcal T$ is isometric up to a multiplicative factor. 

\begin{proposition}[Isometric linear transformations]
\label{prop:linear_isometry_on_white_noise_integral}
Let $U \subset \R^m$ and $V \subset \R^n$ be Borel sets, let $\mu \colon \B(U) \to [0,\infty]$ and $\nu\colon \B(V) \to [0,\infty]$ be $\sigma$-finite measures, let $\wn$ be a $\C^\ell$-valued Gaussian white noise on $U$ with structural measure $\mu$, and let $\mathcal{T} \colon L^2(\nu;\C) \to L^2(\mu;\C)$ be a bounded linear operator such that $\alpha\mathcal{T}$ is isometric for some $\alpha\in\R^+$. 
Then the mapping $\mathcal{T}'(\wn)$ that assigns to every  $A \in \B(V)$ with $\nu(A)<\infty$ the random vector 
\begin{align*}
\bigl[\mathcal{T}'(\wn)\bigr](A) := \bigl[\mathcal{T}'(\wn)\bigr](\indicator{A}) = \int_U \bigl[\mathcal{T}(\indicator{A})\bigr](\x)\,\wn(\d\x)
\end{align*} 
is a $\C^\ell$-valued Gaussian white noise on $V$ with structural measure $\alpha^{-2}\nu$. 
Moreover, for every $\bm{G} \in L^2(\nu;\C^{d\times \ell})$ it holds that
\begin{align*}
\int_V \bm{G}(\y)\cdot \bigl[\mathcal{T}'(\wn)\bigr](\d\y) = \int_U \bigl[\mathcal{T}(\bm{G})\bigr](\x) \cdot \wn(\d\x),
\end{align*}
where $\mathcal{T}'(\bm{G}) = \bigl( \mathcal{T}'(G_{j,k})\bigr)_{j,k}$ for $\bm{G}=(G_{j,k})_{j,k}$.
\end{proposition}

\begin{proof}
We first treat the one-dimensional case. Observe that $\mathcal{T}' (\wnr) \colon \B_0(V) \to L^2(\Pr;\C)$ is obviously centered and Gaussian. The additivity property of $\Re[\mathcal{T}' (\wnr)]$ and $\Im[\mathcal{T}' (\wnr)]$ follows immediately from the linearity of $\mathcal{T}$. Moreover, for any sets $A,B \in \B_0(V)$ we use the fact that
\begin{align*}
\Re\bigl[\mathcal{T}' (\wnr)\bigr](A) & = \int_U \Re \;\!\mathcal{T}(\indicator{A})(\x) \,\Re\;\!\wnr(\d\x) - \int_U \Im\;\! \mathcal{T}(\indicator{A})(\x) \,\Im\;\!\wnr(\d\x),
\\ \Im\bigl[\mathcal{T}' (\wnr)\bigr](B) & = \int_U \Re \;\!\mathcal{T}(\indicator{B})(\x) \,\Im\;\!\wnr(\d\x) + \int_U \Im \;\!\mathcal{T}(\indicator{B})(\x) \,\Re\;\!\wnr(\d\x),
\end{align*}
as well as the isometric property \eqref{eq:isometric_property_1} of the stochastic integral \wrt $\Re\;\!\wnr$ and $\Im\;\!\wnr$, and the isometric property of $\alpha\mathcal{T}$ combined with the independence of $\Re\;\!\wnr$ and $\Im\;\!\wnr$ to obtain
\begin{align*}
& \E \Bigl[\Re\bigl[\mathcal{T}' (\wnr)\bigr](A) \, \Im\bigl[\mathcal{T}' (\wnr)\bigr](B) \Bigr]  
\\ & \qquad = \frac12 \int_U \Re \;\!\mathcal{T}(\indicator{A})(\x) \,\Im \;\!\mathcal{T}(\indicator{B})(\x) - \Im \;\!\mathcal{T}(\indicator{A})(\x) \,\Re \;\!\mathcal{T}(\indicator{B})(\x) \, \mu(\d\x)
\\ & \qquad = - \frac{1}{2\alpha^2}\Im\! \int_U \alpha\mathcal{T}(\indicator{A})(\x) \,\overline{\alpha\mathcal{T}(\indicator{B})}(\x)  \,\mu(\d\x) = -\frac{1}{2\alpha^2}\Im \!\int_V \indicator{A}(\y) \, \overline{\indicator{B}}(\y)  \,\nu(\d\y) = 0.
\end{align*}
Since $\mathcal{T}' (\wnr)$ is a Gaussian process, this implies the independence of $\Re[\mathcal{T}' (\wnr)]$ and $\Im[\mathcal{T}' (\wnr)]$. Similarly
\begin{equation}\label{eq:E_ReA_ReB}
 \E \Bigl[ \bigl| \Re\bigl[\mathcal{T}' (\wnr)\bigr](A) \bigr|^2 \Bigr] = \E \Bigl[ \bigl| \Im\bigl[\mathcal{T}' (\wnr)\bigr](A) \bigr|^2 \Bigr] = \frac{1}{2\alpha^2} \nu(A).
\end{equation}
Using that $\alpha\mathcal{T}$ is linear and isometric, the claimed integral identity can be proven in the usual way by first showing the identity for simple functions and then apply an approximation argument for the general case, where we employ the fact that $f_n \to f$ in $L^2(\nu;\C)$ if, and only if, $\mathcal{T}(f_n) \to \mathcal{T}(f)$ in $L^2(\mu;\C)$.

In the multidimensional case a similar argument shows that $\mathcal{T}' (\wn) \colon \B_0(V) \to L^2(\Pr;\C^{\ell})$ is a centered Gaussian process with independent components and independent real and imaginary parts satisfying the additivity property and \eqref{eq:E_ReA_ReB}. The integral identity follows by a componentwise application of the one-dimensional case.
\end{proof}

\subsection{Differentiability properties}\label{subsec:appendix_differentiability}

Here we briefly recall the concepts of mean-square continuity and mean-square differentiability before presenting results concerning the mean-square differentiability of general second-order random fields and of random fields given in terms of white noise integrals. 

Let $\rf = (\rf(\x))_{\x\in\R^n}$ be an $\R^d$-valued, second-order random field and let $(\bm{e}_j)_{j=1}^n$ denote the standard basis of $\R^n$. The random field $\rf$ is said to be \emph{mean-square continuous} if for all $\x \in \R^n$ it holds that $\lim_{\y \to \bm{0}} \E [\| \rf(\x+\y) - \rf(\x) \|^2] = 0$. 
Moreover, $\rf$ is called \emph{partially mean-square differentiable} in direction $\bm{e}_j$ at the point $\x \in \R^n$ if there exists an $\R^d$-valued, square-integrable random variable $\partial_{x_j} \rf(\x)$ such that
\begin{align}\label{eq:def_ms_diff}
\lim_{h \rightarrow 0}\E \bigg[ \Bigl\|	\frac{1}{h}\bigl(\rf(\x+h\bm{e}_j) - \rf(\x)\bigr) - \partial_{x_j}\rf(\x) \Bigr\|^2\bigg] = 0.
\end{align}
The random variable $\partial_{x_j} \rf(\x)$ is uniquely determined within the space $L^2(\Pr;\R^d)$, i.e., up to $\Pr$-almost sure equality, and is called the \emph{partial mean-square derivative} of $\rf$ in direction $\bm{e}_j$ at the point $\x$.
Furthermore, $\rf$ is said to be \emph{continuously mean-square differentiable} if for all $j \in \{1,\ldots,n\}$  the partial mean-square derivative $\partial_{x_j} \rf(\x)$ exists at every point $\x \in \R^n$ and the random field $\partial_{x_j} \rf=(\partial_{x_j} \rf(\x))_{\x\in\R^n}$ is mean-square continuous.

Partial mean-square differentiability can be characterized in terms of the covariance function.

\begin{lemma}\label{lem:cov_ms-diff}
Let $\rf = (\rf(\x))_{\x\in\R^n}$ be an $\R^d$-valued, centered, second-order random field and let $\bm{C}\colon \R^n \times \R^n \rightarrow \R^{d\times d}$ be the covariance function given by $\bm C(\x,\xalt)=\E[\rf(\x) \otimes \rf(\xalt)]$. 
Then the following assertions hold: 
\begin{enumerate}[label= \emph{\textbf{\alph*)}}]
\item The random field $\rf$ is mean-square differentiable in direction $\bm{e}_j$ at the point $\x \in \R^n$ if, and only if, the limit
\begin{align}\label{eq:specific_derivative_of_cov_function}
\lim_{h,\tilde h \to 0} \frac{1}{h \tilde h} \bigl( \bm{C}(\x + h \bm{e}_j, \x + \tilde h \bm{e}_j) - \bm{C}(\x + h \bm{e}_j, \x) - \bm{C}(\x, \x + \tilde h \bm{e}_j) + \bm{C}(\x, \x) \bigr)
\end{align} 
exists in $\R^{d\times d}$. Moreover, if $\rf$ is mean-square differentiable in direction $\bm{e}_j$ at the point $\x \in \R^n$ and mean-square differentiable in direction $\bm{e}_k$ at the point $\xalt \in \R^n$, then the mixed partial derivative $\partial_{x_j} \partial_{\tilde x_k} \bm{C}(\x,\xalt)$ exists and it holds that
\begin{align}\label{eq:derivative_of_cov_function}
\E\bigl[\partial_{x_j}\rf(\x) \otimes \partial_{\tilde x_k}\rf(\xalt)\bigr] = \partial_{x_j} \partial_{\tilde x_k} \bm{C}(\x,\xalt).
\end{align}
\item 
If the mapping $(h,\tilde h)\mapsto \bm{C}(\x + h \bm{e}_j, \x + \tilde h \bm{e}_j)$ is such that its first order partial derivatives and mixed second-order partial derivatives exist and are continuous in a neighborhood of $(0,0)\in\R\times \R$, 
then $\rf$ is mean-square differentiable in direction $\bm{e}_j$ at the point~$\x$.
\end{enumerate}
\end{lemma}

We remark that the existence of the limit \eqref{eq:specific_derivative_of_cov_function} in item a) of Lemma~\ref{lem:cov_ms-diff} ensures the existence of the mixed partial derivative $\partial_{x_j} \partial_{\tilde x_j} \bm{C}(\x,\xalt)|_{\tilde\x=\x}$. 
In order to conclude the former from the latter,  additional regularity assumptions are imposed in item b). 

\begin{proof}[Proof of Lemma~\ref{lem:cov_ms-diff}]
The first assertion of part a) is a consequence of an $\R^{d\times d}$-valued version of \cite[Lemma~3 in Section~I.1]{GS74} and the identity
\begin{equation*}
\begin{aligned}
& \E \Bigl[\frac{1}{h}\bigl(\rf(\x+h\bm{e}_j) - \rf(\x)\bigr) \otimes \frac{1}{\tilde h}\bigl(\rf(\x+\tilde h\bm{e}_j) - \rf(\x)\bigr)\Bigr] 
\\ & \qquad = \frac{1}{h \tilde h} \bigl( \bm{C}(\x + h \bm{e}_j, \x + \tilde h \bm{e}_j) - \bm{C}(\x + h \bm{e}_j, \x) - \bm{C}(\x, \x + \tilde h \bm{e}_j) + \bm{C}(\x, \x) \bigr).
\end{aligned}
\end{equation*}
Assuming that $\rf$ is mean-square differentiable in direction $\bm{e}_j$ at the point $\x \in \R^n$ and mean-square differentiable in direction $\bm{e}_k$ at the point $\xalt \in \R^n$, property \eqref{eq:derivative_of_cov_function} follows immediately from 
the definition of partial mean-square differentiability \eqref{eq:def_ms_diff}; cf., e.g., \cite[Corollary~2.10]{Kru01}. 

Regarding assertion b) observe that
\begin{equation*}\label{eq:proof_lemma_differentiability}
\begin{aligned}
& \frac{1}{h \tilde h} \bigl( \bm{C}(\x + h \bm{e}_j, \x + \tilde h \bm{e}_j) - \bm{C}(\x + h \bm{e}_j, \x) - \bm{C}(\x, \x + \tilde h \bm{e}_j) + \bm{C}(\x, \x) \bigr)
\\ & \qquad = \frac{1}{h \tilde h} \int_0^h \int_0^{\tilde h} \partial_{s\,} \partial_{\tilde s\,}  \bm{C}(\x + s \bm{e}_j, \x + \tilde s \bm{e}_j) \,\d \tilde s \,\d s.
\end{aligned}
\end{equation*}
Due to the continuity of the mixed partial derivative, the term on the right hand side converges to $\partial_{x_j} \partial_{\tilde x_j} \bm{C}(\x,\tilde\x)|_{\tilde \x=\x}$ as $h,\tilde h\to 0$, hence part~b) follows from part a).
\end{proof}

In the special case of a random field given in terms of white noise integrals, partial differentiation and stochastic integration can be interchanged under certain assumptions. The next lemma specifies these assumptions.

\begin{lemma}\label{lem:ms-diff}
Let $U \subset \R^m$ be a Borel set, let $\mu \colon \B(U) \to [0,\infty]$ be a $\sigma$-finite measure, and let $\wn$ be a $\C^\ell$-valued Gaussian white noise on $U$ with structural measure $\mu$. 
Moreover, let the $\R^d$-valued random field $\rf=(\rf(\x))_{\x\in\R^n}$ be such that for every $\x\in\R^n$ it holds $\Pr$-almost surely that
\begin{align*}
\rf(\x) = \Re\! \int_{U} \bm{G}(\x,\y) \cdot \wn(\d \y), 
\end{align*}
where $\bm{G} \colon \R^n \times U \to \C^{d\times \ell}$ is a measurable function satisfying the following conditions for every $x\in\R^n$, $j\in\{1,\ldots,n\}$: 
\begin{enumerate}[label= \emph{\textbf{\alph*)}}]
\item 
It holds that $\bm{G}(\x,\bdot) \in L^2(\mu;\C^{d\times \ell})$.
\item 
For $\mu$-almost all $\y \in U$  the partial derivative $\partial_{x_j}\bm{G}(\x,\y)$ exists and it holds that $\partial_{x_j}\bm{G}(\x,\bdot) \in L^2(\mu;\C^{d\times \ell})$. 
\item 
It holds that 
\begin{align*}
\lim_{h \rightarrow 0} \int_{U} \Bigl\| \frac{1}{h}\bigl(\bm{G}(\x + h\bm{e}_j,\y)-\bm{G}(\x,\y)\bigr) - \partial_{x_j} \bm{G}(\x,\y)\Bigr\|^2 \mu(\d \y) = 0.
\end{align*}
\end{enumerate}
Then for every $\x\in\R^n$, $j \in \{1,...,n\}$ the partial mean-square derivative $\partial_{x_j}\rf(\x)$ exists and satisfies
\begin{align*}
\partial_{x_j} \rf(\x)  =  \Re \!\int_{U} \partial_{x_j} \bm{G}(\x,\y) \cdot \wn(\d \y).
\end{align*}
\end{lemma}

\begin{proof}
Throughout this proof let $j \in \{1,...,n\}$ be fixed and let the random field $\v=(\v(\x))_{\x\in\R^n}$ be such that for every $\x\in\R^n$ it holds $\Pr$-almost surely that $\v(\x) = \Re\! \int_U \partial_{x_j}\bm{G}(\x,\y) \cdot \wn(\d\y)$. 
Note that conditions a) and b) ensure that the white noise integrals defining the random fields $\rf$ and $\v$ exist. 
Furthermore, observe that the linearity of the stochastic integral implies for every $\x\in\R^n$, $h\in\R^+$ that
\begin{align*}
\frac{1}{h} \bigl( \rf(\x + h \bm{e}_j) - \rf(\x) \bigr) - \v(\x) = \Re \!\int_U \Bigl( \frac{1}{h} \bigl(\bm{G}(\x + h\bm{e}_j,\y)-\bm{G}(\x,\y)\bigr) - \partial_{x_j} \bm{G}(\x,\y)\Bigr) \cdot \wn(\d\y).
\end{align*}
This and the isometric property \eqref{eq:isometric_property_real_part} of the white noise integral yield 
\begin{align*}
\E \biggl[ \Bigl\|\frac{1}{h} \;\!\! \bigl( \rf(\x + h \bm{e}_j) - \rf(\x) \bigr) - \v(\x)\Bigr\|^2\biggr]	= \frac{1}{2} \;\!\! \int_{U} \;\!\! \Big\| \frac{1}{h} \bigl(\bm{G}(\x + h\bm{e}_j,\y)-\bm{G}(\x,\y)\bigr) - \partial_{x_j} \bm{G}(\x,\y)\Big\|^2 \!\mu(\d \y).
\end{align*}
Condition c) and the definition of partial mean-square differentiability \eqref{eq:def_ms_diff} thus imply that the partial mean-square derivative $\partial_{x_j} \rf(\x)$ exists and equals $\v(x)$ in $L^2(\Pr;\R^d)$. 
\end{proof}


\section{Ergodicity of Gaussian random vector fields}
\label{sec:ergodicity}

In this section we establish in Proposition \ref{prop:mean_ergo_thm} and Corollary \ref{cor:ergo} suitable variants of Wiener's multiparameter mean ergodic theorem \cite{Wie39} which are tailor-made for the homogeneous and inhomogeneous turbulence models considered in this article and lay the ground for the ergodicity results presented in Proposition \ref{prop:ergodicity_hom} and Theorem \ref{thm:ergodicity_inhom}. Related results can be found, e.g., in \cite{BE72, Bec81, Tay18}.

\begin{lemma}[Abstract mean ergodic theorem]\label{lem:abstract_mean_ergo_thm}
Let $H$ be a (real or complex) Hilbert space, let $(\U(\z))_{\z\in\R^n} \subset L(H)$ be a strongly continuous group of unitary operators, 
set $K = \{ \varphi \in H \colon \U(\z)\varphi = \varphi$ for all $\z \}$, and let  $P_K \in L(H)$ denote the orthogonal projection in $H$ onto the closed subspace $K$. Let $\varphi_j\in H$, $j\in\N\cup\{\infty\}$, be such that $\lim_{j\to\infty} \varphi_j = \varphi_\infty$ and let $B_j \subset \R^n$, $j \in \N$, be Borel sets which satisfy that $0<\lambda^n(B_j)<\infty$ and
\begin{equation}\label{eq:assump_ergo_abstract}
\forall\, \z \in \R^n \colon \lim_{j\to\infty} \frac{\lambda^n\bigl( B_j \triangle (\z + B_j)
\bigr)}{\lambda^n(B_j)} = 0,
\end{equation}
where  $B_j \triangle (\z + B_j)$ denotes the symmetric difference of the sets $B_j$ and $\z+B_j=\{\z+\bm{b}\colon \bm{b}\in B_j\}$. Then it holds that
\begin{equation*}
\lim_{j\to\infty} \mint{-}_{B_j} \U(\z)\varphi_j \,\d\z = P_K \varphi_\infty.
\end{equation*}
\end{lemma}

\begin{proof}
Throughout this proof let $R$ be the closure in $H$ of the linear span of $\bigcup_{\z\in\R^n} \text{Range}(I-\U(\z))$. Using that $(\U(\z))_{\z \in \R^n}$ is a group of unitary operators, it is straightforward to show that $K = R^\bot$, which implies the orthogonal decomposition $H = K \oplus {R}$. This yields 
\begin{equation}\label{eq:ergo_thm_sum}
 \mint{-}_{B_j} \U(\z)\varphi_j \,\d\z = \mint{-}_{B_j} \U(\z) P_K \varphi_\infty \,\d\z + \mint{-}_{B_j}  \U(\z) P_{{R}} \varphi_\infty \,\d\z + \mint{-}_{B_j} \U(\z)(\varphi_j - \varphi_\infty) \,\d\z,
\end{equation}
where $P_R \in L(H)$ denotes the orthogonal projection in $H$ onto $R$. 
By definition of $K$ the first summand in \eqref{eq:ergo_thm_sum} equals $\mint{-}_{B_j} P_K \varphi_\infty \,\d\z = P_K \varphi_\infty$. 
Moreover, the fact that the operators $\U(\z)$ are unitary implies that the third summand in \eqref{eq:ergo_thm_sum} satisfies the estimate $\| \mint{-}_{B_j} \U(\z)(\varphi_j - \varphi_\infty) \,\d\z \|_H \leq \| \varphi_j - \varphi_\infty \|_H$ and thus converges to zero 
as $j \to \infty$. 
It remains to show that the second summand in \eqref{eq:ergo_thm_sum} converges to zero as $j \to \infty$. 
To see this, observe that for any $\z_0 \in \R^n$ and any $\varphi \in \text{Range}(I-\U(\z_0))$ there exists a $\psi \in H$ such that $\varphi = (I - \U(\z_0))\psi$, leading to
\begin{equation*}
\begin{aligned}
\Bigl\| \mint{-}_{B_j} \! \U(\z) \varphi \,\d \z \Bigr\|_H 
&= 
\Bigl\| \mint{-}_{B_j} \! \U(\z) \psi \,\d \z - \mint{-}_{B_j} \! \U(\z + \z_0) \psi \,\d \z \Bigr\|_H  
\leq
\frac{1}{\lambda^n(B_j)} \int_{B_j \triangle (\z_0 + B_j)} \!\|\U(\z)  \psi\|_H \, \d \z.
\end{aligned}
\end{equation*}
The fact that the operators $\U(\z)$ are unitary and \eqref{eq:assump_ergo_abstract} hence ensure that $ \mint{-}_{B_j} \U(\z) \varphi \,\d \z$ converges to zero as $j\to\infty$. 
This and a density argument in the subspace $R$ yield the convergence of the second summand in \eqref{eq:ergo_thm_sum}. 
\end{proof}

In the next lemma we prepare the application of the abstract ergodic result from  Lemma~\ref{lem:abstract_mean_ergo_thm} in the context of random fields in Proposition~\ref{prop:mean_ergo_thm} below. 
We point out the notational difference between the random field $\w$ and its possible realizations  $\bom$ in the space $\mathcal O$ appearing in the formulation of the considered setting. 

\begin{lemma}\label{lem:ergo}
Let $\w = (\w(\x,t))_{(\x,t) \in \R^3 \times \R}$ be an $\R^{d}$-valued, centered, mean-square continuous, homogeneous Gaussian random field  
and let $\bm{C} \colon \R^3 \times \R \to \R^{d\times d}$ be the covariance function of $\w$ given by $\bm{C}(\y,s) = \E[\w(\x+\y,t+s) \otimes \w(\x,t)]$. 
Let $\mathcal{O}$ be the set of all functions from $\R^3 \times \R$ to $\R^{d}$, endowed with the $\sigma$-algebra $\mathcal{A}$ generated by the point evaluation mappings $\mathcal O\ni\bom\mapsto\bom(\x,t)\in\R^{d}$, $(\x,t)\in \R^3\times\R$, 
and let $\nu$ be the distribution of $\w$ on the measurable space $(\mathcal{O},\mathcal{A})$, where $\w$ is interpreted as an $\mathcal O$-valued random variable. 
Moreover, let the translation operators $\tau_{(\y,s)} \colon \mathcal{O} \to \mathcal{O}$, $(\y,s) \in \R^3 \times \R$, be defined by $(\tau_{(\y,s)}\bom)(\x,t) = \bom(\x+\y,t+s)$, $\bom \in \mathcal{O}$,
and let $H = L^2(\nu;\R)$. 
Then the following assertions hold: 
\begin{enumerate}[label= \emph{\textbf{\alph*)}}]
\item If we define $\U(\y,s) \colon H \to H$, $(\y,s) \in \R^3 \times \R$, by $\U(\y,s)\varphi = \varphi \circ \tau_{(\y,s)}$, then the family $(\U(\y,s))_{(\y,s) \in \R^3 \times \R}$ constitutes a strongly continuous group of unitary operators on $H$, and in this case the condition
\begin{equation*}
\lim_{\| (\y,s) \| \to \infty} \| \bm{C}(\y,s) \| = 0
\end{equation*}
implies that the space $K = \{ \varphi \in H \colon \U(\y,s) \varphi = \varphi$ for all $(\y,s) \}$ consists precisely of those functions $\varphi$ that are constant $\nu$-almost everywhere.
\item If we define $\U(\y) \colon H \to H$, $\y \in \R^3$, by $\U(\y)\varphi = \varphi \circ \tau_{(\y,0)}$, then the family $(\U(\y))_{\y\in \R^3}$ constitutes a strongly continuous group of unitary operators on $H$, and in this case the condition
\begin{equation*}
\forall\, s \in\R \colon \lim_{\| \y\| \to \infty} \| \bm{C}(\y,s) \| = 0
\end{equation*}
implies that the space $K = \{ \varphi \in H \colon \U(\y) \varphi = \varphi$ for all $\y \}$ consists precisely of those functions $\varphi$ that are constant $\nu$-almost everywhere.
\item If we define $\U(s) \colon H \to H$, $s \in \R$, by $\U(s)\varphi = \varphi \circ \tau_{(\bm{0},s)}$, then the family $(\U(s))_{s \in \R}$ constitutes a strongly continuous group of unitary operators on $H$, and in this case the condition
\begin{equation*}
\forall \, \y \in \R^3 \colon \lim_{|s| \to \infty} \| \bm{C}(\y,s) \| = 0
\end{equation*}
implies that the space $K = \{ \varphi \in H \colon \U(s) \varphi = \varphi$ for all $s \}$ consists precisely of those functions $\varphi$ that are constant $\nu$-almost everywhere.
\end{enumerate}
\end{lemma}

\begin{proof}
We first consider item a) and verify that $(U(\y,s))_{(\y,s) \in \R^3 \times \R}$ is a strongly continuous group of unitary operators, which implies the corresponding assertions in items b) and c). 
Clearly, the group property is a direct consequence of the definition of the operators $U(\y,s)$ via translation. Moreover, the fact that $\w$ is a homogeneous random field and the uniqueness theorem for measures ensure that the transformations $\tau_{(\y,s)} \colon \mathcal{O} \to \mathcal{O}$ are measure preserving, i.e., $\nu \circ \tau_{(\y,s)}^{-1} = \nu$, hence the operators $U(\y,s)$ are unitary. 
Next observe that an application of the monotone class theorem and standard approximation arguments yield that a dense subspace of $L^2(\nu;\R)$ is given by the class of functions of the form
\begin{equation}\label{eq:ergo_lemma_density}
\varphi(\bom) = g\bigl( \bom(\x_1,t_1), \ldots, \bom(\x_\ell, t_\ell) \bigr), \quad \bom \in \mathcal{O},
\end{equation}
where $\ell \in \N$, $(\x_1,t_1),\ldots,(\x_\ell,t_\ell) \in \R^3 \times \R$, and $g \in C_c^\infty(\R^{{d}\ell},\R)$.
Here and below we identify $\R^{d\ell}$ with $(\R^d)^\ell$. In order to establish the strong continuity of the group of unitary operators $(U(\y,s))_{(\y,s) \in \R^3 \times \R}$, it is thus sufficient to verify the continuity of $(\y,s) \mapsto \varphi \circ \tau_{(\y,s)}$ from $\R^3 \times \R$ to $L^2(\nu;\R)$ for functions $\varphi$ of type \eqref{eq:ergo_lemma_density}. The latter follows from the mean-square continuity of $\w$ and the estimate
\begin{equation}\label{eq:est_B2}
\begin{aligned}
\| \varphi \circ \tau_{(\y,s)} - \varphi \|_{L^2(\nu;\R)} & = \| \varphi(\tau_{(\y,s)}\w) - \varphi(\w) \|_{L^2(\Pr;\R)} 
\\ & \leq C \sum_{j=1}^\ell \| \w(\x_j + \y, t_j + s) - \w(\x_j,t_j) \|_{L^2(\Pr;\R^d)},
\end{aligned}
\end{equation}
where we use the notation $\varphi(\tau_{(\y,s)}\w) := \varphi \circ \tau_{(\y,s)} \circ \w$ and $\varphi(\w) := \varphi \circ \w$ based on interpreting $\w$ as a mapping from $\Omega$ to $\mathcal{O}$, and $C$ is the Lipschitz constant associated with the function $g$ in \eqref{eq:ergo_lemma_density}. 

Next we focus on item b) and prove the statement involving the covariance condition and the subspace $K = \{ \varphi \in H \colon U(\y) \varphi = \varphi$ for all $\y \}$ of $H$, where $U(\y)\varphi = \varphi \circ\tau_{(\y,0)}$. The corresponding assertions in items a) and c) are proven analogously. To begin with, observe that Lemma \ref{lem:abstract_mean_ergo_thm} implies that $K$ consists only of $\nu$-almost everywhere constant functions provided that
\begin{align}\label{eq:ergo_lem_sufficient_condition}
\forall\,\varphi,\psi \in L^2(\nu;\R)\colon\,
\lim_{R\to \infty} \mint{-}_{\| \y \| \leq R} \bigl\langle \varphi \circ \tau_{(\y,0)}, \psi \bigr\rangle_{L^2(\nu;\R)} \d \y = \Bigl( \int_\mathcal{O} \varphi \, \d\nu \Bigr) \Bigl( \int_\mathcal{O} \psi \,\d\nu \Bigr).
\end{align}
As the dependence of the integrals in \eqref{eq:ergo_lem_sufficient_condition} on $\varphi$ and $\psi$ is continuous, it suffices to verify the identity in \eqref{eq:ergo_lem_sufficient_condition} for all $\varphi, \psi$ in some dense subspace of $L^2(\nu;\R)$, such as the class of functions of the form \eqref{eq:ergo_lemma_density}. Considering $\varphi(\bom) = g( \bom(\x_1,t_1), \ldots, \bom(\x_\ell, t_\ell) )$ and $\psi(\bom) = h( \bom(\x_1,t_1), \ldots, \bom(\x_\ell, t_\ell) )$ with $g,h \in C_c^\infty(\R^{d\ell},\R)$, the Fourier inversion formula implies the representation
\begin{align*}
( \varphi \circ \tau_{(\y,0)} )(\bom) = \int_{\R^{d\ell}} \exp \Bigl\{ i \sum_{j = 1}^\ell \bom(\x_j + \y, t_j) \cdot \bxi_j \Bigr\} (\mathcal{F}g)(\bxi) \,\d\bxi
\end{align*}
and an analogous identity for $\psi$. Using Fubini's theorem and the dominated convergence theorem, we thus obtain that \eqref{eq:ergo_lem_sufficient_condition} holds if
for all $\ell \in \N$, $(\x_1,t_1), \ldots, (\x_\ell, t_\ell) \in \R^3 \times \R$, $\bxi = (\bxi_1,\ldots,\bxi_\ell) \in \R^{d\ell}$, and $\bm{\eta} = (\bm{\eta}_1,\ldots,\bm{\eta}_\ell) \in \R^{d\ell}$ we have
\begin{equation}\label{eq:ergo_lem_sufficient_after_Fubini}
\begin{aligned}
\lim_{R\to \infty} \mint{-}_{\| \y \| \leq R} &\E \Bigl[ \exp \Bigl\{ i \sum_{j = 1}^\ell \w(\x_j + \y, t_j) \cdot \bxi_j \Bigr\} \, \exp \Bigl\{ i \sum_{k = 1}^\ell \w(\x_k, t_k) \cdot \bm{\eta}_k \Bigr\}  \Bigr]\, \d\y
\\ &\; = \E \Bigl[ \exp \Bigl\{ i \sum_{j = 1}^\ell \w(\x_j, t_j) \cdot \bxi_j \Bigr\} \Bigr] \; \E \Bigl[ \exp \Bigl\{ i \sum_{k = 1}^\ell \w(\x_k, t_k) \cdot \bm{\eta}_k \Bigr\}  \Bigr].
\end{aligned}
\end{equation}
Furthermore, note that the specific structure of characteristic functions of Gaussian random vectors ensures that \eqref{eq:ergo_lem_sufficient_after_Fubini} is implied by the condition that
for all $\ell \in \N$, $(\x_1,t_1), \ldots, (\x_\ell, t_\ell) \in \R^3 \times \R$, $\bxi = (\bxi_1,\ldots,\bxi_\ell) \in \R^{d\ell}$, and $\bm{\eta} = (\bm{\eta}_1,\ldots,\bm{\eta}_\ell) \in \R^{d\ell}$ it holds that
\begin{equation}\label{eq:ergo_lem_convergence}
\lim_{R \to \infty} \mint{-}_{\| \y \| \leq R} \exp \Bigl\{ -\sum_{j,k = 1}^\ell \bxi_j \cdot \bm{C}(\x_j + \y - \x_k, t_j - t_k) \cdot \bm{\eta}_k \Bigr\} \,\d \y = 1.
\end{equation}
Finally, observe that the estimates $|e^r - 1| \leq |r|e^C$ and $\| \bm{C}(\y,s) \| \leq \sqrt{3} \| \bm{C}(\bm{0},0)\|$ for $C \in \R_0^+$, $r \in [-C,C]$, and $(\y,s) \in \R^3 \times \R$ establish that \eqref{eq:ergo_lem_convergence} holds provided that
\begin{equation*}
\forall\,s \in \R\colon\,
\lim_{R \to \infty} \mint{-}_{\| \y \| \leq R} \| \bm{C}(\y,s) \| \,\d \y = 0.
\end{equation*}
This, the change of variables formula, and the dominated convergence theorem lead to the covariance condition in item b).
\end{proof}

\begin{proposition}[Mean ergodic theorem]\label{prop:mean_ergo_thm}
Assume the setting of Lemma~\ref{lem:ergo} and suppose that the covariance function $\bm{C}$ of $\w$ admits a spectral representation of the form 
\begin{equation}\label{eq:cov_mean_ergo}
\bm{C}(\y,s) = \Re\! \int_{\R^3 \times \R} \exp\Bigl\{i \bigl( \kap \cdot \y + \omega s \bigr) \Bigr\} \bm{F}(\kap, \omega) \,\d (\kap,\omega),
\end{equation}
$(\y,s) \in \R^3 \times \R$, for some integrable function $\bm{F} \colon \R^3 \times \R \to \C^{d\times d}$. 
In addition, let $\varphi \in L^2(\nu;\R)$, let $(\mathcal{R}_j)_{j \in \N}$, $(\mathcal{R}'_j)_{j \in \N} \subset \R_0^+$ and $L,L' \in [0,\infty]$ be such that $\lim_{j\to\infty} \mathcal{R}_j = L$ and $\lim_{j\to\infty} \mathcal{R}'_j = L'$, and assume that at least one of the limit values $L, L'$ equals infinity. 
Moreover, let $((\x_j,t_j))_{j\in \N} \subset \R^3 \times \R$ be an arbitrary sequence. 
Then, for every $(\y,s) \in \R^3 \times \R$ the random variable $\varphi(\tau_{(\y,s)} \w) := \varphi \circ \tau_{(\y,s)} \circ \w$ is an element of $L^2(\Pr;\R)$, the mapping $\R^3 \times \R \ni (\y,s) \mapsto \varphi(\tau_{(\y,s)} \w) \in L^2(\Pr;\R)$ is continuous, and we have that
\begin{equation}\label{eq:limit_mean_ergo}
\lim_{j\to \infty} \mint{-}_{B_{\mathcal{R}_j}^{(1)}(0)} \mint{-}_{B_{\mathcal{R}'_j}^{(3)}(\bm{0})} \varphi\bigl( \tau_{(\x_j+\y, t_j + s)} \w \bigr) \,\d \y \, \d s = \E \bigl[\varphi(\w)\bigr]
\end{equation}
as a limit in $L^2(\Pr;\R)$. 
The average integrals are understood as $L^2(\Pr;\R)$-valued Bochner integrals if the radius of the respective domain of integration is non-zero and as point evaluations at $\y=0$ or $s=0$ otherwise.
\end{proposition}

\begin{proof}
Observe that the change of variables formula for image measures ensures that the mapping $\psi\mapsto\psi(\w)=\psi\circ\w$ is an isometry from $L^2(\nu;\R)$ to $L^2(\Pr;\R)$. 
This and the strong continuity assertion in Lemma~\ref{lem:ergo}~a) imply that the mapping $(\y,s)\mapsto\varphi(\tau_{(\y,s)}\w)$ is continuous from $\R^3\times\R$ to $L^2(\Pr;\R)$. 
Next note that we may assume without loss of generality that the sequence $((\x_j,t_j))_{j\in\N}\subset\R^3\times\R$ is identically zero. 
Indeed, for every fixed $j\in\N$ the distribution of the double integral in \eqref{eq:limit_mean_ergo} does not depend on the choice of $(\x_j,t_j)$ due to the homogeneity of $\w$, and the limit value on the right hand side of \eqref{eq:limit_mean_ergo} is non-random, so that the claimed $L^2(\Pr;\R)$ convergence is invariant \wrt changing the choice of the sequence $((\x_j,t_j))_{j\in\N}$. 
We are going to establish the convergence  in \eqref{eq:limit_mean_ergo} by applying Lemma~\ref{lem:abstract_mean_ergo_thm} with $H=L^2(\nu;\R)$. 
To this end, observe that \eqref{eq:cov_mean_ergo} and the Riemann-Lebesgue lemma yield that all of the covariance conditions in Lemma~\ref{lem:ergo} are fulfilled. 
With regard to the choice of $n$, $U(\z)$, $\varphi_j$, and $B_j$ in Lemma~\ref{lem:abstract_mean_ergo_thm}, we distinguish the cases 
a) $L=L'=\infty$, b) $L<\infty\,\land\,L'=\infty$, and c) $L=\infty\,\land\,L'<\infty$. 
In case a) we adopt the specifications in Lemma~\ref{lem:ergo}~a) and choose $n=3+1$, $\varphi_j=\varphi$, and $B_j=B_{\mathcal{R}'_j}^{(3)}(\bm{0})\times B_{\mathcal{R}_j}^{(1)}(0)$. 
The convergence in \eqref{eq:limit_mean_ergo} then follows by taking into account the isometry property of the mapping $L^2(\nu;\R)\ni\psi\mapsto\psi(\w)\in L^2(\Pr;\R)$, Fubini's theorem for Bochner integrals, and the identity $P_K\varphi=\int_{\mathcal O}\varphi\,\d\nu= \E [\varphi(\w)]$. 
In case b) we apply Lemma~\ref{lem:abstract_mean_ergo_thm} by adopting the specifications in Lemma~\ref{lem:ergo}~b) and choosing $n=3$, $B_j=B_{\mathcal{R}'_j}^{(3)}(\bm{0})$, and 
\begin{equation*}
\varphi_j = \mint{-}_{B_{\mathcal{R}_j}^{(1)}(0)} \varphi\circ\tau_{(\bm0,s)} \, \d s,
\end{equation*}
where $\mathcal R_j$ is replaced by $L$ if $j=\infty$. 
The average integral is understood as an $L^2(\nu;\R)$-valued Bochner integral whenever the radius $\mathcal R_j$ is non-zero and as a point evaluation at $s=0$ otherwise. 
Case c) is treated analogously, using the specifications in Lemma~\ref{lem:ergo}~c).
\end{proof}

We remark that if the functional $\varphi$ in Proposition \ref{prop:mean_ergo_thm} is a point evaluation of the form $\varphi(\bom) = \phi(\bom(\bm{0},0))$, $\bom \in \mathcal{O}$, with $\phi \in L^2(\Pr_{\w(\bm{0},0)};\R)$, then the integrand in \eqref{eq:limit_mean_ergo} simplifies to $\phi(\w(\x_j+\y, t_j + s))$. 
The following corollary is restricted to point evaluations for ease of exposition. It extends Proposition \ref{prop:mean_ergo_thm} to the inhomogeneous case in the sense that it allows for small perturbations of the distribution parameters of homogeneous fields, represented by the additional  
arguments  
of $\w$.

\begin{corollary}[Inhomogeneous averaging]\label{cor:ergo}
Let $\w = (\w(\x,t,\xbb,\tbb))_{(\x,t),(\xbb,\tbb) \in \R^3 \times \R}$ be an $\R^{d}$-valued, centered, mean-square continuous Gaussian random field with covariance function of the form
\begin{equation}\label{eq:cov_pert_ergo}
\begin{aligned}
\E\bigl[ \w(\x,t,\xbb,\tbb) \otimes \w(\xalt,\talt,\xbbalt,\tbbalt) \bigr]
= \Re\! \int_{\R^3\times\R} \exp\Bigl\{ i \bigl( \kap \cdot (\x-\xalt) + \omega (t-\talt\;\!) \bigr) \Bigr\} \bm{F}\bigl(\kap,\omega;\xbb,\tbb,\xbbalt,\tbbalt\bigr) \,\d(\kap,\omega),
\end{aligned}
\end{equation}
$(\x,t),(\xbb,\tbb),(\xalt,\talt\;\!),(\xbbalt,\tbbalt) \in \R^3 \times \R$, where $\bm{F} \colon (\R^3 \times \R)^3 \to \C^{d \times d}$ is a measurable function satisfying for all $(\xbb,\tbb),(\xbbalt, \tbbalt) \in \R^3 \times \R$ that $\int_{\R^3 \times \R} \| \bm{F}(\kap,\omega;\xbb,\tbb,\xbbalt,\tbbalt) \| \,\d(\kap,\omega) < \infty$. 
In addition, let the function $\phi \in \bigcap_{(\xbb,\tbb) \in \R^3 \times \R} L^2(\Pr_{\w(\bm{0},0,\xbb,\tbb)};\R)$ be such that the random field $( \phi(\w(\x,t,\xbb,\tbb)))_{(\x,t),(\xbb,\tbb) \in\R^3 \times \R}$ is mean-square continuous, 
let $(\mathcal{R}_j)_{j \in \N}$, $(\mathcal{R}'_j)_{j \in \N} \subset \R_0^+$, $L,L' \in [0,\infty]$, and $(\delta_j)_{j\in\N} \subset \R^+$ satisfy that 
$\lim_{j\to\infty} \mathcal{R}_j = L$, $\lim_{j\to\infty} \mathcal{R}'_j = L'$, and 
$\lim_{j\to \infty} \delta_j \mathcal{R}_j = \lim_{j\to \infty} \delta_j \mathcal{R}'_j = 0$, 
and assume that at least one of the limit values $L, L'$ equals infinity. 
Moreover, 
let $((\x_j,t_j))_{j\in \N} \subset \R^3 \times \R$ be an arbitrary sequence and
let $\bm{\alpha}_j \colon \R^3 \times \R \to \R^3$, $\beta_j \colon \R^3 \times \R \to \R$, $j \in \N$, be measurable 
functions of uniform linear growth, i.e., $\sup_{j\in\N}\sup_{(\y,s) \in \R^3 \times \R} (\|\bm{\alpha}_j(\y,s) \| + |\beta_j(\y,s)| ) / ( 1 + \|(\y,s) \| ) < \infty$.
Then, for every $(\xbb,\tbb) \in \R^3 \times \R$ we have that
\begin{equation}\label{eq:limit_pert_ergo}
\begin{aligned}
\lim_{j\to \infty} \;\!\!
\mint{-}_{B_{\mathcal{R}_j}^{(1)}(0)} \mint{-}_{B_{\mathcal{R}'_j}^{(3)}(\bm{0})}
\!\!\phi\Bigl( \w\bigl( \x_j + \y ,\;\! t_j + s ,\;\! \xbb + \delta_j \bm{\alpha}_j(\y,s) ,\;\! \tbb + \delta_j \beta_j(\y,s) \bigr) \Bigr) \;\!\d \y \, \d s 
= \E \Bigl[ \phi\bigl(\w(\bm{0},0,\xbb,\tbb)\bigr)\Bigr]
\end{aligned}
\end{equation}
as a limit in $L^2(\Pr;\R)$. 
The average integrals are understood as $L^2(\Pr;\R)$-valued Bochner integrals if the radius of the respective domain of integration is non-zero and as point evaluations at $\y=0$ or $s=0$ otherwise.
\end{corollary}

\begin{proof}
Let us first note that the mean-square continuity of $( \phi(\w(\x,t,\xbb,\tbb)))_{(\x,t),(\xbb,\tbb) \in\R^3 \times \R}$ and the measurability and growth assumptions on $\bm{\alpha}$, $\beta$ ensure that the $L^2(\Pr;\R)$-valued Bochner integrals in \eqref{eq:limit_pert_ergo} are defined. 
Moreover, observe that the covariance condition in \eqref{eq:cov_pert_ergo} and the uniqueness theorem for measures imply that $\w$ is homogeneous in the first two variables in the sense that
\begin{align}\label{eq:hom_pert_ergo}
\forall\,(\y,s)\in\R^3\times\R:\, \tau_{(\y,s)}\w \, \stackrel{\mathrm{d}}{=} \, \w,
\end{align}
where $\w$ is interpreted as a random variable with values in the set of all functions from $(\R^3\times\R)^2$ to $\R^d$, the notation $\stackrel{\mathrm{d}}{=}$ stands for equality in distribution, and the translation operators $\tau_{(\y,s)}$ are defined by $(\tau_{(\y,s)}\bom)(\x,t,\xbb,\tbb)=\bom(\x+\y,t+s,\xbb,\tbb)$ for functions $\bom$ from $(\R^3\times\R)^2$ to $\R^d$. 
We point out the notational difference between the random field $\w$ and its possible realizations $\bom$. 
In analogy to the setting of Lemma~\ref{lem:ergo}, the set of all functions from  $(\R^3\times\R)^2$ to $\R^d$ is endowed with the $\sigma$-algebra generated by the point evaluation mappings $\bom\mapsto \bom(\x,t,\xbb,\tbb)$, $(\x,t), (\xbb,\tbb)\in\R^3\times\R$. 
In order to establish the convergence in \eqref{eq:limit_pert_ergo} we decompose the double integral in the form
\begin{equation}\label{eq:dec_pert_ergo}
\begin{aligned}
& \mint{-}\mint{-}
\phi\Bigl( \w\bigl( \x_j + \y ,\;\! t_j + s ,\;\! \xbb ,\;\!\tbb \bigr) \Bigr) \,\d \y \, \d s \\
& + \mint{-}\mint{-} 
\Bigl[\phi\Bigl( \w\bigl( \x_j + \y ,\;\! t_j + s ,\;\! \xbb + \delta_j \bm{\alpha}_j(\y,s) ,\;\! \tbb + \delta_j \beta_j(\y,s) \bigr) \Bigr) - \phi\Bigl( \w\bigl( \x_j + \y ,\;\! t_j + s ,\;\! \xbb ,\;\! \tbb \bigr) \Bigr)\Bigr]  
\,\d \y \, \d s
\end{aligned}
\end{equation}
with domains of integration $B_{\mathcal{R}_j}^{(1)}(0)$, $B_{\mathcal{R}'_j}^{(3)}(\bm{0})$ as in \eqref{eq:limit_pert_ergo}.
For fixed $(\xbb,\tbb)\in\R^3 \times \R$, an application of Proposition~\ref{prop:mean_ergo_thm} to the homogeneous random field $(\w(\x,t,\xbb,\tbb))_{(\x,t) \in \R^3 \times \R}$ and the functional $\varphi$ that assigns to each function $\bom\colon\R^3 \times \R\to\R^d$ the value $\varphi(\bom):=\phi(\bom(\bm0,0))$ ensures the $L^2(\Pr;\R)$ convergence of the first double integral in \eqref{eq:dec_pert_ergo} to $\E [ \phi(\w(\bm{0},0,\xbb,\tbb)) ]$. 
In addition, the homogeneity property \eqref{eq:hom_pert_ergo} yields that the $L^2(\Pr;\R)$ norm of the second double integral in \eqref{eq:dec_pert_ergo} can be estimated from above by
\begin{align*}
\mint{-}_{B_{\mathcal{R}_j}^{(1)}(0)} \mint{-}_{B_{\mathcal{R}'_j}^{(3)}(\bm{0})}
\Bigl\|\phi\Bigl( \w\bigl( \bm0 ,\;\! 0 ,\;\! \xbb + \delta_j \bm{\alpha}_j(\y,s) ,\;\! \tbb + \delta_j \beta_j(\y,s) \bigr) \Bigr) - \phi\Bigl( \w\bigl( \bm0, 0, \xbb, \tbb \bigr) \Bigr)\Bigr\|_{L^2(\Pr;\R)} 
\,\d \y \, \d s,
\end{align*}
which converges to zero as $j\to\infty$ due to the mean-square continuity of $\phi( \w( \bm0, 0, \xbb, \tbb ) )$ in $(\xbb, \tbb)$, the fact that $\lim_{j\to \infty} \delta_j \mathcal{R}_j = \lim_{j\to \infty} \delta_j \mathcal{R}'_j = 0$, and the linear growth property of $\bm{\alpha}_j$ and $\beta_j$.
\end{proof}


\section{Technical calculations concerning the example model}\label{sec:calc_estimates}

Example~\ref{ex:energy_spectrum_correlation_function} presents a closure of our inhomogeneous turbulence model based on the energy spectrum $E$ from Example~\ref{ex:energy_spectrum} and the time integration kernel $\eta$ from Example~\ref{ex:temporal_cor}. 
The following calculations and estimates complete the arguments in Example \ref{ex:energy_spectrum_correlation_function}.
 
First, we analyze the differentiability properties of the energy spectrum $E(\kappa;\zeta)$. 
To this end, set $V=\{(\kappa_1,\kappa_2)\in\R^+\!\times\R^+\colon \kappa_1<\kappa_2\}$ and for every $(\kappa,\kappa_1,\kappa_2)\in\R^+\!\times V$ 
let $\widetilde E(\kappa;\kappa_1,\kappa_2)$ be defined by the right hand side of \eqref{eq:def_E} in Example~\ref{ex:energy_spectrum}. 
Observe that the function $\widetilde E$ is continuously differentiable on $\R^+\!\times V$.
Further recall that the transition wave numbers $\kappa_1=\kappa_1(\zeta)$ and $\kappa_2=\kappa_2(\zeta)$ used in the definition of $E(\kappa;\zeta)$ depend implicitly on $\zeta$ via the integral identities \eqref{eq:energy_spectrum_1}. The integral relations can be reformulated as a nonlinear system for $\kappa_1$ and $\kappa_2$
\begin{align}\label{eq:nonlin_syst_zeta}
\hat{a}_1 \kappa_1^{-2/3}-\hat{b}_1 \kappa_2^{-2/3}= C_\mathrm{K}^{-1}, \quad -\hat{a}_2 \kappa_1^{4/3}+\hat{b}_2 \kappa_2^{4/3}=(2 C_\mathrm{K}\zeta)^{-1}
\end{align}
with positive parameters 
\begin{align*}
\hat{a}_1=\frac{3}{2}+\sum_{j=4}^6\frac{a_j}{j+1}, \quad \hat{a}_2=\frac{3}{4}-\sum_{j=4}^6\frac{a_j}{j+3}, &\quad \hat{b}_1=\frac{3}{2}-\sum_{j=7}^9\frac{b_j}{j-1}, \quad \hat{b}_2=\frac{3}{4}+\sum_{j=7}^9\frac{b_j}{j-3}.
\end{align*}
The condition $0<\zeta < \zeta_\text{crit} =(2 C_\mathrm{K}^3(\hat{b}_2 - \hat{a}_2)(\hat{b}_1-\hat{a}_1)^2)^{-1} \approx 3.86$  ensures that $0<\kappa_1(\zeta) < \kappa_2(\zeta) < \infty$ is fulfilled. Substituting $x=\hat{a}_1 \kappa_1^{-2/3}/ A(\zeta)$ and $y=\hat{b}_1 \kappa_2^{-2/3} /B(\zeta)$,  where $A(\zeta) = \hat{a}_1 (2 C_\mathrm{K}\hat{a}_2 \zeta)^{1/2}$ and
$B(\zeta) = \hat{b}_1 (2 C_\mathrm{K} \hat{b}_2 \zeta)^{1/2}$,  
the system \eqref{eq:nonlin_syst_zeta}  becomes
\begin{align*}
A(\zeta) x - B(\zeta) y = C_\mathrm{K}^{-1}, \quad y^{-2}-x^{-2}=1.
\end{align*}
Hence, the transition wave numbers are 	
\begin{align}\label{eq:func_rel_zeta} 
\kappa_1(\zeta) = \bigg( \frac{\hat{a}_1}{C_\mathrm{K}^{-1} + B(\zeta)y_0(\zeta)}\bigg)^{3/2}, \quad 
\kappa_2(\zeta) =  \bigg( \frac{\hat{b}_1}{B(\zeta)y_0(\zeta)} \bigg)^{3/2},	
\end{align}
where $y_0(\zeta)$ is the unique solution of $( C_\mathrm{K}^{-1} + B(\zeta)y)^2(1-y^2)  - A(\zeta)^2y^2=0$ on $(0,1)$. 
Using the implicit function theorem, it can be shown that the mapping $\zeta \mapsto y_0(\zeta)$ is continuously differentiable on $(0,\zeta_\mathrm{crit})$. 
This in combination with the identity $E(\kappa;\zeta)=\widetilde E(\kappa;\kappa_1(\zeta),\kappa_2(\zeta))$, the differentiability properties of $\widetilde E$, and the functional relation \eqref{eq:func_rel_zeta} establishes the continuous differentiability of $E$ on $\R^+\!\times(0,\zeta_\mathrm{crit})$.

It remains to verify the estimate \eqref{eq:estimate_grad_fxsbb_to_prove_in_app}. 
The definition of $\fxsbb$ in \eqref{eq:def_etabb_fxbb} with the relation $\fx(\kappa;\zeta)= (4 \pi \kappa^2)^{-1}E(\kappa;\zeta)$
implies 
\begin{align*}
\fxsbb(\kap;\y,s) = \frac{\sx^{1/2}(\y,s)}{\sqrt{4\pi} \nkap} E^{1/2} \bigl( \sx(\y,s) \nkap; \sz(\y,s) z \bigr)
\end{align*}
for all $\kap\in\R^3\setminus\{\bm0\}$, $(\y,s)\in\R^3\times\R$.
Observe that the differentiability properties of $E$ investigated above and the regularity assumptions on $\sx$, $\sz$ ensure that the gradient $\nabla_\y\fxsbb(\kap;\y,s)$ exists whenever $(\y,s) \in \R^3 \times \R$.
In this case, extensive but elementary estimates lead to
\begin{align}\nonumber
\bigl\| \nabla_\y \fxsbb(\kap;\y,s) \bigr\|^2 & \leq \frac{3}{16 \pi \nkap^2} \frac{\|\nabla_\y \sx(\y,s)\|^2}{\sx(\y,s)} E \bigl( \sx(\y,s) \nkap; \sz(\y,s) z \bigr) 
\\ \nonumber & \qquad + \frac{3}{16 \pi} \frac{\sx(\y,s) \|\nabla_\y \sx(\y,s)\|^2}{E \bigl( \sx(\y,s) \nkap; \sz(\y,s) z \bigr) } \bigl|(\partial_1 E) \bigl( \sx(\y,s) \nkap; \sz(\y,s) z \bigr)\bigr|^2
\\ \nonumber & \qquad + \frac{3z^2}{16 \pi \nkap^2} \frac{\sx(\y,s) \|\nabla_\y \sz(\y,s)\|^2}{E \bigl( \sx(\y,s) \nkap; \sz(\y,s) z \bigr) } \bigl|(\partial_2 E) \bigl( \sx(\y,s) \nkap; \sz(\y,s) z \bigr)\bigr|^2
\\ \label{eq:estimate_grad_fxsbb} & \leq\frac{1}{4\pi \nkap^2} 
\begin{cases}
c_1(\y,s) \nkap^4, & \nkap < \frac{\kappa_1(\y,s) }{\sx(\y,s)},
\\ c_2(\y,s) \nkap^{-5/3}, & \frac{\kappa_1(\y,s) }{\sx(\y,s)} \leq \nkap \leq \frac{\kappa_2(\y,s) }{\sx(\y,s)},
\\  c_3(\y,s) \nkap^{-7}, & \nkap > \frac{\kappa_2(\y,s) }{\sx(\y,s)},
\end{cases}
\end{align}
with $\kappa_1(\y,s) = \kappa_1(\sz(\y,s)z)$ and $\kappa_2(\y,s) = \kappa_2(\sz(\y,s)z)$. Here,
\begin{align*}
c_1(\y,s) & = \frac{3}{4} C_\mathrm{K} \frac{\sx^3(\y,s)}{\kappa_1^{17/3}(\y,s) } \biggl( \!\! \Bigl( a_4 + 16 a_4^2 \Bigr)  \bigl\| \nabla_\y \sx(\y,s) \bigr\|^2 \!\!+\! \frac{289}{9} a_4^2 z^2 \bigl| \kappa_1'(\y,s) \bigr|^2 \frac{\sx^2(\y,s)}{\kappa_1^2(\y,s)} \bigl\| \nabla_\y \sz(\y,s) \bigr\|^2\biggr),
\\ c_2(\y,s) & = \frac{3}{4} C_\mathrm{K} \frac{34}{9} \sx^{-8/3}(\y,s) \bigl\| \nabla_\y \sx(\y,s) \bigr\|^2,
\\ c_3(\y,s) & = \frac{3}{4} C_\mathrm{K} \frac{\kappa_2^{16/3}(\y,s)}{\sx^{8}(\y,s)} \biggl( \!\! \Bigl( b_7 + 49 b_7^2 \Bigr)  \bigl\| \nabla_\y \sx(\y,s) \bigr\|^2 \!\! + \! \frac{256}{9} b_7^2 z^2 \bigl| \kappa_2'(\y,s) \bigr|^2 \frac{\sx^2(\y,s)}{\kappa_2^2(\y,s)} \bigl\| \nabla_\y \sz(\y,s) \bigr\|^2\biggr) 
\end{align*}
holds with $\kappa_1'(\y,s) = (\frac{\d}{\d \zeta} \kappa_1)(\sz(\y,s)z)$ and $\kappa_2'(\y,s) = (\frac{\d}{\d \zeta} \kappa_2)(\sz(\y,s)z)$. 
Let $(\x,t) \in \R^3 \times \R$ be fixed. Then, by continuity and positivity of $\kappa_1$, $\kappa_2$, $\sz$, and $\sx$ we may choose a radius $r \in \R^+$ such that
\begin{align*}
0 < \inf_{(\y,s) \in B_r(\x,t)} \frac{\kappa_1(\y,s)}{\sx(\y,s)} \leq \sup_{(\y,s) \in B_r(\x,t)} \frac{\kappa_2(\y,s)}{\sx(\y,s)} < \infty.
\end{align*}
For the sake of presentation we introduce the abbreviations 
\begin{align*}
\kappa_1^- = \kappa_1^-(\x,t,r) = \inf_{(\y,s) \in B_r(\x,t)} \frac{\kappa_1(\y,s)}{\sx(\y,s)}, \qquad \kappa_2^+ = \kappa_2^+(\x,t,r) = \sup_{(\y,s) \in B_r(\x,t)} \frac{\kappa_2(\y,s)}{\sx(\y,s)}.
\end{align*}
Observe that the continuity of $\kappa_1$, $\frac{\d}{\d \zeta} \kappa_1$, $\kappa_2$, $\frac{\d}{\d \zeta} \kappa_2$, $\sx$, $\nabla_\x \sx$, $\sz$, and $\nabla_\x \sz$ and the positivity of $\kappa_1$, $\kappa_2$, and $\sx$ imply that the functions $c_1$, $c_2$, and $c_3$ are all bounded on the compact set $B_r(\x,t)$, leading to the finite constants
\begin{align*}
C_1(\x,t,r) = \sup_{(\y,s) \in B_r(\x,t)} c_1(\y,s), \quad  C_3(\x,t,r) = \sup_{(\y,s) \in B_r(\x,t)} c_3(\y,s).
\end{align*}
Further note that the constants $\kappa_1^-$ and $\kappa_2^+$ divide $\R^3$ into the three regions
\begin{equation*}
\{ \nkap < \kappa_1^-\}, \quad \{ \kappa_1^- \leq \nkap \leq \kappa_2^+\}, \quad \{ \nkap > \kappa_2^+\}.
\end{equation*}
For $\kappa_1^- \leq \nkap \leq \kappa_2^+$ and $(\y,s) \in B_r(\x,t)$
it can be shown that $\| \nabla_{\y} \fxsbb(\kap;\y,s) \|^2$ is bounded from above by the constant
\begin{align*}
C_2(\x,t,r) = \max\Bigl\{ C_1(\x,t,r), \sup_{(\y,s) \in B_r(\x,t)} c_2(\y,s),\; C_3(\x,t,r)\Bigr\} \max\bigl\{(\kappa_1^-)^{-7},\, (\kappa_2^+)^4\bigr\}.
\end{align*}
Finally, this and \eqref{eq:estimate_grad_fxsbb} give rise to the estimate
\begin{align}\label{eq:estimate_property_fxsbb}
\sup_{(\y,s) \in B_r(\x,t)} \bigl\| \nabla_{\y} \fxsbb(\kap;\y,s) \bigr\|^2 & \leq\frac{1}{4\pi \nkap^2} 
\begin{cases}
C_1(\x,t,r) \nkap^4, & \nkap < \kappa_1^-,
\\ C_2(\x,t,r), & \kappa_1^- \leq \nkap \leq \kappa_2^+,
\\  C_3(\x,t,r) \nkap^{-7}, & \nkap > \kappa_2^+,
\end{cases}
\end{align}
which proves \eqref{eq:estimate_grad_fxsbb_to_prove_in_app}.


\printbibliography

\end{document}